\documentclass[a4paper, leqno, twoside,12pt]{amsart}
\usepackage[utf8]{inputenc}
\usepackage{esint}
\usepackage{amssymb,amsmath,mathtools,latexsym}
\usepackage{amsthm}%% tracking editors
\usepackage[usenames,dvipsnames]{color}
\usepackage{marginnote}
\usepackage[alphabetic]{amsrefs}
\usepackage{setspace}
\usepackage{graphicx}
\usepackage{mathrsfs}
\usepackage{bbm}
\usepackage{mdwlist}
\usepackage{url}
\usepackage[hidelinks]{hyperref}
\usepackage{breakurl}
\usepackage[final]{showlabels}% option: left
\usepackage[all,cmtip]{xy}
\usepackage{ifthen}%% for conditionals
\renewcommand{\mod}{\,\mathrm{mod}\,}
\newcommand{\one}{\mathbbm{1}}
\newcommand{\ii}{\mathbf{i}}

\newcommand{\weakstar}{weak$^{\ast}$~}
\newcommand{\set}[2]{\left\{ {#1} \,;\, {#2} \right\}}
\newcommand{\der}[1][]{% differential
   \ifthenelse{ \equal{#1}{} }%
      {\ensuremath{\mathrm{d}}}%
      {\ensuremath{\mathrm{d}_{#1}\!}}%
}
\newcommand{\quotient}[2]{\mathchoice
            {% \displaystyle
              \raisebox{.6ex}{\small \newline ${#1}$}\!\big/\!\raisebox{-.6ex}{\small ${#2}$}
            }
            {% \textstyle
              {#1}/{#2}%\raisebox{.4ex}{\small \newline ${#1}$}\!/\!\raisebox{-.4ex}{\small ${#2}$}
            }
            {% \scriptstyle
              {#1}/{#2}%\raisebox{.4ex}{\tiny \newline ${#1}$}\!/\!\raisebox{-.4ex}{\tiny ${#2}$}
            }
            {% \scriptscriptstyle
              {#1}/{#2}%\raisebox{.3ex}{\tiny \newline ${#1}$}\!/\!\raisebox{-.3ex}{\tiny ${#2}$}
            } }
\newcommand{\rquotient}[2]{\mathchoice
            {% \displaystyle
              \raisebox{-.6ex}{\small \newline ${#1}$}\!\big\backslash\!\raisebox{.6ex}{\small ${#2}$}
            }
            {% \textstyle
              {#1}\backslash{#2}%\raisebox{-.4ex}{\small \newline ${#1}$}\!\backslash\!\raisebox{.4ex}{\small ${#2}$}
            }
            {% \scriptstyle
              \raisebox{-.4ex}{\tiny \newline ${#1}$}\!\backslash\!\raisebox{.4ex}{\tiny ${#2}$}
            }
            {% \scriptscriptstyle
              \raisebox{-.3ex}{\tiny \newline ${#1}$}\!\backslash\!\raisebox{.3ex}{\tiny ${#2}$}
            } }
\newcommand{\biquotient}[3]{\mathchoice
            {% \displaystyle
              \raisebox{-.6ex}{\small \newline ${#1}$}\!\big\backslash\!\raisebox{.6ex}{\small ${#2}$}\!\big/\!\raisebox{-.6ex}{\small ${#3}$}
            }
            {% \textstyle
              {#1}\backslash{#2}/{#3}%\raisebox{-.4ex}{\small \newline ${#1}$}\!\backslash\!\raisebox{.4ex}{\small ${#2}$}\!/\!\raisebox{-.4ex}{\small ${#3}$}
            }
            {% \scriptstyle
              \raisebox{-.4ex}{\tiny \newline ${#1}$}\!\backslash\!\raisebox{.4ex}{\tiny ${#2}$}\!/\!\raisebox{-.4ex}{\tiny ${#3}$}
            }
            {% \scriptscriptstyle
              \raisebox{-.3ex}{\tiny \newline ${#1}$}\!\backslash\!\raisebox{.3ex}{\tiny ${#2}$}\!/\!\raisebox{-.3ex}{\tiny ${#3}$}
            } }

\newcommand{\cl}[1]{\overline{ {#1} }}
\newcommand{\vol}[1]{\mathrm{Vol}( {#1} )}

\newcommand{\Ltwo}{L^{2}}

\newcommand{\norm}[1]{\lVert {#1} \rVert}
\newcommand{\normp}[1]{\lVert {#1} \rVert_{p}}
\newcommand{\normS}[1]{\lVert {#1} \rVert_{S}}
\DeclareFontFamily{U}{matha}{\hyphenchar\font45}
\DeclareFontShape{U}{matha}{m}{n}{
      <5> <6> <7> <8> <9> <10> gen * matha
      <10.95> matha10 <12> <14.4> <17.28> <20.74> <24.88> matha12
      }{}
\DeclareSymbolFont{matha}{U}{matha}{m}{n}
\DeclareFontSubstitution{U}{matha}{m}{n}
\DeclareFontFamily{U}{mathx}{\hyphenchar\font45}
\DeclareFontShape{U}{mathx}{m}{n}{
      <5> <6> <7> <8> <9> <10>
      <10.95> <12> <14.4> <17.28> <20.74> <24.88>
      mathx10
      }{}
\DeclareSymbolFont{mathx}{U}{mathx}{m}{n}
\DeclareFontSubstitution{U}{mathx}{m}{n}
\DeclareMathDelimiter{\vvvert}{0}{matha}{"7E}{mathx}{"17}
\newcommand{\exponentdecayharishchandraunipotent}{\varepsilon_{0}}
\newcommand{\exponentunipotentgrowth}{c_{0}}
\newcommand{\maxexponentunipotentgrowth}{c_{1}}

\newcommand{\exponentdecayhorocycles}{\kappa_{0}}
\newcommand{\exponentdecaytimes}{\varrho_{0}}%% Spectral gap of the times y map on X_S
\newcommand{\exponentdecaytimesprime}{\varrho_{1}}%% Spectral gap of the times y map on X_S
\newcommand{\exponentdecaytimesjoint}{\varrho_{\jmath}}%% Spectral gap of the times y map on T_S x X_S
\newcommand{\exponentpolynomialsobolevunipotent}{\eta}%% S(u_t f)\ll (1+t)^{\eta}S(f)
\newcommand{\exponentdecayrationalpointscongruence}{\kappa_{1}}%% decay rate due to rational points on congruence quotients
\newcommand{\exponentdecayrationalpointscongruencejoining}{\kappa_{2}}%% decay rate due to rational points on congruence quotients of the reals and of G
\newcommand{\exponentdecayrationalpointsS}{\kappa_{2}}%% decay rate due to rational points s-arithmetic joining
\newcommand{\exponentdecayprimitiverationalpointsSfactor}{\kappa_{5}}%% decay rate due to primitive rational points s-arithmetic semisimple
\newcommand{\rationalsn}{\Pcal(n)}
\newcommand{\rationalsnr}{\Pcal_{\infty}(n)}
\newcommand{\rationalsnrpiece}{\Pcal_{\infty}(n;\alpha,\beta)}
\newcommand{\rationalsnrXpiece}{\Pcal_{\infty,X}(n;\alpha,\beta)}
\newcommand{\primitiverationalsnproduct}{\Pcal(n)^{\times}}%% Primitive rational points in the S-arithmetic extension
\newcommand{\primitiverationalsnabcrtwotorus}{\Qcal_{\infty}^{\times d}(n;a,b,c)}%% Primitive rational points with enemies
\newcommand{\primitiverationalsnanbncnrtwotorus}{\Qcal_{\infty}^{\times d}(n;a_{n},b_{n},c_{n})}%% Primitive rational points for sequence of enemies
\newcommand{\primitiverationalsnabctwotorus}{\Qcal^{\times d}(n;a,b,c)}%% Primitive rational points in the S-arithmetic extension with enemies
\newcommand{\primitiverationalsnanbncntwotorus}{\Qcal^{\times d}(n;a_{n},b_{n},c_{n})}%% Primitive rational points in the S-arithmetic extension for sequence of enemies
\newcommand{\measureprimitiverationalsnabctwotorusr}{\overline{\mu}_{n;a,b,c}^{\times d}}%% Counting measure on primitive rational points in the real product of the two-torus and the modular surface with enemies
\newcommand{\measureprimitiverationalsnanbncntwotorusr}{\overline{\mu}_{n;a_{n},b_{n},c_{n}}^{\times d}}%% Counting measure on primitive rational points in the real product of the two-torus and the modular surface for sequences of enemies
\newcommand{\measureprimitiverationalsnabctwotorus}{\overline{\nu}_{n;a,b,c}^{\times d}}%% Counting measure on primitive rational points in the S-arithmetic product of the two-torus and the modular surface with enemies
\newcommand{\measureprimitiverationalsnanbncntwotorus}{\overline{\nu}_{n;a_{n},b_{n},c_{n}}^{\times d}}%% Counting measure on primitive rational points in the S-arithmetic product of the two-torus and the modular surface for sequences of enemies
\newcommand{\orbitnNpiece}[2]{%
  \ifthenelse{ \equal{#2}{} }%
  {%
    \ensuremath{\mathcal{O}_{n; {#1} }(N)}%
  }%
  {%
    \ensuremath{\mathcal{O}_{n; {#1} , {#2}}(N)}%
  }%
}

\newcommand{\continuous}{C}
\newcommand{\compact}{C_{c}}
\newcommand{\smooth}{C^{\infty}}
\newcommand{\compactsmooth}{C_{c}^{\infty}}
\newcommand{\tensorcompactsmooth}{\Acal_{c}^{\infty}}

\newcommand{\Qp}{\mathbb{Q}_{p}}
\newcommand{\Zp}{\mathbb{Z}_{p}}

\newcommand{\Qinfty}{\mathbb{Q}_{\infty}}
\newcommand{\QS}{\mathbb{Q}_{S}}
\newcommand{\ZS}{\mathbb{Z}[S^{-1}]}
\newcommand{\Zs}{\mathbb{Z}_{S}}
\newcommand{\Sf}{S_{\mathrm{f}}}%% finite places
\newcommand{\QSf}{\mathbb{Q}_{\Sf}}

\newcommand{\Zsf}{\mathbb{Z}_{\Sf}}

\newcommand{\SLtwo}{\mathrm{SL}_{2}}

\newcommand{\GL}{\mathrm{GL}}

\newcommand{\SL}{\mathrm{SL}}

\newcommand{\Stab}{\mathrm{Stab}}
\newcommand{\Ad}{\mathrm{Ad}}
\newcommand{\Gp}{G_{p}}%% p-adic points
\newcommand{\Gr}{G_{\infty}}%% real points
\newcommand{\GS}{G_{S}}%% Q_{S}-points
\newcommand{\Gs}{G^{S}}%% Z_{S}-points (good compact)
\newcommand{\GGamma}{G_{\Gamma}}%% intersection with the lattice
\newcommand{\GSK}{K[0]}%% Z_{Sf}-points
\newcommand{\Hp}{H_{p}}%% p-adic points
\newcommand{\Hq}{H_{q}}%% p-adic points
\newcommand{\Hr}{H_{\infty}}%% real points
\newcommand{\HS}{H_{S}}%% Q_{S}-points
\newcommand{\Hs}{H^{S}}%% Z_{S}-points (good compact)
\newcommand{\HGamma}{H_{\Gamma}}%% intersection with the lattice
\newcommand{\HSK}{H_{\mathrm{f}}}%% Z_{Sf}-points
\newcommand{\tildeHp}{\tilde{H}_{p}}%% p-adic points
\newcommand{\tildeHr}{\tilde{H}_{\infty}}%% real points
\newcommand{\tildeHS}{\tilde{H}_{S}}%% Q_{S}-points
\newcommand{\Up}{U_{p}}%% p-adic points
\newcommand{\Ur}{U_{\infty}}%% real points
\newcommand{\US}{U_{S}}%% Q_{S}-points
\newcommand{\UGamma}{U_{\Gamma}}%% intersection with the lattice
\newcommand{\Vp}{V_{p}}%% p-adic points
\newcommand{\Vr}{V_{\infty}}%% real points
\newcommand{\VS}{V_{S}}%% Q_{S}-points
\newcommand{\VGamma}{V_{\Gamma}}%% intersection with the lattice
\newcommand{\AS}{A_{S}}%% Q_{S}-points
\newcommand{\GammaS}{\Gamma_{S}}%% G(Z[S^{-1}])
\newcommand{\Gammar}{\Gamma_{\infty}}%% G(Z)
\newcommand{\GammaSK}{\mathrm{SL}_{2}(\mathbb{Z})}%% G(Z) in G(Z_{S})
\newcommand{\Lie}{\mathrm{Lie}}

\newcommand{\lieg}{\mathfrak{g}}
\newcommand{\liegR}{\mathfrak{g}_{\mathbb{R}}}
\newcommand{\liegZ}{\mathfrak{g}_{\mathbb{Z}}}

\newcommand{\liegQS}{\mathfrak{g}_{\QS}}

\newcommand{\liegZS}{\mathfrak{g}_{\ZS}}

\newcommand{\liesltwo}{\mathfrak{sl}_{2}}

\newcommand{\TS}{\mathbb{T}_{S}}%% S-arithmetic torus
\newcommand{\Tr}{\mathbb{T}_{\infty}}%% real torus
\newcommand{\XS}{X_{S}}%% Gamma_{S}\G_{S}
\newcommand{\Xr}{X_{\infty}}%% Gamma_{\infty}\G_{\infty}
\newcommand{\YS}{Y_{S}}%% Gamma_{S}^{K}\G^{S}

\newcommand{\Av}{\mathrm{Av}}
\newcommand{\pr}{\mathrm{pr}}
\newcommand{\height}{\mathrm{ht}}

\newcommand{\Xfrak}{\mathfrak{X}}

\newcommand{\Acal}{\mathcal{A}}
\newcommand{\Bcal}{\mathcal{B}}
\newcommand{\Ccal}{\mathcal{C}}
\newcommand{\Dcal}{\mathcal{D}}
\newcommand{\Ecal}{\mathcal{E}}
\newcommand{\Ical}{\mathcal{I}}
\newcommand{\Pcal}{\mathcal{P}}
\newcommand{\Qcal}{\mathcal{Q}}
\newcommand{\Scal}{\mathcal{S}}
\newcommand{\Ucal}{\mathcal{U}}

\newcommand{\C}{\mathbb{C}}
\newcommand{\D}{\mathbb{D}}
\newcommand{\N}{\mathbb{N}}
\renewcommand{\P}{\mathbb{P}}
\newcommand{\Q}{\mathbb{Q}}
\newcommand{\R}{\mathbb{R}}
\renewcommand{\S}{\mathbb{S}}
\newcommand{\Z}{\mathbb{Z}}
\newcommand{\T}{\mathbb{T}}

\newtheorem{theorem}{Theorem}\numberwithin{theorem}{section}
\newtheorem{thm}[theorem]{Theorem}
\newtheorem*{thmwo}{Theorem}
\newtheorem{lem}[theorem]{Lemma}
\newtheorem{lemma}[theorem]{Lemma}
\newtheorem{proposition}[theorem]{Proposition}
\newtheorem{prop}[theorem]{Proposition}
\newtheorem{cor}[theorem]{Corollary}
\theoremstyle{definition}\newtheorem{definition}[theorem]{Definition}
\theoremstyle{remark}\newtheorem{remark}[theorem]{Remark}
\begin{document}
\title[Primitive rational points]{Primitive rational points on expanding horocycles in products of the modular surface with the torus}
\author{Manfred Einsiedler, Manuel Luethi, Nimish Shah}
\email{manuel.luethi@math.ethz.ch}
\subjclass[2000]{37A45, 11L05, 11J71 (Primary)}
\thanks{M.~Einsiedler was supported by SNSF grant 200021-178958, M.~Luethi was supported by SNSF grants 200021-178958 and P1EZP2\_181725, and N.A.~Shah was supported by NSF grant DMS-1700394.}
\date{\today}
\begin{abstract}
  We prove effective equidistribution of primitive rational points and of primitive rational points defined by monomials along long horocycle orbits in products of the torus and the modular surface. This answers a question posed in joint work by the first and the last named author with Shahar Mozes and Uri Shapira. Under certain congruence conditions we prove the joint equidistribution of conjugate rational points in the two-torus and the modular surface.
\end{abstract}
\maketitle
\allowdisplaybreaks
\setcounter{tocdepth}{1}
\tableofcontents
%%%%%%%%%%%% Introduction %%%%%%%%%%%%
\section{Introduction}
Let~$n$ be a natural number and~$k\in\Z$ coprime to~$n$, denoted~$(k,n)=1$. Denote by~$\cl{k}\in\Z$ any choice of a modular inverse of~$k\,\mod n$. The examination of modular forms naturally leads to the question of statistical independence of~$k$ and~$\cl{k}$ in~$\quotient{\Z}{n\Z}$, see for example \cite{Selberg65}. Naturally, such a question would be asked in terms of asymptotics for large~$n$. To this end, it is useful to recast the formulation on the torus~$\T=\rquotient{\Z}{\R}$. Given an integer~$k\in\Z$ coprime to~$n$, the tuple~$(\frac{k}{n},\frac{\cl{k}}{n})\in\T^{2}$ is independent of the choice of the representatives~$k$ and~$\cl{k}$. The group~$\T^{2}$ carries a natural probability Haar measure~$m$ coming from the uniform measure on the real plane and a natural way to state statistical independence of the tuples~$(\frac{k}{n},\frac{\cl{k}}{n})$,~$(k,n)=1$ is to say that the average of a continuous function~$f$ on~$\continuous(\T^{2})$ over these tuples converges to the integral of~$f$ with respect to that natural measure, i.e.~
\begin{equation}\label{eq:kloosterman}
  \frac{1}{\phi(n)}\sum_{(k,n)=1}f(\tfrac{k}{n},\tfrac{\cl{k}}{n})\to\int_{\T^{2}}f(x,y)\der m(x,y).
\end{equation}
By Fourier expansion this becomes a problem of estimating certain exponential sums and in fact the above convergence has been proven by Kloosterman \cite{Kloosterman} with a rate. The rate has been optimized in seminal work by Weil \cite{WeilExponentialSums}.

More recently, Jens Marklof has interpreted the above set in terms of intersection points of certain horospheres \cite{Primitive}. To motivate the formulation of our problem, we will repeat this observation here. For this introduction, consider the lattice~$\Gamma=\SLtwo(\Z)$ inside~$G=\SLtwo(\R)$ and denote the subgroups
\begin{align*}
  U&=\set{u_{t}=\begin{pmatrix}
      1 & t \\
      0 & 1
    \end{pmatrix}}{t\in\R}\\
  V&=\set{v_{s}=\begin{pmatrix}
      1 & 0 \\
      s & 1
    \end{pmatrix}}{s\in\R}\\
  A&=\set{a_{y}=\begin{pmatrix}
      y & 0 \\
      0 & \tfrac{1}{y}
    \end{pmatrix}}{y\in(0,\infty)}
\end{align*}
It is well known that the~$U$ and~$V$-orbits of~$\Gamma$ are closed and that~$\Gamma Ua_{y}$ and~$\Gamma Va_{y}$ equidistribute in~$\rquotient{\Gamma}{G}$ as~$y\to0$ and~$y\to\infty$ respectively (cf.~\cite{Sarnak1981}). For small~$y$ one could wonder, whether the long orbit~$\Gamma Ua_{y}$ intersects the orbit~$\Gamma V$. An elementary calculation shows that intersections occur if and only if there is some~$n\in\N$ so that~$y=\frac{1}{n}$. In this case~$\Gamma u_{t}a_{y}=\Gamma v_{s}$ implies that~$t=\frac{k}{n}$ and~$s=\frac{l}{n}$ for some~$k,l\in\Z$. Finally,~$1=\det(u_{t}a_{y}v_{-s})$ yields~$kl\equiv 1\,\mod n$. In particular,~$l=\cl{k}$ in~$\quotient{\Z}{n\Z}$. As~$\Gamma U\cong\T$ and~$\Gamma V\cong\T$, the measure appearing in~\eqref{eq:kloosterman} can be identified with the normalized counting measure on the set
\begin{equation*}
  \set{(\Gamma u_{k/n},\Gamma u_{k/n}a_{n}^{-1})}{(k,n)=1}\subseteq \Gamma U\times\Gamma V.
\end{equation*}
Given~$n$ and~$\alpha\in\R$, denote
\begin{equation*}
  \Pcal(n)_{\alpha}=\set{(\Gamma u_{k/n},\Gamma u_{k/n}a_{n^{\alpha}}^{-1})}{(k,n)=1}\subseteq\Gamma U\times\rquotient{\Gamma}{G}.
\end{equation*}
We have just argued that the set~$\Pcal(n)_{1}$ equidistributes inside~$\Gamma U\times\Gamma V$ as~$n\to\infty$. The problem of equidistribution of the primitive rational points~$\Pcal(n)_{\alpha}$ inside~$\Gamma U\times\rquotient{\Gamma}{G}$ has applications to Gauss sums and have been examined in \cite{Emek}, \cite{EmekMarklof}. Our work provides a considerable strengthening of some results in the first mentioned article. For the sake of simplicity of exposition we are going to focus only on the case~$\alpha=\frac{1}{2}$. Moreover, as our method of proof allows it, we are going to discuss a more general version of the problem where instead of the primitive rational points we look at multiplies of monomial residues. More precisely, given~$a,b,d\in\N$, we let
\begin{equation*}
  \Pcal^{\times d}(n;a,b)=\big\{(\tfrac{ak^{d}}{n},\Gamma u_{bk^{d}/n}a_{\sqrt{n}}^{-1}):(k,n)=1\big\}.
\end{equation*}
We can now state our first main result, which implies equidistribution of~$\Pcal(n)_{1/2}$ as~$n\to\infty$.
\begin{thm}\label{thm:mainthm}
  Fix~$d\in\N$. There exists an~$\Ltwo$-Sobolev norm~$\Scal$ on~$\compactsmooth(\T\times\rquotient{\Gamma}{G})$ and positive constants~$\kappa,\eta,C$ such that for all~$n\in\N$ and~$a,b\in\Z$ satisfying~$(n,ab)=1$ and for all~$F\in\compactsmooth(\T\times\rquotient{\Gamma}{G})$ we have
  \begin{equation*}
    \bigg\lvert\frac{1}{\lvert\Pcal^{\times d}(n;a,b)\rvert}\sum_{(t,x)\in\Pcal^{\times d}(n;a,b)}F(t,x)-\int_{\T\times\rquotient{\Gamma}{G}}F\bigg\rvert\leq C(ab)^{\eta}n^{-\kappa}\Scal(F).
  \end{equation*}
\end{thm}

Denote by~$\Pcal_{X}^{\times d}(n;b)\subseteq\rquotient{\Gamma}{G}$ the projection of~$\Pcal^{\times d}(n;a,b)$ to~$\rquotient{\Gamma}{G}$, which does of course not depend on~$a$. The method of proof applied in the proof of Theorem \ref{thm:mainthm} yields the following
\begin{cor}\label{cor:cormainthm}
  Fix~$d\in\N$. There exist an~$\Ltwo$-Sobolev norm~$\Scal$ on~$\compactsmooth(\rquotient{\Gamma}{G})$ and positive constants~$\kappa^{\prime},C_{1}$ such that for all~$n\in\N$ and~$b_{n}\in\Z$ satisfying~$(n,b_{n})=1$ and for all~$f\in\compactsmooth(\rquotient{\Gamma}{G})$ we have
  \begin{equation*}
    \bigg\lvert\frac{1}{\lvert\Pcal_{X}^{\times d}(n;b_{n})\rvert}\sum_{x\in\Pcal_{X}^{\times d}(n;b_{n})}f(x)-\int_{\rquotient{\Gamma}{G}}f\bigg\rvert\leq C_{1}n^{-\kappa^{\prime}}\Scal(F).
  \end{equation*}
\end{cor}

A natural generalization of the problems described above is to ask for the joint distribution of primitive rational points for~$\alpha=0$,~$\alpha=\frac{1}{2}$ and~$\alpha=1$ simultaneously, that is the distribution of the sets
\begin{equation*}
  \set{(\Gamma u_{k/n},\Gamma u_{k/n}a_{n^{-1/2}},\Gamma u_{k/n}a_{n^{-1}})}{(k,n)=1}\subseteq \Gamma U\times\rquotient{\Gamma}{G}\times\Gamma V.
\end{equation*}
Rearranging factors, equidistribution of these sets can be interpreted as orthogonality of Kloosterman sums to averages along primitive rational points on expanding horocycles.
Using Theorem~\ref{thm:mainthm} and entropy arguments, we show that for~$a,b,c,d,n\in\N$ the sets
\begin{equation*}
  \Qcal^{\times d}(n;a,b,c)=\set{(\tfrac{ak^{d}}{n},\tfrac{b\cl{k^{d}}}{n},\Gamma u_{ck^{d}/n}a_{\sqrt{n}}^{-1})}{(k,n)=1}
\end{equation*}
equidistribute as~$n\to\infty$ along some congruence condition. More precisely, we prove the following
\begin{thm}\label{thm:realtwotorus}
  Let~$p,q$ be two distinct primes and let~$\D(pq)=\{n\in\N:(n,pq)=1\}$. Let~$a_{n},b_{n},c_{n}\in\Z$ be a sequence of integers coprime to~$n$. Let~$F\in\compact(\T\times\T\times\rquotient{\Gamma}{G})$. Then
  \begin{equation*}
    \frac{1}{\lvert\Qcal^{\times d}(n;a_{n},b_{n},c_{n})\rvert}\sum_{x\in\Qcal^{\times d}(n;a_{n},b_{n},c_{n})}F(x)\to\int_{\T\times\T\times\rquotient{\Gamma}{G}}F
  \end{equation*}
  as~$n\to\infty$ with~$n\in\D(pq)$.
\end{thm}

The equidistribution of the sets~$\Pcal(n)_{1}$ has other natural generalizations, for example to~$\SL_{N}(\R)$ which was examined in \cite{Primitive}. The ineffective equidistribution proven there has been effectivized more recently, first in the case~$N=3$ by Lee and Marklof in \cite{LeeMarklof} and later for general~$N$ by El-Baz, Huang and Lee in \cite{BazHuangLee}. These generalizations all concern variations of the problem for the fixed scaling parameter~$\alpha=1$. The generalization in the present article concerns variation of the scaling parameter~$\alpha$ and we want to quickly explain why we only discuss the case~$\alpha=\frac{1}{2}$. Assume first that~$\alpha>1$, then~$\Pcal(n)_{\alpha}\subseteq\Gamma U\times\Gamma Va_{n^{1-\alpha}}$ and the orbit in the second component diverges into the cusp uniformly. Hence the limit measure of the normalized counting measures on these sets is trivial. The case~$\alpha<0$ shows similar behaviour. Hence, one can restrict to the case~$\alpha\in(0,1)$. A detailed treatment for~$\alpha\in(0,\frac{1}{2}]$ can be found in \cite{numberfield} and the case~$\alpha\in[\frac{1}{2},1)$ can be reduced to the former using the above relationship between~$\Gamma U$ and~$\Gamma V$. It becomes clear from the arguments in \cite{numberfield} that~$\alpha=\frac{1}{2}$ is the most difficult case due to the fact that the points~$\Gamma u_{k/n}a_{\sqrt{n}}^{-1}$,~$0\leq k<n$, are separated by distance one along the~$U$-orbit and along the~$V$-orbit. We also refer to \cite{survey} where a weaker version of Theorem~\ref{thm:mainthm} was announced.
\subsection*{A sketch of the proof of Theorem \ref{thm:mainthm}}
For the sake of illustration, we sketch an argument to prove equidistribution as in Theorem \ref{thm:mainthm} for the second component, assuming for simplicity that~$d,b=1$. To this end we assume equidistribution of the rational points, i.e.~assume that for all compactly supported continuous functions~$f$ on~$\rquotient{\Gamma}{G}$ we have
\begin{equation*}
  \frac{1}{n}\sum_{k=0}^{n}f(\Gamma u_{k/n}a_{\sqrt{n}}^{-1})\overset{n\to\infty}{\longrightarrow}\int_{\rquotient{\Gamma}{G}}f.
\end{equation*}
Fix a prime~$p$ and a small value~$\varepsilon>0$. Let
\begin{equation*}
  N(p,\varepsilon)=\set{n\in\N}{(p,n)=1\text{ and }\tfrac{\phi(n)}{n}>\varepsilon},
\end{equation*}
where~$\phi$ denotes Euler's totient function counting the number of units in~$\quotient{\Z}{n\Z}$. We denote
\begin{equation*}
  \Pcal(n)=\set{\Gamma u_{k/n}a_{\sqrt{n}}^{-1}}{0\leq k<n},\qquad\Pcal(n)^{\times}=\set{\Gamma u_{k/n}a_{\sqrt{n}}^{-1}}{(k,n)=1},
\end{equation*}
as well as~$\Pcal(n)^{0}=\Pcal(n)\setminus\Pcal(n)^{\times}$. We denote by~$\mu_{n},\mu_{n}^{\times}$ and~$\mu_{n}^{0}$ the corresponding normalized counting measures. Then
\begin{equation*}
  \mu_{n}=\tfrac{\phi(n)}{n}\mu_{n}^{\times}+\tfrac{n-\phi(n)}{n}\mu_{n}^{0}.
\end{equation*}
Note that for all~$n$ satisfying~$(p,n)=1$, these measures are invariant under the map given by~$\Gamma u_{k/n}a_{\sqrt{n}}^{-1}\mapsto\Gamma u_{p^{2}k/n}a_{\sqrt{n}}^{-1}$. Assume (falsely) that there was a lattice element~$\gamma\in\Gamma$ of infinite order inducing this map via right multiplication on~$\rquotient{\Gamma}{G}$ and assume furthermore that~$\gamma$ commutes with~$A$. As~$\mu_{n}$ converges to the invariant probability measure as~$n\to\infty$ along elements in~$N(p,\varepsilon)$, so does the right hand side. As~$\frac{\phi(n)}{n}>\varepsilon$ along this sequence, it follows---after possibly passing to a further subsequence---that~$\mu_{n}^{\times}$ converges to a~$\gamma$-invariant probability measure on~$\rquotient{\Gamma}{G}$ and ergodicity of the invariant probability measure with respect to right multiplication by~$\gamma$ implies that this limit measure has to be the invariant probability measure. As the limit is independent of the subsequence, it follows that $\mu_{n}^{\times}$ converges to the invariant probability measure.

In order to make this argument precise, one can find an element~$\gamma$ with the desired property by considering a different space, namely the~$p$-adic extension of~$\rquotient{\Gamma}{G}$. This is obtained by considering the group~$\SLtwo(\R\times\Qp)\cong G\times\SLtwo(\Qp)$ instead where~$\Qp$ is the completion of~$\Q$ with respect to the~$p$-adic norm. Then~$\SLtwo(\Z[\frac{1}{p}])$ is a lattice in~$\SLtwo(\R\times\Qp)$ and the element
\begin{equation*}
  a_{p}=\left(\begin{pmatrix}
      p & 0 \\ 0 & p^{-1}
    \end{pmatrix},\begin{pmatrix}
      p & 0 \\ 0 & p^{-1}
    \end{pmatrix}\right)
\end{equation*}
is an element in the lattice. For~$t_{\infty}\in\R$ and~$t_{p}\in\Qp$ one calculates
\begin{equation*}
  a_{p}u_{(t_{\infty},t_{p})}a_{p}^{-1}=u_{p^{2}(t_{\infty},t_{p})}
\end{equation*}
as desired and the proof sketched above actually works in this case. It remains to take care of the fact that~$\liminf_{n}\frac{\phi(n)}{n}=0$. To this end and for the sake of a better rate of equidistribution, we replace the proof sketched above by an effective, more general argument which uses for every~$n\in\N$ some finite collection of valid primes at once.
\subsection*{Structure of the article}
The paper is organized as follows. In Section~\ref{sec:setup} we introduce the~$S$-arithmetic groups and identify the lattice we want to consider. In Section~\ref{sec:sobolev}, we introduce~$S$-arithmetic Sobolev norms on the homogeneous spaces under consideration. In Section~\ref{sec:longhorocycles}, we prove equidistribution of long horocycle orbits in the~$S$-arithmetic extension, which illustrates a technical step occurring again in the later, notationally more heavy steps of the proofs. In Section~\ref{sec:rationalpoints}, we prove equidistribution of the rational points of distance~$1$ in the~$S$-arithemtic extension. Finally, Section~\ref{sec:primitivepoints} provides a short discussion of the~$\times p$-map in the~$S$-arithmetic setup, the fact that it is mixing and finally the proof of Theorem~\ref{thm:mainthm}. Section~\ref{sec:twotorus} gives the argument involving rigidity phenomena for higher rank actions to prove equidistribution in the product of the two-torus and the modular surface.
\subsection*{Acknowledgements}
The authors would like to thank Jens Marklof for his comments on a preliminary version of the paper and the anonymous referee for his detailed report. We also thank Shahar Mozes and Uri Shapira for many discussions related to this problem. M.L.~would like to thank Elon Lindenstrauss and Andreas Wieser for many fruitful conversations. During the process of writing, M.L.~enjoyed the hospitality of the Ohio State University and the Hebrew University of Jerusalem.
%%%%%%%%%%%% Setup %%%%%%%%%%%
\section{The~$S$-arithmetic extension}\label{sec:setup}
\subsection{The modular surface}
Given a finite set of places~$S$ of~$\Q$, we let~$\QS=\prod_{p\in S}\Q_{p}$ be the product of the completions~$\Qp$ of~$\Q$ where~$\Q_{\infty}=\R$ and
\begin{equation*}
  \ZS=\Z\big[\big\{\tfrac{1}{p};p\in S\setminus\{\infty\}\big\}\big].
\end{equation*}
Given~$t\in\QS$ and~$p\in S$, we let~$t_{p}$ be the~$\Qp$-coordinate of~$t$. Let~$\Sf=S\setminus\{\infty\}$ denote the set of finite places in~$S$. We set~$\Zsf=\prod_{p\in\Sf}\Zp$ and if~$\infty\in S$, then we denote~$\Zs=\R\times\Zsf$. Given~$p\in S$ and~$t\in\Qp$, we denote by~$\imath_{p}:\Qp\to\QS$ the map sending~$t$ to the element~$\imath_{p}(t)$ satisfying~$\imath_{p}(t)_{p}=t$ and~$\imath_{p}(t)_{q}=0$ whenever~$q\neq p$.

We set~$\GS=\SLtwo(\QS)$,~$\Gs=\SLtwo(\Zs)$,~$\GSK=\SLtwo(\Zsf)$ and~$\GammaS=\SLtwo(\ZS)$ where we understand~$\GammaS$ as a subgroup of~$\SLtwo(\QS)$ via the diagonal embedding of~$\ZS$ in~$\QS$. Note that~$\SLtwo(\QS)$ is isomorphic to the direct product of the~$\SLtwo(\Qp)$ over all~$p\in S$. For any group, we will denote its identity element by~$\one$. Given~$p\in S$ and some~$g\in\SLtwo(\Qp)$, we will denote by~$\imath_{p}(g)\in\SLtwo(\QS)$ the element whose component equals the identity for all places in~$S\setminus\{p\}$ and~$g$ at the place~$p$. Conversely, given an element~$g\in\GS$ and a place~$p\in S$, we let~$g_{p}$ denote the~$p$-coordinate of~$g$ and more generally for a subset~$S^{\prime}\subseteq S$, we denote by~$g_{S^{\prime}}$ the projection of~$g$ to~$\SLtwo(\Q_{S^{\prime}})$. Given~$g\in\SLtwo(\Q)$, we denote by~$\Delta(g)$ the diagonal embedding in~$\SLtwo(\QS)$. Given a place~$p$ of~$\Q$, we write~$\Gp$ for~$G_{\{p\}}$. We let~$\Gr=\SLtwo(\R)$ and~$\Gammar=\SLtwo(\Z)$. The goal of this short section is to introduce general notation, to establish the well-known fact that~$\GammaS$ is a lattice in~$\GS$ if~$\infty\in S$ and to naturally relate the space~$\XS=\rquotient{\GammaS}{\GS}$ to the space~$\Xr=\rquotient{\Gammar}{\Gr}$, which is our space of interest. The relation is found by first proving that~$\XS\cong\rquotient{\GammaSK}{\Gs}$ where~$\GammaSK$ is identified with its image under the embedding in~$\Gs$ induced by the diagonal embedding of~$\Z$ in~$\Zs$. We will denote~$\YS=\rquotient{\GammaSK}{\Gs}$. The first step towards proving that~$\GammaS$ is a lattice in~$\GS$ (assuming~$\infty\in S$) is to show that~$\SLtwo$ has \emph{class number one}, which is expressed in the following proposition, for which we refer the reader to \cite{Rapinchuk1994}.
\begin{prop}\label{prop:isomorphism_qp_zp}
  The group~$\Gs$ acts transitively on~$\XS$ and the stabilizer of~$\GammaS$ in~$\Gs$ is~$\GammaSK$. In particular the map~$\GammaSK g\mapsto\GammaS g$ is an isomorphism~$\YS\cong\XS$ of~$\Gs$-spaces.
\end{prop}
The isomorphism~$\psi_{S}:\XS\to\YS$ in Proposition~\ref{prop:isomorphism_qp_zp} is given by writing a representative~$g$ in~$G_{S}$ as~$g=\gamma\eta_{S}$ with~$\gamma\in\GammaS$ and~$\eta_{S}\in\Gs$. It is relatively easy to see that~$\GammaSK$ is a non-uniform lattice in~$\Gs$ if~$\infty\in S$. One obtains the following
\begin{cor}\label{cor:extension}
  If~$\infty\in S$, then~$\GammaS$ is a lattice in~$\GS$ and~$\quotient{\XS}{K[0]}\cong\Xr$ as~$\Gr$-spaces.
\end{cor}
As of~$\Gs$-equivariance, the push-forward of any invariant probability measure on~$\XS$ under~$\psi_{S}$ is an invariant probability measure on~$\YS$. In particular the systems defined by~$\Gs\curvearrowright\YS$ and~$\Gs\curvearrowright\XS$ are isomorphic as dynamical systems. In what follows, we will abuse notation and denote by~$\nu_{S}$ both the invariant probability measure on~$\YS$ and the invariant (under~$\GS$) probability measure on~$\XS$.

Let~$H\leq\GS$ be a closed subgroup and assume that~$H$ is the set of~$\QS$-points of some algebraic group defined by polynomials with coefficients in~$\Q$. The~$\Zs$-points are defined by~$\Hs=H\cap\Gs$. We define the groups~$\HS$,~$\Hp$ and so on in the corresponding fashion. Furthermore, we write~$\HSK=H\cap K[0]$ and~$\HGamma=H\cap\GammaS$.
\subsection{Periodic orbits for horospherical subgroups}\label{sec:periodicorbits}
We define the subgroup
\begin{equation*}
  U_{S}=\set{u_{t}=\begin{pmatrix}1 & t\\0 & 1\end{pmatrix}}{t\in\QS}\leq\GS,
\end{equation*}
and in analogy to the real case, i.e.~$S=\{\infty\}$ want to look at the closed~$\US$-orbits in~$\XS$. To this end we fix the Haar measure~$m_{\QS}$ on~$\QS$ as the product of the Haar measures~$m_{\Qp}$ on the components~$\Qp$, ($p\in S$) where~$m_{\Qinfty}$ is the Lebesgue measure and~$m_{\Qp}$ is normalized so that~$m_{\Qp}(\Zp)=1$ for~$p\in\Sf$. Define the Haar measure~$m_{U_{S}}$ on~$U_{S}$ to be the push-forward of~$m_{\QS}$ under the isomorphism~$t\mapsto u_{t}$. A point~$x\in\XS$ has periodic~$U_{S}$-orbit if and only if there is some~$\alpha\in\QS^{\times}$ such that~$\Stab_{U_{S}}(x)=\{u_{\alpha t};t\in\ZS\}$ (cf.~\cite[Proposition~8.1]{KleinbockTomanov}). This can be used to show that a point~$x\in\XS$ has periodic~$U_{S}$-orbit if and only if it is of the form~$x=\GammaS au$ for some~$u\in U_{S}$ and some~$a\in A_{S}$ where
\begin{equation*}
  A_{S}=\set{a_{y}=
    \begin{pmatrix}
      y & 0 \\
      0 & y^{-1}
    \end{pmatrix}
  }{y\in\Q_{S}^{\times}}.
\end{equation*}
In what follows, we write~$\Ucal_{y}=\GammaS a_{y}\US$ whenever~$y\in\QS^{\times}$. Let~$y\in\QS^{\times}$, then the volume of the orbit~$\GammaS a_{y}U_{S}$ is the covolume of~$y^{-2}\ZS$ in~$\QS$, which equals~$\lvert y^{-2}\rvert_{S}=\prod_{p\in S}\lvert y_{p}^{-2}\rvert_{p}$. Let~$\lvert y\rvert_{\Sf}=\prod_{p\in\Sf}\lvert y_{p}\rvert_{p}\in\ZS$ and note that~$a_{\lvert y^{-1}\rvert_{\Sf}}$ is contained in~$A_{S}\cap\GammaS$. Using~$\GammaS a_{y}\US=\GammaS a_{z}\US\implies yz^{-1}\in\ZS$, we obtain the following
\begin{cor}
  There is a one-to-one correspondence between~$\R_{>0}\times\prod_{p\in\Sf}\Zp^{\times}$ and periodic~$U_{S}$-orbits, given by sending an element~$y$ to~$\Ucal_{y}$.
\end{cor}
%%%%%%%%%%%% Sobolev %%%%%%%%%%%
\section{$S$-arithmetic Sobolev norms and congruence quotients}\label{sec:sobolev}
In this section, we will introduce Sobolev norms and collect several properties used. These have been discussed in greater generality in \cite{EMMV} and we will often provide references instead of detailed proofs. Along the discussion, we will have to introduce the notion of a smooth function on certain~$S$-arithmetic quotient spaces. It will turn out, that such functions will come from smooth functions on congruence quotients of~$\SLtwo(\R)$. This feature will be very useful for the subsequent effective equidistribution statements, once we have found the relation between the Sobolev norms on these real homogeneous spaces and the Sobolev norms considered in the~$S$-arithmetic setup.

\subsection{The space of smooth functions on~$\XS$}
We make use of the following notation. We denote by~$\R^{\Sf}$ the set of functions from $\Sf$ to $\R$. It is convenient to think of elements in~$\R^{\Sf}$ as vectors in~$\R^{\lvert\Sf\rvert}$ whose entries are indexed by~$\Sf$. For~$m\in\R^{\Sf}$, we denote~$S^{m}=\prod_{p\in S_{f}}p^{m_{p}}$. If~$m\in\N^{\Sf}$, then we define~$\Zsf[m]=\prod_{p\in\Sf}p^{m_{p}}\Zp$, which is an ideal in~$\Zsf$. The Chinese Remainder Theorem yields~$\quotient{\Zsf}{\Zsf[m]}\cong\quotient{\Z}{S^{m}\Z}$. Applying the projection~$\Zsf\to\quotient{\Zsf}{\Zsf[m]}$ in each entry, yields a homomorphism
\begin{equation*}
  K[0]\to\SLtwo(\quotient{\Zsf}{\Zsf[m]}),
\end{equation*}
whose kernel~$K[m]$ is a closed subgroup of finite index and in particular a compact open subgroup. The restriction of the normalized Haar measure on~$K[0]$ to the subgroups~$K[m]$ yields a finite, bi-invariant Haar measure on these subgroups. Given a continuous, compactly supported function~$f$ on~$\XS$, some~$p\in\Sf$ and~$m_{p}\in\mathbb{N}_{0}$, we denote by~$K_{p}[m_{p}]$ the kernel of the homomorphism~$K[0]\to\SLtwo(\quotient{\Zp}{p^{m_{p}}\Zp})$,~$g\mapsto g_{p}\mod p^{m_{p}}\Zp$ and we write~$\Av_{p}[m_{p}](f)$ for the function defined by
\begin{equation*}
  \Av_{p}[m_{p}](f)(x)=\frac{1}{\vol{K_{p}[m_{p}]}}\int_{K_{p}[m_{p}]}f\big(x\imath_{p}(g)\big)\der{}g.
\end{equation*}
Furthermore we define~$\pr_{p}[m_{p}]=\Av_{p}[m_{p}]-\Av_{p}[m_{p}-1]$ for~$m_{p}\geq1$ and for simplicity write~$\pr_{p}[0]=\Av_{p}[0]$. The operator~$\pr_{p}[m_{p}]$ is called the \emph{level~$m_{p}$ projection at~$p$}. We note that~$K[m]=\prod_{p\in\Sf}K_{p}[m_{p}]$ for all~$m\in\N_{0}^{\Sf}$.
\begin{definition}
  A continuous function~$f$ on~$\XS$ is called \emph{smooth}, if it is invariant under~$K[m]$ for some~$m\in\N_{0}^{\Sf}$ and if it is smooth at the real place. The space of smooth functions is denoted by~$\smooth(\XS)$ and~$\compactsmooth(\XS)$ is the space of compactly supported smooth functions.
\end{definition}
For the notion of \emph{smoothness in the real place}, recall that~$\GammaS\leq\GS$ is a lattice and thus every point in~$\XS$ has a neighbourhood which is homeomorphic to some neighbourhood of the identity in~$\GS$. This neighbourhood contains an open neighbourhood which is a direct product of neighbourhoods of the identity in~$\Gp$ ($p\in S$). On such a neighbourhood the notion of smoothness in the real component can be defined using the notion of smoothness for Lie groups. Given~$m\in\N_{0}^{\Sf}$, we let~$\pr[m]=\prod_{p\in\Sf}\pr_{p}[m_{p}]$. This is well-defined, as the projections for distinct places commute. For~$f\in\compactsmooth(\XS)$ we have~$f=\sum_{m\in\N_{0}^{\Sf}}\pr[m]f$ and the right-hand side is a finite sum. Given~$f\in\compactsmooth(\XS)$, we call~$\pr[m]f$ the \emph{pure level-$m$ component} of~$f$.

Let~$N\in\N$, then~$\Gammar(N)$ denotes the congruence lattice for level~$N$, i.e.~the kernel of the homomorphism~$\SLtwo(\Z)\to\SLtwo(\quotient{\Z}{N\Z})$ induced by the canonical projection~$\Z\to\quotient{\Z}{N\Z}$. In what follows, we denote~$\Xr(N)=\rquotient{\Gammar(N)}{\Gr}$. One can show that~$\XS[m]=\quotient{\XS}{K[m]}$ is isomorphic to~$\Xr(S^{m})$ as a~$\Gr$-space with isomorphism given by
\begin{equation}
  \label{eq:congruence}\psi^{(m)}:\Xr(S^{m})\to\XS[m],\quad\Gammar(S^{m})g\mapsto\GammaS\imath_{\infty}(g)K[m].
\end{equation}
We denote by~$\pi^{(m)}:\XS\to\XS[m]$ the canonical projection. Then a function~$f\in\continuous(\XS)$ is smooth, if and only if there is some~$m\in\N_{0}^{\Sf}$ and a smooth function~$\tilde{f}_{m}\in\smooth(\rquotient{\Gammar(S^{m})}{\Gr})$ such that~$f=\tilde{f}_{m}\circ(\psi^{(m)})^{-1}\circ\pi^{(m)}$. The picture to keep in mind is the commuting diagram given in~\eqref{eq:commutingdiagram} where~$g$ denotes the action by some element~$g\in\Gr$.
\begin{equation}\label{eq:commutingdiagram}
  \begin{gathered}
    \xymatrix{
      \XS\ar@{->>}[rr]^{\pi^{(m)}}\ar[d]_{g} && \XS[m]\ar[d]_{g} && \Xr(S^{m})\ar[ll]_{\psi^{(m)}}\ar[d]^{g}\\
      \XS\ar@{->>}[rr]_{\pi^{(m)}} && \XS[m] && \Xr(S^{m})\ar[ll]^{\psi^{(m)}}
    }
  \end{gathered}
\end{equation}
Equivariance of~$\pi^{(m)}$ for the~$\Gr$-action implies that the push-forward of the~$\GS$-invariant probability measure on~$\XS$ to~$\XS[m]$ is a~$\Gr$-invariant probability measure on~$\XS[m]$. As the~$\Gr$-invariant probability measure~$m_{\Xr(S^{m})}$ on~$\Xr(S^{m})$ is unique, equivariance of~$\psi^{(m)}$ implies that~${\psi^{(m)}}_{\ast}m_{\Xr(S^{m})}$ is the unique~$\Gr$-invariant probability measure on~$\XS[m]$ and agrees with~${\pi^{(m)}}_{\ast}\nu_{S}$. Let now~$f\in C_{c}(\XS)$ and assume that~$f$ is invariant under~$K[m]$. Then there is a unique~$\check{f}_{m}\in C_{c}(\XS[m])$ such that~$f=\check{f}_{m}\circ\pi^{(m)}$. In particular, using~$\check{f}_{m}\circ\psi^{(m)}=\tilde{f}_{m}$, it follows that
\begin{equation}\label{eq:integralinvariantfunctions}
  \int_{\XS}f\der{}\nu_{S}=\int_{\XS[m]}\check{f}_{m}\der{}{\pi^{(m)}}_{\ast}\nu_{S}=\int_{\Xr(S^{m})}\tilde{f}_{m}\der{m_{\Xr(S^{m})}}.
\end{equation}

In what follows, we will ask for effective equidistribution results, i.e.~we examine the equidistribution properties of sequences of subsets of~$\XS$ and quantify the error in terms of the parametrization of the sequence and of the test function involved. The error rates rely on smoothness properties of the functions and we will hence only use smooth test functions. The implicit equidistribution statements then follow, as~$\compactsmooth(\XS)\subseteq C_{c}(\XS)$ is a dense subspace (with respect to the uniform topology).
\subsection{Noncompactness and the height function}\label{sec:heightfunction}
In what follows, let~$\liegZ\subseteq\mathrm{Mat}_{2,2}(\Z)$ denote the submodule generated by the elements
\begin{equation}\label{eq:basissl2}
  H=\begin{pmatrix}
    -1 & 0 \\ 0 & 1
  \end{pmatrix},\quad X=
  \begin{pmatrix}
    0 & 1 \\ 0 & 0
  \end{pmatrix},\quad Y=\begin{pmatrix}
    0 & 0 \\ 1 & 0
  \end{pmatrix}.
\end{equation}
When equipped with the bracket~$[v,w]=vw-wv$, ($v,w\in\liegZ$) this is an integral Lie algebra. The commutator relations for the generating set show that for any ring~$R$ we have~$[\lieg_{R},\lieg_{R}]\subseteq\lieg_{R}$ where~$\lieg_{R}=\liegZ\otimes_{\Z}R\cong\liesltwo(R)$. An explicit calculation shows furthermore, that~$\lieg_{R}$ is preserved by the adjoint action~$\Ad$ given by conjugation with elements in~$\SLtwo(R)$. Note that~$\liegZS$ is a lattice in~$\liegQS$, in the sense that it is a finitely generated~$\ZS$-module satisfying~$\liegQS=\liegZS\otimes_{\ZS}\QS$. For what follows, given~$d\in\N$ and~$u\in\QS^{d}$, we let~$\normS{u}=\prod_{p\in S}\normp{u_{p}}$ where~$\normp{\cdot}$ is the maximum of the~$p$-adic absolute value of the entries of~$u_{p}$. Here, by the ``$\infty$-adic absolute value'' we mean the usual absolute value on~$\R$. If~$u\in\GL_{d}(\QS)$, we write~$\vvvert u\vvvert_{S}=\max\{\normS{u},\normS{u^{-1}}\}$.
\begin{definition}
  The height function on~$\XS$ is defined as
  \begin{equation*}
    \height_{\XS}:\XS\to\R,\quad\height_{\XS}(x)=\sup\set{\normS{\Ad(g^{-1})v}^{-1}}{v\in\liegZS,\GammaS g=x}.
  \end{equation*}
\end{definition}
Note that the height function does not depend on the choice of the representative~$g$ of~$x$, as~$\liegZS$ is~$\Ad(\GammaS)$-invariant.
\begin{prop}\label{prop:propertiesheight}
  The height function is a proper map bounded away from~$0$ with the following properties:
  \begin{enumerate}
    \item For all~$g\in\GS$ and~$x\in\XS$ we have~$\height_{\XS}(xg)\ll\vvvert g\vvvert_{S}^{2}\height_{\XS}(x)$. If~$g_{\infty}=\one$, then the implicit constant is~$1$.
    \item \label{item:propertiesheight-invariancecompact} For all~$x\in\XS$ and all~$g\in\GSK$, we have~$\height_{\XS}(xg)=\height_{\XS}(x)$.
    \item There exist positive constants~$\kappa_{1},c_{1}$ such that for all~$x\in\XS$ the map~$g\mapsto xg$ defined on the set~$\set{g\in\GS}{d(g_{\infty},1)\leq c_{1}\height_{\XS}(x)^{-\kappa_{1}},g_{\Sf}\in\GSK}$ is injective.
  \end{enumerate}
\end{prop}
This is discussed in Appendix A of \cite{EMMV}. Observe that for every pair~$\Lambda\leq\Gamma$ of lattices in a group~$G$, an injectivity radius at~$x=\Gamma g\in\rquotient{\Gamma}{G}$ is also an injectivity radius at~$\tilde{x}=\Lambda g\in\rquotient{\Lambda}{G}$. To this end we denote
\begin{equation*}
  \height_{\Xr(S^{m})}(x)=\sup\set{\norm{\Ad(g^{-1})v}^{-1}}{v\in\liegZ,\Gammar(S^{m})g=x}
\end{equation*}
\begin{lem}\label{lem:relationheight}
  Let~$x\in\XS$ and identify~$xK[m]\in\Xr(S^{m})$ with its image under~$(\psi^{(m)})^{-1}$. Then~$\height_{\XS}(x)=\height_{\Xr(S^{m})}(xK[m])$.
\end{lem}
\begin{proof}
  Let~$v\in\liegZS$,~$g\in\GS$ and~$k\in K[m]$, then
  \begin{equation*}
    \lVert\Ad(k)\Ad(g^{-1})v\rVert_{S}=\lVert\Ad(g^{-1})v\rVert_{S},
  \end{equation*}
  as for all~$p\in\Sf$, the norm~$\lVert\cdot\rVert_{p}$ is~$\SLtwo(\Zp)$-invariant. Hence the height function on~$\XS$ descends to a well-defined function on~$\Xr(S^{m})$. It remains to show that
  \begin{equation*}
    \height_{\XS}(x)=\sup\set{\lVert\Ad(g^{-1})v\rVert_{\infty}^{-1}}{v\in\liegZ,\pi^{(m)}(x)=\Gammar(S^{m})g}.
  \end{equation*}
  Let~$g_{\infty}\in\Gr$ satisfy~$\pi^{(m)}(x)=\Gammar(S^{m})g_{\infty}$, then for any~$g\in\GS$ with the property that~$\GammaS g=x$, we know~$g\equiv g_{\infty}\mod\GSK$, so that~$\height_{\XS}(x)=\height_{\XS}(\GammaS g_{\infty})$ as argued above. We first show that the supremum is achieved for some~$w\in\liegZ$. First, it follows from discreteness of~$\Ad(g_{\infty}^{-1})\liegZS$ and properness of~$\lVert\cdot\rVert_{S}$ that the supremum is achieved for some~$v\in\liegZS$. Let~$\alpha\in\Z$ be a common denominator for the entries of~$v$, so that~$v=\alpha^{-1}w$ for some~$w\in\liegZ$. We can assume that~$\alpha$ is a product of the primes in~$\Sf$. It follows that
  \begin{equation*}
    \lVert\Ad(g_{\infty}^{-1})v\rVert_{S}=\lVert\Ad(g_{\infty}^{-1})w\rVert_{S}\prod_{p\in S}\lvert\alpha^{-1}\rvert_{p}=\lVert\Ad(g_{\infty}^{-1})w\rVert_{S}
  \end{equation*}
  and the supremum is achieved at~$w\in\liegZ$. We show that we can assume~$\lVert(\Ad(g_{\infty}^{-1})w)_{p}\rVert_{p}=1$ for all~$p\in\Sf$. First, note that~$(\Ad(g_{\infty}^{-1})w)_{p}=w$ for~$p\in\Sf$. This already implies that~$\lVert(\Ad(g_{\infty}^{-1})w)_{p}\rVert_{p}=\lVert w\rVert_{p}\leq 1$. Assume that~$\lVert w\rVert<1$, then~$w=pu$ for some~$u\in\liegZ$ and thus~$\lVert\Ad(g_{\infty}^{-1})w\rVert_{S}=\lVert\Ad(g_{\infty}^{-1})u\rVert_{S}$ as~$\lvert p\rvert_{q}=1$ for all~$q\in\Sf\setminus\{p\}$. In particular, after replacing~$w$ finitely many times in this way, we can assume that~$p^{-1}w\not\in\liegZ$ for all~$p\in\Sf$ and in particular that~$\lVert w\rVert_{p}=1$ for all~$p\in\Sf$. This shows the claim.
\end{proof}
For what follows, we denote by~$\Xfrak$ a choice of a basis of~$\liegR$ -- i.e.~a maximal linearly independent set of degree 1 differential operators at the identity in~$\Gr$ -- and by~$\Dcal_{D}(\Xfrak)$ the set of all monomials in~$\Xfrak$ of degree at most~$D$. These monomials define differential operators on~$\compactsmooth(\XS)$. To this end, a differential operator~$X$ at the identity of~$\Gr$ defines a differential operator~$\bar{X}$ on~$\XS$ which for~$f\in\compactsmooth(\XS)$ is given by
\begin{equation*}
  \bar{X}f(x)=X(f\circ p\circ l_{g})\quad(x=\GammaS g\in\XS),
\end{equation*}
where~$p:\GS\to\XS$ is the canonical projection and~$l_{g}$ is left-multiplication on~$\GS$ by~$g$. In what follows, we will abuse notation and just write~$Xf$ instead of~$\bar{X}f$.
\begin{definition}
  The~$\Ltwo$-Sobolev norm of degree~$D$ with respect to the basis~$\Xfrak$ on~$\compactsmooth(\XS)$ is the norm~$\Scal_{D}:\compactsmooth(\XS)\to[0,\infty)$ given by
  \begin{equation*}
    \Scal_{D}(f)^{2}=\sum_{m\in\N_{0}^{\Sf}}\sum_{X\in\Dcal_{D}(\Xfrak)}S^{Dm}\lVert\pr[m](1+\height_{\XS})^{D}Xf\rVert_{2}^{2}\quad(f\in\compactsmooth(\XS)).
  \end{equation*}
\end{definition}
It is easy to see that for~$D\leq D^{\prime}$ and two $\Ltwo$-Sobolev norms~$\Scal_{D}$,~$\Scal_{D^{\prime}}$ with respect to bases~$\Xfrak$ and~$\Xfrak^{\prime}$ respectively, we have
\begin{equation*}
  \Scal_{D}(f)\ll\Scal_{D^{\prime}}(f)\quad(f\in\compactsmooth(\XS)).
\end{equation*}
If~$\Xfrak=\Xfrak^{\prime}$, then the implicit constant can be set to one. In what follows, we will usually implicitly assume a fixed choice of a basis of~$\liegR$; it could be useful to think of the basis provided in~\eqref{eq:basissl2}. In what follows, we list several properties of $\Ltwo$-Sobolev norms. For the proofs we refer the reader to \cite[Appendix~A]{EMMV} and \cite[Section~3.3]{ERW}.
\begin{proposition}\label{prop:propertiessobolevnorms}
  Let~$\Scal$ be an~$\Ltwo$-Sobolev norm of degree~$D$ on~$\compactsmooth(\XS)$. Then the following are true.
  \begin{enumerate}
    \item \label{item:sobolevembedding} \emph{(Sobolev Embedding)} There is some~$D_{0}\in\N_{0}$ such that~$D\geq D_{0}$ implies
      \begin{equation*}
        \lVert f\rVert_{\infty}\ll\Scal(f)\quad\text{ for all }f\in\compactsmooth(\XS).
      \end{equation*}
    \item \label{item:continuityofrepresentation} \emph{(Continuity of the regular representation)} Let~$g\in\GS$ and~$f\in\compactsmooth(\XS)$. Define a function~$g\cdot f\in\compactsmooth(\XS)$ by~$g\cdot f(x)=f(xg)$, ($x\in\XS$). Then
      \begin{equation*}
        \Scal(g\cdot f)\ll\vvvert g\vvvert_{S}^{4D}\Scal(f).
      \end{equation*}
    \item \label{item:multiplicativity} There is some~$\Ltwo$-Sobolev norm~$\Scal_{D+r}$ of degree~$D+r$,~$r\in\N_{0}$, such that for all~$f_{1},f_{2}\in\compactsmooth(\XS)$ we have
      \begin{equation*}
        \Scal(f_{1}f_{2})\ll\Scal_{D+r}(f_{1})\Scal_{D+r}(f_{2}).
      \end{equation*}
  \end{enumerate}
\end{proposition}
An important feature of these Sobolev norms is their relation to Sobolev norms on the congruence quotients. Note that the definition of an $\Ltwo$-Sobolev norm as above includes the case~$S=\{\infty\}$ and in fact works for any lattice~$\Lambda\leq\Gr$. Hence by the~$\Ltwo$-Sobolev norm~$\Scal$ on~$\compactsmooth(\rquotient{\Lambda}{\Gr})$ of degree~$D$ with respect to the basis~$\Xfrak$ we mean the map
\begin{equation}
  \label{eq:realsobolevnorms}\Scal(f)^{2}=\sum_{X\in\Dcal_{D}(\Xfrak)}\lVert(1+\height)^{D}Xf\rVert_{2}^{2}\quad(f\in\compactsmooth(\rquotient{\Lambda}{\Gr})),
\end{equation}
where the height function~$\height$ is the height function for~$\Lambda$. This sort of Sobolev norm also satisfies the properties listed in Proposition~\ref{prop:propertiessobolevnorms}. The following lemma yields the desired relation between $\Ltwo$-Sobolev norms on~$\compactsmooth(\XS)$ and on~$\compactsmooth(\Xr(S^{m}))$ for~$m\in\N_{0}^{\Sf}$. It will be helpful to introduce a bit of notation.
\begin{lem}\label{lem:relationsobolevnorms}
  Let~$\Xfrak$ denote a basis of~$\liegR$ and let~$D\in\N_{0}$ be a degree. Let~$\Scal_{D}$ denote the~$\Ltwo$-Sobolev norm on~$\compactsmooth(\XS)$ of degree~$D$ with respect to~$\Xfrak$ and for~$m\in\N_{0}^{\Sf}$ let~$\Scal_{D,m}$ denote the~$\Ltwo$-Sobolev norm on~$\compactsmooth(\Xr(S^{m}))$ of degree~$D$ with respect to~$\Xfrak$. Let~$f\in\compactsmooth(\XS)$. Let~$\tilde{f}_{m}\in\compactsmooth(\Xr(S^{m}))$ such that~$\pr[m]f=\tilde{f}_{m}\circ(\psi^{(m)})^{-1}\circ\pi^{(m)}$. Then
  \begin{equation*}
    \Scal_{D}(f)^{2}=\sum_{m}S^{Dm}\Scal_{D,m}(\tilde{f}_{m})^{2}.
  \end{equation*}
\end{lem}
\begin{proof}
  In this proof, we will write~$\phi^{(m)}=(\psi^{(m)})^{-1}\circ\pi^{(m)}$. Note that~$\Gr$-equivariance of~$\psi^{(m)}$ and~$\pi^{(m)}$ implies that for all~$X\in\liegR$ and~$h\in\compactsmooth(\Xr(S^{m}))$ we have
  \begin{equation*}
    X(h\circ\phi^{(m)})(x)=\big(Xh\circ\phi^{(m)}\big)(x).
  \end{equation*}
  Inductively this then extends to polynomials in elements of~$\liegR$. Using Lemma~\ref{lem:relationheight}, we have~$\height_{\XS}=\height_{\Xr(S^{m})}\circ\phi^{(m)}$, so that~$m_{\Xr(S^{m})}={\phi^{(m)}}_{\ast}\nu_{S}$ implies
  \begin{equation*}
  \Scal_{D,m}(\tilde{f}_{m})^{2}=\sum_{X\in\Dcal_{D}(\Xfrak)}\lVert\pr[m](1+\height_{\XS})^{D}Xf\rVert_{\Ltwo(\XS,\nu_{S})}^{2},
\end{equation*}
by Proposition~\ref{prop:propertiesheight}~(\ref{item:propertiesheight-invariancecompact}) and the fact that the differential operators in~$\liegR$ commute with the level~$m_{p}$ projections for all~$p\in\Sf$. The lemma now follows from the definition of~$\Scal_{D}$.
\end{proof}
Of course any Sobolev norm extends to the space~$\compactsmooth(\XS)\oplus\C\mathbf{1}_{\XS}$ which is defined to be the space of smooth functions~$f:\XS\to\C$ for which there is a constant~$c\in\C$ satisfying~$f-c\in\compactsmooth(\XS)$.
\subsection{Sobolev norms on~$S$-arithmetic extensions of tori}
We will also be interested in the~$S$-arithmetic extension~$\rquotient{\ZS}{\QS}$ of the torus~$\T=\rquotient{\Z}{\R}$. More generally, given an integer~$N$, we will use the notation~$\T(N)=\rquotient{N\Z}{\R}$. We will give a quick discussion of this space, as well as of smooth functions and of Sobolev norms. Most of it can be seen as a special case of what was done previously, up to some simplifications. Hence we will keep the discussion fairly brief.

Let~$S$ be a finite set of primes including~$\infty$ and denote~$\TS=\rquotient{\ZS}{\QS}$. The space~$\TS$ is locally isomorphic to~$\QS$, as~$\ZS$ is a lattice (cf.~\cite{KleinbockTomanov}). Hence we can make the following
\begin{definition}
  A function~$\varphi:\TS\to\C$ is smooth, if it is smooth in the real direction and if there is some~$m\in\N_{0}^{\Sf}$ such that~$\varphi$ is invariant under the subgroup~$\Zsf[m]$. We denote by~$\smooth(\TS)$ the vector space of smooth functions on~$\TS$.
\end{definition}
Note that in this case, the quotient~$\TS$ is compact, as follows from the discussion below. Given~$p\in\Sf$ and~$m_{p}\in\N_{0}$, we will denote by~$\Av_{p}[m_{p}]:\smooth(\TS)\to\smooth(\TS)$ the averaging operator for the subgroup~$p^{m_{p}}\Zp$. As before, denote~$\pr_{p}[0]=\Av_{p}[0]$ and~$\pr_{p}[m_{p}]=\Av_{p}[m_{p}]-\Av_{p}[m_{p}-1]$ if~$m_{p}>0$. For~$m\in\N_{0}^{\Sf}$, we again set~$\Av[m]=\prod_{p\in\Sf}\Av_{p}[m_{p}]$ and~$\pr[m]=\prod_{p\in\Sf}\pr_{p}[m_{p}]$. Every~$f\in\smooth(\TS)$ satisfies~$f=\sum_{m\in\N_{0}^{\Sf}}\pr[m]f$ and the right-hand side is a finite sum. Fix a basis~$\Xfrak$ -- i.e.~any non-zero element -- of the Lie algebra of~$\R$. The~$\Ltwo$-Sobolev norm~$\Scal$ of degree~$D$ on~$\smooth(\TS)$ with respect to the basis~$\Xfrak$ is the norm given by
\begin{equation*}
  \Scal(f)^{2}=\sum_{m\in\N_{0}^{\Sf}}S^{Dm}\sum_{X\in\Dcal_{D}(\Xfrak)}\lVert \pr[m]Xf\rVert_{2}^{2},
\end{equation*}
where again~$\Dcal_{D}(\Xfrak)$ is the set of monomials of degree at most~$D$ in~$\Xfrak$. We remark here that as of compactness of~$\TS$ (and similarly for~$\T(S^{m})$,~$m\in\N_{0}^{\Sf}$), there is a uniform injectivity radius and hence we were able to choose the \emph{height function} (cf.~Section~\ref{sec:heightfunction}) to be constant equal to~$1$.

Similarly to the discussion of~$S$-arithmetic quotients of~$\SLtwo$, one has
\begin{equation*}
  \biquotient{\ZS}{\QS}{\Zsf[m]}\cong\biquotient{\Z}{\R\times\Zsf}{\Zsf[m]}\cong\T(S^{m}).
\end{equation*}
The first isomorphism follows immediately from the fact that~$\R\times\Zsf$ acts transitively on~$\rquotient{\ZS}{\QS}$, which again follows from density of~$\ZS$ in~$\QSf$. For the second isomorphism define a map
\begin{equation*}
  \T(S^{m})\to\biquotient{\Z}{\R\times\Zsf}{\Zsf[m]},\quad S^{m}\Z+v\mapsto\Z+\imath_{\infty}(v)+\Zsf[m].
\end{equation*}
This is well-defined, onto and injective, where injectivity follows from strong approximation \cite[Chapter~3,~Lemma~3.1]{Cassels}. Using exactly the same argument as for the proof of Lemma~\ref{lem:relationsobolevnorms}, one obtains
\begin{equation}\label{eq:relationsobolevnormstorus}
  \Scal_{D}(f)^{2}=\sum_{m\in\N_{0}^{\Sf}}S^{Dm}\Scal_{D,m}(\tilde{f}_{m})^{2},
\end{equation}
where~$\Scal_{D}$ is the~$\Ltwo$-Sobolev norm of degree~$D$ on~$\smooth(\TS)$ with respect to the basis~$\Xfrak$ on the Lie algebra of~$\R$ and~$\Scal_{D,m}$ is the~$\Ltwo$-Sobolev norm of degree~$D$ on~$\smooth(\T(S^{m}))$ with respect to the basis~$\Xfrak$.

For the sake of completeness, let us point out that there is an analog to the Sobolev embedding theorem for functions on $\TS$, cf.~Proposition~\ref{prop:propertiessobolevnorms}.
\begin{proposition}\label{prop:propertiessobolevnormstorus}
  Let $\Scal$ be an $\Ltwo$-Sobolev norm of degree $D$ on $\smooth(\TS)$. Then the following are true.
    \begin{enumerate}
    \item \label{item:sobolevembeddingtorus} \emph{(Sobolev Embedding)} There is some~$D_{0}\in\N_{0}$ such that~$D\geq D_{0}$ implies
      \begin{equation*}
        \lVert f\rVert_{\infty}\ll\Scal(f)\quad\text{ for all }f\in\smooth(\TS).
      \end{equation*}
    \item \label{item:multiplicativitytorus} There is some~$\Ltwo$-Sobolev norm~$\Scal_{D+r}$ of degree~$D+r$,~$r\in\N_{0}$, such that for all~$f_{1},f_{2}\in\smooth(\TS)$ we have
      \begin{equation*}
        \Scal(f_{1}f_{2})\ll\Scal_{D+r}(f_{1})\Scal_{D+r}(f_{2}).
      \end{equation*}
  \end{enumerate}
\end{proposition}
We leave it to the reader to adapt to this simpler situation the corresponding proofs in the references provided for Proposition~\ref{prop:propertiessobolevnorms}.
\begin{remark}
  The above discussion also has a higher dimensional generalization, i.e.~to smooth functions on the~$S$-arithmetic cover~$\TS^{n}$ of the~$n$-dimensional torus~$\T^{n}$. 
\end{remark}
\subsection{The maximal cross norm on the product}\label{sec:maximalcrossnorms}
We are mainly interested in examining equidistribution properties of subsets in the product~$\TS\times\XS$. For this we make use of a special kind of Sobolev norms on~$\TS\times\XS$, the so-called maximal cross norms (cf.~\cite{BEG}). We will consider the following set of test functions. Let~$\tensorcompactsmooth(\TS\times\XS)$ be the linear hull generated by the set of functions~$\varphi\otimes f$ for~$\varphi\in\smooth(\TS)$ and~$f\in\compactsmooth(\XS)$ where we define~$\varphi\otimes f(t,x)=\varphi(t)f(x)$ for all~$t\in\TS$ and all~$x\in\XS$. Then~$\tensorcompactsmooth(\TS\times\XS)$ is a dense subspace of~$\compact(\TS\times\XS)$.
\begin{definition}\label{def:sobolevnormproduct}
An~$\Ltwo$-maximal cross norm of degree~$(D_{1},D_{2})$ on~$\tensorcompactsmooth(\TS\times\XS)$ is a norm~$\Scal_{\Acal}$ which for~$F\in\tensorcompactsmooth(\TS\times\XS)$ is given by
\begin{equation*}
  \Scal_{\Acal}(F)=\inf\set{\sum_{i}\Scal_{D_{1},\TS}(\varphi_{i})\Scal_{D_{2},\XS}(f_{i})}{F=\sum_{i}\varphi_{i}\otimes f_{i}},
\end{equation*}
where~$\Scal_{D_{1},\TS}$ and~$\Scal_{D_{2},\XS}$ denote the~$\Ltwo$-Sobolev norms of degree~$D_{1}$ and~$D_{2}$ on~$\smooth(\TS)$ and~$\compactsmooth(\XS)$ for some fixed bases of~$\Lie(\R)$ and~$\liesltwo(\R)$ respectively.
\end{definition}
In what follows we will call~$\Ltwo$-maximal cross norms on~$\tensorcompactsmooth(\TS\times\XS)$ just cross norms.
We note that for any cross norm~$\Scal_{\Acal}$ on~$\tensorcompactsmooth(\TS\times\XS)$ of degree~$(D_{1},D_{2})$ there is a cross norm~$\Scal_{\Acal}^{\prime}$ on~$\tensorcompactsmooth(\TS\times\XS)$ of degree~$(D_{1}^{\prime},D_{2}^{\prime})$ with~$D_{i}\leq D_{i}^{\prime}$ ($i=1,2$) such that for all~$F,G\in\tensorcompactsmooth(\TS\times\XS)$ we have
\begin{equation}\label{eq:productsobolevnormonproduct}
  \Scal_{\Acal}(FG)\leq\Scal_{\Acal}^{\prime}(F)\Scal_{\Acal}^{\prime}(G).
\end{equation}
To this end assume~$F=\sum_{i}\varphi_{i}\otimes f_{i}$ and~$G=\sum_{j}\psi_{j}\otimes g_{j}$. Then using Propositions~\ref{prop:propertiessobolevnorms}~(\ref{item:multiplicativity}) and~\ref{prop:propertiessobolevnormstorus}~(\ref{item:multiplicativitytorus}), we obtain
\begin{align*}
  \Scal_{\Acal}(FG)&\leq\sum_{i,j}\Scal_{D_{1}}(\varphi_{i}\psi_{j})\Scal_{D_{2}}(f_{i}g_{j})\\
  &\leq\left(\sum_{i}\Scal_{D_{1}+r}(\varphi_{i})\Scal_{D_{2}+r}(f_{i})\right)\left(\sum_{j}\Scal_{D_{1}+r}(\psi_{j})\Scal_{D_{2}+r}(g_{j})\right)
\end{align*}
for~$r$ sufficiently large but independent of~$F$ and~$G$. Hence choose~$D_{i}^{\prime}=D_{i}+r$,~$(i=1,2)$ and let~$\Scal_{\Acal}^{\prime}$ be the cross norm defined using~$\Scal_{D_{1}^{\prime}}$ and~$\Scal_{D_{2}^{\prime}}$. Let~$\varepsilon>0$ arbitrary and assume that the representations of~$F$ and~$G$ were chosen so that
\begin{equation*}
  \Scal_{\Acal}^{\prime}(F)+\varepsilon\geq\sum_{i}\Scal_{D_{1}^{\prime}}(\varphi_{i})\Scal_{D_{2}^{\prime}}(f_{i})\quad\text{and}\quad\Scal_{\Acal}^{\prime}(G)+\varepsilon\geq\sum_{j}\Scal_{D_{1}^{\prime}}(\psi_{j})\Scal_{D_{2}^{\prime}}(g_{j}).
\end{equation*}
Using this for the preceding estimate, one obtains
\begin{equation*}
  \Scal_{\Acal}(FG)\leq\Scal_{\Acal}^{\prime}(F)\Scal_{\Acal}^{\prime}(G)+O_{F,G}(\varepsilon),
\end{equation*}
and as~$\varepsilon$ was arbitrary, the claim follows.

For later use, let us explicitly state the following
\begin{prop}\label{prop:sobolevembeddingcrossnorm}
  There exist $D_{1},D_{2}\in\N$ such that the following is true. Given any cross norm $\Scal_{\Acal}$ on $\tensorcompactsmooth(\TS\times\XS)$ of degree $(D_{1}^{\prime},D_{2}^{\prime})$ with $D_{1}\leq D_{1}^{\prime}$ and $D_{2}\leq D_{2}^{\prime}$ and any function $F\in\tensorcompactsmooth(\TS\times\XS)$, we have
  \begin{equation*}
    \lVert F\rVert_{\infty}\ll_{D_{1}^{\prime},D_{2}^{\prime}}\Scal_{\Acal}(F).
  \end{equation*}
\end{prop}
\begin{proof}
  Let $\varphi\otimes f$ be any pure tensor, then
  \begin{equation*}
    \lVert\varphi\otimes f\rVert_{\infty}=\lVert\varphi\lVert\rVert f\rVert_{\infty}.
  \end{equation*}
  Hence, as of Propositions~\ref{prop:propertiessobolevnorms}~(\ref{item:sobolevembedding}) and~\ref{prop:propertiessobolevnormstorus}~(\ref{item:sobolevembeddingtorus}), we obtain
  \begin{equation*}
    \lVert\varphi\otimes f\rVert_{\infty}\ll_{D_{1}^{\prime},D_{2}^{\prime}}\Scal_{D_{1}^{\prime}}(\varphi)\Scal_{D_{2}^{\prime}}(f)
  \end{equation*}
  for all sufficiently large $D_{1}^{\prime}$ and $D_{2}^{\prime}$. The claim now follows from the triangle inequality. 
\end{proof}
\subsection{Comparing maximal cross-norms and $\Ltwo$-Sobolev norms}
The maximal cross-norms defined above are well-defined on a dense subspace of~$\compactsmooth(\TS\times\XS)$. We want to close the discussion with a treatment of the relation between maximal cross-norms and Sobolev norms on~$\compactsmooth(\TS\times\XS)$. This justifies our decision to restrict to functions in~$\tensorcompactsmooth(\TS\times\XS)$ for the remainder of the article.

Let us first say what we mean by an $\Ltwo$-Sobolev norm of degree $D$ on $\compactsmooth(\TS\times\XS)$. Given $m\in\N_{0}^{\Sf}$ we denote by the operator $\pr[m]$ from $\compactsmooth(\TS\times\XS)$ to $\compactsmooth(\T(S^{m})\times\Xr(S^{m}))$ the level-$m$ projection with respect to the compact-open subgroup $\Zsf[m]\times K[m]$ of $\QS\times\GS$. Let $\Xfrak$ denote a basis of the Lie algebra of $\R\times\Gr$ and let $\Dcal_{D}(\Xfrak)$ denote the set of monomials of degree at most $D$ in $\Xfrak$. An $\Ltwo$-Sobolev norm of degree $D$ is a map $\Scal$ from $\compactsmooth(\TS\times\XS)$ to $[0,\infty)$ of the form
\begin{equation*}
  \Scal(F)^{2}=\sum_{m\in\N_{0}^{\Sf}}S^{Dm}\sum_{X\in\Dcal_{D}(\Xfrak)}\lVert\pr[m](1+\height_{\XS})^{D}XF\rVert_{2}^{2}.
\end{equation*}

In order to give a clear relation between maximal cross-norms and $\Ltwo$-Sobolev norms, we are going to show the following
\begin{thm}\label{thm:boundsobolevnorm}
  Let~$\Scal_{\Acal}$ be a cross norm on~$\tensorcompactsmooth(\TS\times\XS)$. There exists an $\Ltwo$-Sobolev norm~$\Scal$ on~$\compactsmooth(\TS\times\XS)$ such that:
  \begin{enumerate}
    \item For all~$F\in\tensorcompactsmooth(\TS\times\XS)$ we have~$\Scal_{\Acal}(F)\leq\Scal(F)$.
    \item The space~$\tensorcompactsmooth(\TS\times\XS)\subseteq\compactsmooth(\TS\times\XS)$ is a dense subspace for the topology defined by $\Scal$. That is, for every~$F\in\compactsmooth(\TS\times\XS)$ there is a sequence~$F_{N}\in\tensorcompactsmooth(\TS\times\XS)$ such that~$\Scal(F-F_{N})\overset{N\to\infty}{\longrightarrow}0$.
  \end{enumerate}
\end{thm}
This implies the following
\begin{cor}\label{cor:effectiveequidistcrossnormimplieseffectiveequidist}
  Let~$\mu,\nu$ be probability measures on~$\TS\times\XS$ and assume that~$\varepsilon>0$ is such that
  \begin{equation*}
    \bigg\lvert\int_{\TS\times\XS}F\der\mu-\int_{\TS\times\XS}F\der\nu\bigg\rvert\leq\varepsilon\Scal_{\Acal}(F)
  \end{equation*}
  for all~$F\in\tensorcompactsmooth(\TS\times\XS)$ and some cross norm~$\Scal_{\Acal}$ on~$\tensorcompactsmooth(\TS\times\XS)$. Then there is an $\Ltwo$-Sobolev norm~$\Scal$ on~$\compactsmooth(\TS\times\XS)$ such that
  \begin{equation*}
    \bigg\lvert\int_{\TS\times\XS}F\der\mu-\int_{\TS\times\XS}F\der\nu\bigg\rvert\ll\varepsilon\Scal(F)
  \end{equation*}
  for all~$F\in\compactsmooth(\TS\times\XS)$.
\end{cor}
Before we deduce Corollary~\ref{cor:effectiveequidistcrossnormimplieseffectiveequidist} from Theorem~\ref{thm:boundsobolevnorm}, for the sake of completeness, we want to quickly deduce a Sobolev embedding theorem for $\Ltwo$-Sobolev norms on $\compactsmooth(\TS\times\XS)$ of sufficiently large degree. Of course, it would also be possible to prove a more general version of Proposition~\ref{prop:propertiessobolevnorms} instead. We fix a cross norm $\Scal_{\Acal}$ as in Proposition~\ref{prop:sobolevembeddingcrossnorm} and let $\Scal$ be a corresponding $\Ltwo$-Sobolev norm on $\compactsmooth(\TS\times\XS)$ as in Theorem~\ref{thm:boundsobolevnorm}. Let $F\in\compactsmooth(\TS\times\XS)$ arbitrary. Choose any sequence of functions $F_{n}\in\tensorcompactsmooth(\TS\times\XS)$ such that $F_{n}\to F$ as $n\to\infty$ with respect to $\Scal$. As $\Scal$ bounds the $\Ltwo$-norm of functions on $\TS\times\XS$, we can assume without loss of generality that $F_{n}$ converges to $F$ pointwise almost surely. In particular continuity implies that $F_{n}$ converges to $F$ pointwise. Let now $(t,x)\in\TS\times\XS$ arbitrary, then
\begin{equation*}
  \lvert F(t,x)\rvert=\lim_{n\to\infty}\lvert F_{n}(t,x)\rvert\ll\lim_{n\to\infty}\Scal_{\Acal}(F_{n})\leq\lim_{n\to\infty}\Scal(F_{n})\leq\Scal(F).
\end{equation*}
As $(t,x)$ was arbitrary, it follows that
\begin{equation}\label{eq:sobolevembeddingproduct}
  \lVert F\rVert_{\infty}\ll\Scal(F)
\end{equation}
with implicit constant independent of $F$.
\begin{proof}[Proof of Corollary~\ref{cor:effectiveequidistcrossnormimplieseffectiveequidist}]
  We write~$\mu(F)$ and~$\nu(F)$ for the integral of~$F$ against~$\mu$ and~$\nu$ respectively. We choose an $\Ltwo$-Sobolev norm as in Theorem~\ref{thm:boundsobolevnorm} and assume without loss of generality that the degree of~$\Scal$ is sufficiently large for Equation~\eqref{eq:sobolevembeddingproduct} to hold.  Using Theorem~\ref{thm:boundsobolevnorm}, there is a function~$F_{\varepsilon}\in\tensorcompactsmooth(\TS\times\XS)$ such that
  \begin{equation*}
    \Scal(F-F_{\varepsilon})\leq\min\{\varepsilon,1\}\Scal(F).
  \end{equation*}
  Hence we have~$\Scal_{\Acal}(F_{\varepsilon})\leq 2\Scal(F)$ and thus
  \begin{align*}
    \lvert\mu(F)-\nu(F)\rvert&\leq 2\lVert F-F_{\varepsilon}\rVert_{\infty}+\lvert\mu(F_{\varepsilon})-\nu(F_{\varepsilon})\rvert\\
    &\leq 2\varepsilon\Scal(F)+\varepsilon\Scal_{\Acal}(F_{\varepsilon})\leq4\varepsilon\Scal(F).
  \end{align*}
\end{proof}
Corollary~\ref{cor:effectiveequidistcrossnormimplieseffectiveequidist} implies that effective equidistribution of a sequence of measures (with a rate) with respect to test functions in~$\tensorcompactsmooth(\TS\times\XS)$ and a maximum cross norm~$\Scal_{\Acal}$ implies effective equidistribution (with the same rate) of the sequence with respect to test functions in~$\compactsmooth(\TS\times\XS)$ and some~$\Ltwo$-Sobolev norm~$\Scal$. Hence after the proof of Theorem~\ref{thm:boundsobolevnorm} we will use~$\tensorcompactsmooth(\TS\times\XS)$ as our set of test functions and we will use the term $\Ltwo$-Sobolev norm also for cross norms. The proof of Theorem~\ref{thm:boundsobolevnorm} makes use of Fourier series. Let~$F\in\compactsmooth(\TS\times\XS)$ and~$m\in\N_{0}^{\Sf}$. The function~$\pr[m]F$ is a function on~$\T(S^{m})\times\Xr(S^{m})$ and hence it has an associated Fourier expansion, i.e.~for all~$t\in\TS$,~$x\in\XS$ we have
\begin{equation}\label{eq:fourierseriespurelevel}
  \pr[m]F(t,x)=\sum_{n\in\Z}a_{n}^{(m)}(F)(x)\chi_{n}^{(m)}(t),
\end{equation}
where
\begin{align*}
  \chi_{n}^{(m)}(t)&=e^{2\pi\ii n\frac{t}{S^{m}}}&\big(t\in\T(S^{m})\big),\\
  a_{n}^{(m)}(F)(x)&=\frac{1}{S^{m}}\int_{\T(S^{m})}\pr[m]F(t,x)\chi_{-n}^{(m)}(t)\der t&\big(x\in\XS\big).
\end{align*}
Let~$X$ be a differential operator on~$\TS\times\XS$ and assume that~$X$ can be written as~$X=X_{2}X_{1}$ where~$X_{1}$ is a differential operator on~$\TS$ and~$X_{2}$ is a differential operator on~$\XS$. We denote by~$\lVert X_{1}\rVert_{1}$ the total degree of the differential operator~$X_{1}$. Then by Parseval's and Fubini's Theorems we get
\begin{equation}\label{eq:componentssobolevnorm}
  \lVert\pr[m](1+\height_{\XS})^{D}XF\rVert_{2}^{2}=(2\pi S^{-m})^{2\lVert X_{1}\rVert_{1}}\sum_{n\in\Z}n^{2\lVert X_{1}\rVert_{1}}\lVert (1+\height_{\Xr(S^{m})})^{D}X_{2}a_{n}^{(m)}(F)\rVert_{2}^{2}.
\end{equation}
In particular, for any~$\Ltwo$-Sobolev norm~$\Scal$ of degree~$D$ on~$\compactsmooth(\TS\times\XS)$ we obtain
\begin{equation}\label{eq:boundsobolevnorm}
  \Scal(F)^{2}\ll\sum_{m\in\N_{0}^{\Sf}}S^{Dm}\sum_{n\in\Z}(1+n^{2D})\Scal_{D,m}\big(a_{n}^{(m)}(F)\big)^{2},
\end{equation}
where~we denote by $\Scal_{D,m}$ a family of~$\Ltwo$-Sobolev norm on~$\compactsmooth(\Xr(S^{m}))$ for a uniform choice of a basis of the Lie algebra of $\Gr$. In what follows, we will suppress the dependence of the norm on $m$ using the correspondence between smooth functions on $\Xr(S^{m})$ and functions on $\XS$ invariant exactly under $K[m]$. At first glance it might not be obvious why the right-hand side is finite. Recall however that smoothness of~$F$ implies that the outer sum is actually a finite sum. It suffices to prove finiteness for each of the finitely many summands. This is done in the following
\begin{lem}\label{lem:sobolevparseval}
  Let~$F\in\compactsmooth\big(\T(S^{m})\times\Xr(S^{m})\big)$. Let~$\Scal_{D}$ be an~$\Ltwo$-Sobolev norm of degree~$D$ on~$\compactsmooth\big(\Xr(S^{m})\big)$. Then
  \begin{equation*}
    \sum_{n\in\Z}(1+n^{2D})\Scal_{D}\big(a_{n}^{(m)}(F)\big)^{2}<\infty.
  \end{equation*}
\end{lem}
\begin{proof}
  Let~$X$ be any differential operator on~$\Xr(S^{m})$, then~$a_{n}^{(m)}(XF)=Xa_{n}^{(m)}(F)$ as of Lebesgue's Dominated Convergence Theorem. We denote by~$F^{(l)}$ the~$l$-th derivative of~$F$ in the torus component. Note that the set~$K$ defined as the projection of the support of~$F$ to~$\Xr(S^{m})$ is a compact set and that~$a_{n}^{(m)}(F)(x)=0$ for all~$x\not\in K$ and the height function is bounded on~$K$. For every fixed~$x\in\Xr(S^{m})$, the function defined by~$t\mapsto XF(t,x)$ is smooth in~$t$ and thus by Fubini's Theorem we get
  \begin{align*}
    \int_{\Xr(S^{m})}\sum_{n\in\Z}&(1+n^{2D})\big\lvert \big(1+\height_{\Xr(S^{m})}(x)\big)^{D}Xa_{n}^{(m)}(F)(x)\big\rvert^{2}\der x\\
    &\ll_{K}\int_{\Xr(S^{m})}\sum_{n\in\Z}(1+n^{2D})\big\lvert a_{n}^{(m)}(XF)(x)\big\rvert^{2}\der x\\
    &=\int_{\Xr(S^{m})}\lVert XF(\cdot,x)\rVert_{2}^{2}+\lVert XF^{(D)}(\cdot,x)\rVert_{2}^{2}\der x\\
    &=\lVert XF\rVert_{2}^{2}+\lVert XF^{(D)}\rVert_{2}^{2}<\infty,
  \end{align*}
  where we again used Parseval's Theorem in the first equality. Applying Fubini once more, we can exchange integration and summation for the expression we need to bound, i.e.~letting $\mathfrak{X}$ denote the basis of the Lie algebra of $\Gr$ used to define $\Scal_{D}$, we have
  \begin{align*}
    \sum_{n\in\Z}(1+n^{2D})\Scal_{D}\big(a_{n}^{(m)}(F)\big)^{2}&=\sum_{n\in\Z}(1+n^{2D})\sum_{X\in\Dcal_{D}(\Xfrak)}\lVert(1+\height_{\Xr(S^{m})})^{D}Xa_{n}^{(m)}(F)\rVert_{2}^{2} \\
    &\ll_{K}\sum_{X\in\Dcal_{D}(\Xfrak)}\lVert XF\rVert_{2}^{2}+\lVert XF^{(D)}\rVert_{2}^{2}.
  \end{align*}
  By the preceding discussion, the latter is a finite sum of finite expressions.
\end{proof}
We next show that the right hand side in~\eqref{eq:boundsobolevnorm} can be bounded from above by an $\Ltwo$-Sobolev norm of larger degree. Using this, we will be able to finally prove Theorem~\ref{thm:boundsobolevnorm}.
\begin{lem}\label{lem:fourierseriesnormbound}
  Let~$\Scal_{D}$ be an~$\Ltwo$-Sobolev norm of degree~$D$ on~$\compactsmooth(\XS)$, then there is an~$\Ltwo$-Sobolev norm~$\Scal_{2D}$ of degree~$2D$ on~$\compactsmooth(\TS\times\XS)$ such that for all~$F\in\compactsmooth(\TS\times\XS)$ we have
  \begin{equation*}
    \sum_{m\in\N_{0}^{\Sf}}S^{Dm}\sum_{n\in\Z}(1+n^{2D})\Scal_{D}\big(a_{n}^{(m)}(F)\big)^{2}\ll\Scal_{2D}(F)^{2}
  \end{equation*}
  with implicit constant independent of~$F$.
\end{lem}
\begin{proof}
  We assume that~$\Scal_{D}$ is defined using the basis~$\Xfrak$. We fix a generator $Y$ of the Lie algebra of $\R$ and let $\Xfrak^{\prime}=\Xfrak\cup\{Y\}$. Define an $\Ltwo$-Sobolev norm $\Scal_{2D}$ on $\compactsmooth(\TS\times\XS)$ using the basis $\Xfrak^{\prime}$. We again denote by~$\lVert X_{1}\rVert_{1}$ the total degree of the differential operator~$X_{1}$. Using~\eqref{eq:componentssobolevnorm}, Proposition~\ref{prop:propertiesheight} and Lemma~\ref{lem:relationheight} one calculates
  \begin{align*}
    \Scal_{2D}(F)^{2}%&=\sum_{m\in\N_{0}^{\Sf}}S^{2Dm}\sum_{X\in\Dcal_{2D}(\Xfrak^{\prime})}\lVert\pr[m](1+\height_{\XS})^{2D}XF\rVert_{2}^{2}\\
    %&=\sum_{m\in\N_{0}^{\Sf}}S^{2Dm}\sum_{X\in\Dcal_{2D}(\Xfrak^{\prime})}(2\pi S^{-m})^{2\lVert X_{1}\rVert_{1}}\sum_{n\in\Z}n^{2\lVert X_{1}\rVert_{1}}\lVert(1+\height_{\Xr(S^{m})})^{2D}X_{2}a_{n}^{(m)}(F)\rVert_{2}^{2}\\
    &\gg\sum_{m\in\N_{0}^{\Sf}}S^{2Dm}\sum_{X\in\Dcal_{2D}(\Xfrak^{\prime})}\sum_{n\in\Z}n^{2\lVert X_{1}\rVert_{1}}\lVert(1+\height_{\Xr(S^{m})})^{2D}X_{2}a_{n}^{(m)}(F)\rVert_{2}^{2}\\
    &=\sum_{m\in\N_{0}^{\Sf}}S^{2Dm}\sum_{X\in\Dcal_{2D}(\Xfrak)}\sum_{n\in\Z}\lVert(1+\height_{\Xr(S^{m})})^{2D}Xa_{n}^{(m)}(F)\rVert_{2}^{2}\\
    &\qquad+\sum_{m\in\N_{0}^{\Sf}}S^{2Dm}\sum_{\ell=1}^{2D}\sum_{X\in\Dcal_{2D-\ell}(\Xfrak)}\sum_{n\in\Z}n^{2\ell}\lVert(1+\height_{\Xr(S^{m})})^{2D}Xa_{n}^{(m)}(F)\rVert_{2}^{2}\\
    &\geq%\sum_{m\in\N_{0}^{\Sf}}S^{2Dm}\sum_{X\in\Dcal_{D}(\Xfrak)}\sum_{n\in\Z}\lVert(1+\height_{\Xr(S^{m})})^{D}Xa_{n}^{(m)}(F)\rVert_{2}^{2}\\
    %&\qquad+\sum_{m\in\N_{0}^{\Sf}}S^{2Dm}\sum_{X\in\Dcal_{D}(\Xfrak)}\sum_{n\in\Z}n^{2D}\lVert(1+\height_{\Xr(S^{m})})^{D}Xa_{n}^{(m)}(F)\rVert_{2}^{2}\\
    %&=
    \sum_{m\in\N_{0}^{\Sf}}S^{2Dm}\sum_{n\in\Z}(1+n^{2D})\sum_{X\in\Dcal_{D}(\Xfrak)}\lVert(1+\height_{\Xr(S^{m})})^{D}Xa_{n}^{(m)}(F)\rVert_{2}^{2}\\
    &=\sum_{m\in\N_{0}^{\Sf}}S^{Dm}\sum_{n\in\Z}(1+n^{2D})\Scal_{D}\big(a_{n}^{(m)}(F)\big)^{2}.
  \end{align*}
\end{proof}
\begin{proof}[Proof of Theorem~\ref{thm:boundsobolevnorm}]
  Before we start with the proof, let us introduce a natural way to approximate a compactly supported smooth function on $\TS\times\XS$. Let~$F\in\compactsmooth(\TS\times\XS)$ arbitrary and fix~$N\in\N_{0}$. We define~$F_{N}$ by
  \begin{equation*}
    F_{N}(t,x)=\sum_{m\in\N_{0}^{\Sf}}\sum_{\lvert n\rvert\leq N}a_{n}^{(m)}(F)(x)\chi_{n}^{(m)}(t)\quad(t\in\TS,x\in\XS).
  \end{equation*}
  It will be convenient to set~$F_{-1}=0$. Note that~$F_{N}\in\tensorcompactsmooth(\TS\times\XS)$ for all~$N\geq -1$ and that there is some~$D\in\N$ such that
  \begin{equation*}
    \Scal_{\Acal}(F_{N}-F_{N+k})\leq\sum_{m\in\N_{0}}\sum_{\lvert n\rvert=N+1}^{N+k}\Scal_{D,\TS}(\chi_{n}^{(m)})\Scal_{D,\XS}\big(a_{n}^{(m)}(F)\big).
  \end{equation*}
  We also note for later use that
  \begin{equation*}
    \sum_{m\in\N_{0}^{\Sf}}S^{-Dm}=\prod_{p\in\Sf}(1-p^{D})^{-1}<\infty
  \end{equation*}
  as well as
  \begin{equation*}
    \sum_{n\in\Z}\frac{1}{(1+n^{2D})^{\alpha}}<\infty
  \end{equation*}
  for all~$\alpha>\frac{1}{2}$. Furthermore we notice that~$\Scal_{D,\TS}(\chi_{n}^{(m)})\ll(1+n^{D})$. Using Cauchy-Schwarz, we obtain
  \begin{align}
      \Scal_{\Acal}(F_{N}-F_{N+k})^{2}&\ll\sum_{m\in\N_{0}^{\Sf}}S^{Dm}\Big(\sum_{\lvert n\rvert=N+1}^{N+k}\Scal_{D,\TS}(\chi_{n}^{(m)})\Scal_{D,\XS}\big(a_{n}^{(m)}(F)\big)\Big)^{2}\nonumber\\
      &\ll\sum_{m\in\N_{0}^{\Sf}}S^{Dm}\sum_{\lvert n\rvert=N+1}^{N+k}(1+n^{2D})\Scal_{D,\XS}\big(a_{n}^{(m)}(F)\big)^{2}\label{eq:intermediatebound}.
    \end{align}
  
  We now turn to the actual proof of the theorem. We assume throughout the discussion that all Sobolev norms are chosen of sufficiently large degree for the various Sobolev Embedding Theorems to hold. In order to show the first claim, let~$F\in\tensorcompactsmooth(\TS\times\XS)$ arbitrary. Combining the bound from~\eqref{eq:intermediatebound} with the finiteness in Lemma~\ref{lem:fourierseriesnormbound} the sequence of approximations~$F_{N}$ is a Cauchy-sequence with respect to~$\Scal_{\Acal}$. Using the Sobolev Embedding Theorem on the product, we know that~$F_{N}$ converges to~$F$ pointwise, i.e.~by smoothness that $F$ is the limit of the sequence $F_{N}$ with respect to $\Scal_{\Acal}$. Hence, using the bound from~\eqref{eq:intermediatebound} once again, we find
  \begin{align*}
    \Scal_{\Acal}(F)^{2}&\ll\lim_{N\to\infty}\Scal_{\Acal}(F_{N})^{2}=\lim_{N\to\infty}\Scal_{\Acal}(F_{-1}-F_{N})^{2}\\
    &\ll\lim_{N\to\infty}\sum_{m\in\N_{0}^{\Sf}}S^{Dm}\sum_{n=-N}^{N}(1+n^{2D})\Scal_{D,\XS}\big(a_{n}^{(m)}(F)\big)^{2}.
  \end{align*}
  Hence the first claim of Theorem~\ref{thm:boundsobolevnorm} follows from Lemma~\ref{lem:fourierseriesnormbound} with~$\Scal$ of degree~$2D$.

  For the second claim, let~$F\in\compactsmooth(\TS\times\XS)$ and define the approximations~$F_{N}$ as above. Using~\eqref{eq:boundsobolevnorm} and Lemma~\ref{lem:sobolevparseval}, we know that
  \begin{equation*}
    \Scal_{2D}(F-F_{N})^{2}\ll\sum_{m\in\N_{0}^{\Sf}}S^{2Dm}\sum_{\lvert n\rvert>N}(1+n^{4D})\Scal_{2D,\XS}\big(a_{n}^{(m)}(F)\big)^{2}\overset{N\to\infty}{\longrightarrow}0.
  \end{equation*}
  As~$F_{N}\in\tensorcompactsmooth(\TS\times\XS)$ for all~$N\in\N$, the second part of Theorem~\ref{thm:boundsobolevnorm} follows.
\end{proof}
%%%%%%%%%%%% Horocycles %%%%%%%%%%%
\section{Congruence quotients and effective~$S$-arithmetic equidistribution}\label{sec:longhorocycles}
In this section we want to illustrate the relation between equidistribution of orbits in congruence quotients and equidistribution in~$S$-arithmetic quotients. We illustrate the relationship by proving effective equidistribution of horocycle orbits~$\Ucal_{y}$ (cf.~Section~\ref{sec:periodicorbits}) for~$y\in(0,\infty)$ as~$y\to0$. This is not new and strictly speaking, the equidistribution in the~$S$-arithmetic quotient is formally not required for what follows. However, we will later prove equidistribution of certain sparse subsets of~$\Ucal_{y}$, so that it is natural to ask whether the full set equidistributes. Moreover, the argument used here gives a simple illustration of part of the procedure that will be applied for the equidistribution of rational points.

The effective equidistribution of the orbits~$\Gammar(S^{m}) a_{y}U_{\infty}$ for varying~$m\in\Z^{\Sf}$ follows from a theorem by Sarnak \cite{Sarnak1981}. We will use this to prove equidistribution of rational points of a certain denominator along these periodic orbits. From our perspective, these results lie in the realm of unitary representations, which naturally occur as follows. Given a locally compact,~$\sigma$-compact group~$G$ acting continuously on the right of a locally compact space~$X$ equipped with a Borel measure~$\mu$ invariant under some closed subgroup~$H\leq G$, we obtain a unitary representation of~$H$ on the space~$L_{\mu}^{2}(X)$ of functions on~$X$ defined~$\mu$-almost everywhere, whose absolute value squared is integrable with respect to~$\mu$. The unitary representation is induced by the action of~$G$ on~$X$ and the element~$h\in H$ sends the element~$f\in L_{\mu}^{2}(X)$ to the element~$h\cdot f\in L_{\mu}^{2}(X)$ which is given by~$(h\cdot f)(x)=f(xh)$ for almost all~$x\in X$. We will denote by~$g$ the map defined by the action of an element~$g\in G$ on~$X$. Note that for any~$f\in C_{c}(X)$ the function~$g\cdot f$ is again continuous with compact support. In particular, the measure~$g_{\ast}\mu$ is the measure defined by
\begin{equation*}
  \int_{X}f\der{}g_{\ast}\mu=\int_{X}g\cdot f\der{}\mu\quad(f\in C_{c}(X)).
\end{equation*}

We denote by~$\mu_{\Ucal_{1}}$ the~$U_{S}$-invariant measure on the periodic orbit~$\Ucal_{1}=\GammaS U_{S}$. We are interested in the behaviour of the push-forward~$(a_{y})_{\ast}\mu_{\Ucal_{1}}$ as~$y_{\infty}\to 0$. As discussed previously, the calculation~$a_{y}u_{t}a_{y}^{-1}=u_{y^{2}t}$ implies that~$\Ucal_{y}=\GammaS U_{S}a_{y}$ has volume~$\lvert y^{-2}\rvert_{S}$. As one would expect, the behaviour of long periodic orbits~$\Ucal_{y}$ can be deduced from the equidistribution of long periodic horocycle orbits in~$\Xr(S^{m})$,~$m\in\N_{0}^{\Sf}$.

Recall that the map sending~$\Z+t\in\rquotient{\Z}{\Zs}$ to~$\ZS+t\in\rquotient{\ZS}{\QS}$ is an isomorphism and the projection~$\rquotient{\Z}{\Zs}\to\Tr$ sending~$\Z+t$ to~$\Z+t_{\infty}$ is onto with fibers homeomorphic to~$\Zsf$. Hence~$[0,1)\times\Zsf$ is a fundamental domain for~$\ZS$ in~$\QS$. In particular, the orbit measure on~$\Ucal_{1}$ is given by
\begin{equation*}
  \int_{\XS}f(x)\der\mu_{\Ucal_{1}}(x)=\int_{\Zsf}\int_{0}^{1}f(\GammaS u_{(t_{\infty},t_{\Sf})})\der{}t_{\infty}\der{}t_{\Sf}\quad(f\in C_{c}(\XS)).
\end{equation*}
As the smooth functions form a dense subset of~$C_{c}(\XS)$, it suffices to show that
\begin{equation*}
  \lim_{\substack{y_{\infty}\to0\\ y\in\R_{>0}\times\Zsf^{\times}}}\int_{\Zsf}\int_{0}^{1}f(\GammaS(u_{t_{\infty}},u_{t_{\Sf}})a_{y})\der{}t_{\infty}\der{}t_{\Sf}=\int_{\XS}f(x)\der\nu_{S}(x)\quad\big(f\in\compactsmooth(\XS)\big).
\end{equation*}
In fact, we will prove this with a bound on the error term. Recall that the isomorphism~$\psi^{(m)}:\Xr(S^{m})\to\XS[m]$ maps~$\Gammar(S^{m})g$ to~$\GammaS\imath_{\infty}(g)K[m]$. Thus~$\Gammar(S^{m})U_{\infty}a_{y_{\infty}}$ is mapped to the set
\begin{equation*}
  \set{\GammaS\imath_{\infty}(u_{t_{\infty}}a_{y_{\infty}})K[m]}{t_{\infty}\in\R}=\GammaS\US\imath_{\infty}(a_{y_{\infty}})K[m].
\end{equation*}
The equality of the sets follows from compactness of~$\Gammar(S^{m})U_{\infty}a_{y_{\infty}}$ and density of~$\ZS$ in~$\Q_{\Sf}$. Indeed, for arbitrary~$t\in\QS$ we can find some~$n\in\ZS$ such that~$u_{(t_{\Sf}-n)}\in K[m]$. Using $u_{n}\in\GammaS$ and $\GammaS gK[m]=\GammaS g_{\infty}K[m]$ whenever $g_{\Sf}\in K[m]$, we thus obtain
\begin{align*}
  \GammaS u_{t}\imath_{\infty}(a_{y_{\infty}})K[m]&=\GammaS u_{t-n}\imath_{\infty}(a_{y_{\infty}})K[m]\\
  &=\GammaS\imath_{\infty}(u_{(t_{\infty}-n)}a_{y_{\infty}})K[m].
\end{align*}
Let now~$a\in\AS$ such that~$a_{\Sf}=\one$. In what follows, we identify~$a$ with its projection~$a_{\infty}$. Let~$\mu_{\GammaS\US a}$ be the unique~$\US$-invariant probability measure on~$\GammaS\US a$. This measure extends to a measure on~$\XS$ and thus~${\pi^{(m)}}_{\ast}\mu_{\GammaS\US a}$ is a probability measure on~$\XS[m]$ with support given by~$\GammaS\US aK[m]$. As~$\pi^{(m)}$ is~$\Gr$-equivariant, the push-forward measure is invariant under~$U_{\infty}$. Using the isomorphism~$\psi^{(m)}:\Xr(S^{m})\to\XS[m]$ and the preceding discussion, it follows that~${\pi^{(m)}}_{\ast}\mu_{\GammaS\US a}$ is actually the push-forward of the unique~$U_{\infty}$-invariant measure~$\mu_{\Gammar(S^{m})\Ur a}$ on the orbit~$\Gammar(S^{m})U_{\infty}a$ under~$\psi^{(m)}$. If now~$f\in\compactsmooth(\XS)$, then~$f=\sum_{m\in\N_{0}^{\Sf}}\pr[m]f$ is a finite sum. Denote by~$\tilde{f}_{m}\in\compactsmooth(\Xr(S^{m}))$ the unique function so that~$\pr[m](f)=\tilde{f}_{m}\circ(\psi^{(m)})^{-1}\circ\pi^{(m)}$. It follows that 
\begin{align*}
  \mu_{\GammaS\US a}(f)&=\sum_{m\in\N_{0}^{\Sf}}{\pi^{(m)}}_{\ast}\mu_{\GammaS\US a}\big(\tilde{f}_{m}\circ(\psi^{(m)})^{-1}\big)\\
  &=\sum_{m\in\N_{0}^{\Sf}}{\psi^{(m)}}_{\ast}\mu_{\Gammar(S^{m})\Ur a}\big(\tilde{f}_{m}\circ(\psi^{(m)})^{-1}\big)\\
  &=\sum_{m\in\N_{0}^{\Sf}}\mu_{\Gammar(S^{m})\Ur a}(\tilde{f}_{m}).
\end{align*}
Fix any basis of~$\liegR$. For~$D\in\N_{0}$ let~$\Scal_{D,m}$ denote the~$\Ltwo$-Sobolev norm of degree~$D$ with respect to this basis on~$\compactsmooth(\Xr(S^{m}))$, and let~$\Scal_{D}$ denote the~$\Ltwo$-Sobolev norm of degree~$D$ with respect to this basis on~$\compactsmooth(\XS)$. Given~$y\in(0,\infty)$, the formula~\eqref{eq:integralinvariantfunctions} and the above description of~$\mu_{\GammaS\US a}$ in terms of measures on closed horocycle orbits in congruence quotients combined with the effective equidistribution of long horocycles \cite{Sarnak1981} and the uniformity of the spectral gap on congruence quotients \cite{Selberg65} implies that there is some degree~$D$ (in fact one can choose~$D=1$) and some~$\exponentdecayhorocycles>0$ such that
\begin{align*}
  \lvert\mu_{\GammaS\US a_{y}}(f)-\nu_{S}(f)\rvert&\leq\sum_{m\in\N_{0}^{\Sf}}\lvert\mu_{\Gammar(S^{m})\Ur a_{y}}(\tilde{f}_{m})-m_{\Xr(S^{m})}(\tilde{f}_{m})\rvert\\
  &\ll y^{-\exponentdecayhorocycles}\sum_{m\in\N_{0}^{\Sf}}\Scal_{D,m}(\tilde{f}_{m})\ll_{\Sf}y^{-\exponentdecayhorocycles}\Scal_{D}(f),
\end{align*}
where the last bound follows from Cauchy-Schwarz and Lemma~\ref{lem:relationsobolevnorms}.
%%%%%%%%%%%% Rational Points %%%%%%%%%%%%
\section{Effective equidistribution of rational points in the torus and the modular surface}\label{sec:rationalpoints}
Let~$n\in\mathbb{N}$ and set
\begin{equation*}
  \rationalsnr=\set{\big(\tfrac{k}{n},a_{\sqrt{n}}u_{k/n}\cdot\Gammar\big)}{k\in\{0,1,\cdots,n-1\}}\subseteq\Tr\times\Xr.
\end{equation*}
It is well-known that the projection of~$\rationalsnr$ to~$\T$ equidistributes effectively and it was recently explained in \cite{numberfield} that the same holds for the projection of~$\rationalsnr$ to the modular surface. However, we will have to refine this result a bit and hence will give a complete proof of said statement here. The goal is to show joint equidistribution with a rate. To this end we will proceed as follows. Equidistribution for the modular surface can be improved to rational points along pieces of closed horocycle orbits, i.e.~to sets of the form
\begin{equation*}
  \rationalsnrXpiece=\set{a_{\sqrt{n}}u_{k/n}\cdot\Gammar}{k/n\in[\alpha,\beta]},\qquad0\leq\alpha<\beta\leq1.
\end{equation*}
Using effective equidistribution for these sets, one can show effective equidistribution of the sets~$\rationalsnr$ in the following way. Given any smooth function~$\varphi$ on the torus and a smooth function~$f$ on the modular surface, uniform continuity of~$\varphi$ implies that~$\rationalsnr$ decomposes as a disjoint union of sets of the form
\begin{equation*}
  \rationalsnrpiece=\set{(\tfrac{k}{n},a_{\sqrt{n}}u_{k/n}\cdot\Gammar)}{k/n\in[\alpha,\beta]}\subseteq\Tr\times \Xr,\qquad0\leq\alpha<\beta\leq1,
\end{equation*}
on each of which the function~$\varphi\otimes f$ is, up to some small error, constant in the first component. The choice of~$\alpha$ and~$\beta$ depends on the smoothness properties of~$\varphi$ and can be captured in the Sobolev norm of~$\varphi\otimes f$. Applying effective equidistribution of rational points in short pieces of closed horocycle orbits will then imply the statement on the product space.

In fact, we prove equidistribution of rational points for products of a congruence quotient with a torus. This is then used to prove effective equidistribution of the lift of~$\rationalsnr$ to~$S$-arithmetic extensions. For what follows, it will be useful to introduce some additional notation. Given~$m,m^{\prime}\in\Z^{\Sf}$, we denote by~$m\vee m^{\prime}\in\Z^{\Sf}$ the coordinate-wise maximum of~$m$ and~$m^{\prime}$. Furthermore, we denote by~$\Delta$ the embedding of the~$\Q$-points of a group in the~$\QS$-points of the group. Let
\begin{equation*}
  \rationalsn=\set{\big(\ZS+\Delta(\tfrac{k}{n}),\GammaS\Delta(u_{\frac{k}{n}})\big)}{k\in\Z}\subseteq\TS\times\XS
\end{equation*}
be the lift of the rational points to the~$S$-arithmetic extension.
\begin{lemma}\label{lem:projectionrationalpoints}
  Let~$n\in\N$ and let~$S$ be a finite set of places including~$\infty$ such that~$(\Sf,n)=1$. Given~$l,m\in\mathbb{N}_{0}^{\Sf}$, let
  \begin{equation*}
    \Pi_{l,m}:\rquotient{\ZS}{\QS}\times\rquotient{\GammaS}{\GS}\to\rquotient{S^{l}\Z}{\R}\times\rquotient{\Gammar(S^{m})}{\Gr}
  \end{equation*}
  denote the canonical projection. Then
  \begin{equation*}
    \Pi_{l,m}\rationalsn=\set{(S^{l}\Z+S^{l\vee m}\tfrac{k}{n},\Gammar(S^{m})u_{S^{l\vee m}\frac{k}{n}})}{k\in\Z}.
  \end{equation*}
\end{lemma}
\begin{proof}
  Let~$k\in\Z$ be fixed. Using strong approximation, we can find some~$r\in\Z$ such that for all~$p\in\Sf$ we have
  \begin{equation*}
    \lvert r-\tfrac{k}{n}\rvert_{p}<S^{-(l\vee m)}.
  \end{equation*}
  In particular,~$\frac{k}{n}=r+q$ for a~$q\in\Q$ satisfying~$\lvert q\rvert_{p}<S^{-(l\vee m)}$ for all~$p\in\Sf$. Let~$b^{\prime}=k-nr$, so that~$q=\frac{b^{\prime}}{n}$. The preceding bound combined with~$(\Sf,n)=1$ implies that~$S^{l\vee m}|b^{\prime}$. Thus there is some~$b\in\Z$ such that
  \begin{align*}
    \ZS+\Delta(\tfrac{k}{n})+S^{l}\Zsf&=\ZS+\imath_{\infty}(S^{l\vee m}\tfrac{b}{n})+S^{l}\Zsf,\\
    \GammaS\Delta(u_{\frac{k}{n}})K[m]&=\GammaS\imath_{\infty}(u_{S^{l\vee m}b/n})K[m].
  \end{align*}
  It remains to show that the map is onto this set of points. To this end let~$k\in\Z$ be given arbitrarily. Then~$\lvert S^{l\vee m}\frac{k}{n}\rvert_{p}\leq p^{-\max\{l_{p},m_{p}\}}$ for all~$p\in\Sf$ and thus
  \begin{align*}
    \ZS+\imath_{\infty}(S^{l\vee m}\tfrac{k}{n})+S^{l}\Zsf&=\ZS+\Delta(S^{l\vee m}\tfrac{k}{n})+S^{l}\Zsf,\\
    \GammaS\imath_{\infty}(u_{S^{l\vee m}\frac{k}{n}})K[m]&=\GammaS\Delta(u_{S^{l\vee m}\frac{k}{n}})K[m].
  \end{align*}
\end{proof}
\begin{cor}\label{cor:equidistributionrationalpointssolenoid}
  Let~$S$ be a finite set of places of~$\mathbb{Q}$ containing the infinite place. There is some~$\Ltwo$-Sobolev norm~$\Scal$ on~$\smooth(\TS)$, such that the following is true. Let~$\varphi\in \smooth(\TS)$ and~$n\in\mathbb{N}$ such that~$(\Sf,n)=1$. Then
  \begin{equation*}
    \bigg\lvert\frac{1}{n}\sum_{k=0}^{n-1}\varphi\big(\Delta(\tfrac{k}{n})\big)-\int_{\TS}\varphi(x)\der m_{\TS}(x)\bigg\rvert\ll\frac{\mathcal{S}(\varphi)}{n}.
  \end{equation*}
\end{cor}
\begin{proof}
  Assume that~$\varphi\in\smooth(\TS)$. For~$m\in\N_{0}^{\Sf}$ denote by~$\pi_{\TS}^{(m)}:\TS\to\rquotient{S^{m}\Z}{\R}$ the projection obtained by the isomorphism~$\quotient{\TS}{S^{m}\Zsf}\cong\rquotient{S^{m}\Z}{\R}$. We denote by~$\tilde{\varphi}_{m}\in\smooth(\rquotient{S^{m}\Z}{\R})$ the unique function such that~$\pr[m]\varphi=\tilde{\varphi}_{m}\circ\pi_{\TS}^{(m)}$. If~$E_{\pr[m]\varphi}$ denotes the integral of~$\pr[m]\varphi$ on~$\TS$, then
  \begin{equation*}
    E_{\pr[m]\varphi}=\int_{\TS}\pr[m]\varphi(x)\der m_{\TS}(x)=\int_{\rquotient{S^{m}\Z}{\R}}\tilde{\varphi}_{m}(t)\der t
  \end{equation*}
  combined with Lemma~\ref{lem:projectionrationalpoints} and the Mean Value Theorem, imply that
  \begin{align*}
    \bigg\lvert\frac{1}{n}\sum_{k=0}^{n-1}&\varphi\big(\Delta(\tfrac{k}{n})\big)-E_{\varphi}\bigg\rvert\leq\sum_{m\in\N_{0}^{\Sf}}\bigg\lvert\frac{1}{n}\sum_{k=0}^{n-1}\pr[m]\varphi\big(\Delta(\tfrac{k}{n})\big)-E_{\pr[m]\varphi}\bigg\rvert\\
    &=\sum_{m\in\N_{0}^{\Sf}}\bigg\lvert\frac{1}{n}\sum_{k=0}^{n-1}\tilde{\varphi}_{m}(S^{m}\tfrac{k}{n})-\int_{\rquotient{S^{m}\Z}{\R}}\tilde{\varphi}_{m}(t)\der t\bigg\rvert\\
    &\leq\sum_{m\in\N_{0}^{\Sf}}\frac{1}{n}\sum_{k=0}^{n-1}\frac{1}{S^{m}}\int_{-\frac{S^{m}}{2}}^{\frac{S^{m}}{2}}\lvert\tilde{\varphi}_{m}(S^{m}\tfrac{k}{n})-\tilde{\varphi}_{m}(S^{m}\tfrac{k}{n}+\tfrac{t}{n})\rvert\der t\\
    &\ll\sum_{m\in\N_{0}^{\Sf}}\frac{S^{m}}{n}\lVert\tilde{\varphi}_{m}^{\prime}\rVert_{\infty}\ll\frac{1}{n}\sum_{m\in\N_{0}^{\Sf}}S^{m}\Scal_{D,m}(\tilde{\varphi}_{m})\ll_{D,\Sf}\frac{1}{n}\Scal_{D}(\varphi),
  \end{align*}
  where in the last line, we used the Sobolev Embedding Theorem~\ref{prop:propertiessobolevnorms}~(\ref{item:sobolevembedding}) for an~$\Ltwo$-Sobolev norm~$\Scal_{D,m}$ on~$\smooth(\rquotient{S^{m}\Z}{\R})$ with respect to a basis independent of~$m$, as well as the Cauchy-Schwarz inequality and~\eqref{eq:relationsobolevnormstorus} for the~$\Ltwo$-Sobolev norm of degree~$D$ on~$\smooth(\TS)$ with respect to the same basis.
\end{proof}
In order to obtain equidistribution of the rational points in the product, we want to apply the above equidistribution in the~$S$-arithmetic extension of the torus and the corresponding result in the modular surface. As mentioned in the beginning, we partition the torus into pieces on which the function varies very little and reduce it to a problem on the~$S$-arithmetic extension of the modular surface only. To this end we need to show the equidistribution of rational points for pieces of the horocycle orbits, i.e.~of the sets
\begin{equation*}
  \set{a_{\sqrt{n}}u_{N\frac{k}{n}}\cdot\Gammar(N)}{k\in \Z\cap[\tfrac{n\alpha}{N},\tfrac{n\beta}{N}]}\quad(0\leq\alpha<\beta<N).
\end{equation*}
In what follows, given real numbers~$a<b$, we will denote~$[a,b]_{\Z}=[a,b]\cap\Z$.
\begin{prop}\label{prop:equidistributionrationalpointscongruence}
  There is a basis~$\Xfrak$ of~$\liesltwo(\R)$, some~$D\in\mathbb{N}$ and positive constants~$\exponentdecayrationalpointscongruence,\maxexponentunipotentgrowth>0$ such that the following is true. If~$N\in\N$,~$a\in\N$,~$f\in\compactsmooth\big(\rquotient{\Gammar(N)}{\Gr}\big)$, and~$0\leq\alpha<\beta<N$ satisfy~$\beta-\alpha<1$, then
  \begin{equation*}
    \bigg\lvert\frac{1}{\lvert[\tfrac{n\alpha}{a},\tfrac{n\beta}{a}]_{\Z}\rvert}\sum_{k\in[\frac{n\alpha}{a},\frac{n\beta}{a}]_{\Z}}f\big(a_{\sqrt{n}}u_{a\frac{k}{n}}\cdot\Gammar(N)\big)-\int_{\rquotient{\Gammar(N)}{\Gr}}f(x)\der x\bigg\rvert\ll\frac{a^{\maxexponentunipotentgrowth}}{\beta-\alpha}n^{-\exponentdecayrationalpointscongruence}\Scal_{D}(f),
  \end{equation*}
  where~$\Scal_{D}$ is an $\Ltwo$-Sobolev norm defined by the monomials of degree at most~$D$ in~$\Xfrak$.
\end{prop}
\begin{proof}
  Given a real-valued function~$f\in C_{c}^{\infty}(\rquotient{\Gammar(N)}{\Gr})$, define the discrepancy (for~$K$ and~$\gamma\in\Gr$) as
\begin{equation*}
  D_{K}f(x)=\frac{1}{K}\sum_{\ell=0}^{K-1}f(x\gamma^{\ell})-E_{f}\qquad\big(x\in\rquotient{\Gammar(N)}{\Gr}\big),
\end{equation*}
where~$E_{f}$ is the integral of~$f$ over~$\rquotient{\Gammar(N)}{\Gr}$. The goal is to use the spectral gap for the action of~$\gamma=u_{a}$ in combination with the right degree of averaging in the discrepancy. Given a function~$f:\rquotient{\Gammar(N)}{\Gr}\to\C$ we denote by~$A_{n,a}^{\alpha,\beta}(f)$ the average of~$f$ on the set~$[\frac{n\alpha}{a},\frac{n\beta}{a}]_{\Z}$ along the stretched horocycle orbit~$a_{\sqrt{n}}\Ur\cdot\Gammar(N)$. First note that~$A_{n,a}^{\alpha,\beta}(f-E_{f})-A_{n,a}^{\alpha,\beta}(D_{K}f)$ is the difference between the average~$A_{n,a}^{\alpha,\beta}(f-E_{f})$ and---exchanging the order of summation---an average of the moving averages~$A_{n,a}^{\alpha,\beta}(f\circ\gamma^{\ell}-E_{f})$ for~$0\leq \ell<K$ where~$f\circ\gamma^{\ell}$ is defined by precomposing~$f$ with right-multiplication by~$\gamma^{\ell}$. More concisely, the term in question is the difference between an average and an average of moving averages. Such a difference is bounded by an appropriate count of the boundary terms and the maximum norm of the underlying sequence. Hence Proposition~\ref{prop:propertiessobolevnorms}~(\ref{item:sobolevembedding}) implies the existence of an~$\Ltwo$-Sobolev norm~$\Scal_{D_{0}}$ of degree~$D_{0}$ on~$\compactsmooth(\rquotient{\Gammar(N)}{\Gr})$ such that for general~$F\in\compactsmooth(\rquotient{\Gammar(N)}{\Gr})$ we have
\begin{equation}\label{eq:averagetodiscrepancy}
  \big\lvert A_{n,a}^{\alpha,\beta}(F-E_{F})-A_{n,a}^{\alpha,\beta}(D_{K}F)\big\rvert\ll\tfrac{K}{\lvert[\frac{n\alpha}{a},\frac{n\beta}{a}]_{\Z}\rvert}\Scal_{D_{0}}(F).
\end{equation}
Let~$I=(-\frac{a}{2n^{\delta}},\frac{a}{2n^{\delta}})$ for~$\delta\in(0,1)$ to be determined later. For any such~$\delta$, the map sending~$t\in I$ and~$k\in[\frac{n\alpha}{a},\frac{n\beta}{a}]_{\Z}$ to~$u_{t}a_{\sqrt{n}}u_{a\frac{k}{n}}\cdot\Gammar(N)$ is injective. In order to see injectivity, note that for~$s\in(-1,1)$ we have
\begin{equation*}
  u_{s\frac{a}{2n^{\delta}}}a_{\sqrt{n}}u_{a\frac{k}{n}}=a_{\sqrt{n}}u_{(k+\frac{s}{2n^{\delta}})\frac{a}{n}}.
\end{equation*}
Using the Mean Value Theorem for~$f$ and Proposition~\ref{prop:propertiessobolevnorms}~(\ref{item:sobolevembedding}) there is some~$\Ltwo$-Sobolev norm of degree~$D_{1}$ so that by~\eqref{eq:averagetodiscrepancy}
\begin{align*}
  A_{n,a}^{\alpha,\beta}(f)-E_{f}&=\tfrac{n^{\delta}}{a}\int_{I}A_{n,a}^{\alpha,\beta}(u_{-t}\cdot f-E_{f})\der t+O(\tfrac{a}{n^{\delta}})\Scal_{D_{1}}(f)\\
  &=\tfrac{n^{\delta}}{a}\int_{I}A_{n,a}^{\alpha,\beta}(D_{K}u_{-t}\cdot f)\der t+O\Big(\tfrac{K}{\lvert[\frac{n\alpha}{a},\frac{n\beta}{a}]_{\Z}\rvert}\Big)\tfrac{n^{\delta}}{a}\int_{I}\Scal_{D_{0}}(u_{-t}\cdot f)\der t\\
  &\qquad+O(\tfrac{a}{n^{\delta}})\Scal_{D_{1}}(f)\\
\end{align*}
Using Proposition~\ref{prop:propertiessobolevnorms}~(\ref{item:continuityofrepresentation}),~$\Scal_{D_{0}}(u_{-t}\cdot f)\ll(1+\lvert t\rvert)^{\exponentpolynomialsobolevunipotent_{0}}\Scal_{D_{0}}(f)$ for some~$\exponentpolynomialsobolevunipotent_{0}>0$ depending on~$D_{0}$, so that
\begin{align*}
  \lvert A_{n,a}^{\alpha,\beta}(f)-E_{f}\rvert&\ll\tfrac{n^{\delta}}{a\lvert[\frac{n\alpha}{a},\frac{n\beta}{a}]_{\Z}\rvert}\sum_{k\in[\frac{n\alpha}{a},\frac{n\beta}{a}]_{\Z}}\int_{I}\lvert D_{K}f(a_{\sqrt{n}}u_{\frac{ak+t}{n}}\cdot\Gammar(N))\rvert\der t\\
  &\qquad+\tfrac{Kn^{\delta}}{a\lvert[\frac{n\alpha}{a},\frac{n\beta}{a}]_{\Z}\rvert}\Scal_{D_{0}}(f)\int_{I}(1+\lvert t\rvert)^{\exponentpolynomialsobolevunipotent_{0}}\der t+\tfrac{a}{n^{\delta}}\Scal_{D_{1}}(f)
\end{align*}
The first summand is the average of~$\lvert D_{K}(f)\rvert$ along the set~$\orbitnNpiece{E}{}=a_{\sqrt{n}}U_{E}\cdot\Gammar(N)$ where we denote~$E=\bigsqcup_{q\in\frac{a}{n}\Z\cap[\alpha,\beta]}q+\tfrac{1}{n}I$. If $E=[s,t]$ for $s<t$, we denote this by $\orbitnNpiece{s}{t}$. Given a real number~$t$, write~$t_{-}(n^{\delta})=t-\frac{a}{2n^{\delta}}$ and similarly~$t_{+}(n^{\delta})=t+\frac{a}{2n^{\delta}}$. Using Cauchy-Schwarz, we have
\begin{align*}
  \int_{\orbitnNpiece{E}{}}\lvert D_{K}f(x)\rvert\der x&=\int_{\orbitnNpiece{\alpha_{-}(n^{\delta})}{\beta_{+}(n^{\delta})}}\mathbf{1}_{\orbitnNpiece{E}{}}(x)\lvert D_{K}f(x)\rvert\der x\\
  &\leq\sqrt{\vol{\orbitnNpiece{E}{}}}\bigg(\int_{\orbitnNpiece{\alpha_{-}(n^{\delta})}{\beta_{+}(n^{\delta})}}\lvert D_{K}f(x)\rvert^{2}\der x\bigg)^{\frac{1}{2}}.
\end{align*}
Given~$0\leq\alpha^{\prime}<\beta^{\prime}\leq N$, denote by~$\mu_{N,n;\alpha^{\prime},\beta^{\prime}}$ the probability measure defined by
\begin{equation*}
  \mu_{N,n;\alpha^{\prime},\beta^{\prime}}(F)=\frac{1}{\beta^{\prime}-\alpha^{\prime}}\int_{\alpha^{\prime}}^{\beta^{\prime}}F(a_{\sqrt{n}}u_{t}\cdot\Gammar(N))\der t\quad(F\in\compactsmooth(\rquotient{\Gammar(N)}{\Gr})).
\end{equation*}
Then it follows from the preceding bound, that
\begin{align}
  \label{eq:boundbydiscrepancyonpiece}\lvert A_{n,a}^{\alpha,\beta}(f)-E_{f}\rvert&\ll\Big(\tfrac{\vol{\orbitnNpiece{\alpha_{-}(n^{\delta})}{\beta_{+}(n^{\delta})}}}{\vol{\orbitnNpiece{E}{}}}\Big)^{\frac{1}{2}}\mu_{N,n;\alpha_{-}(n^{\delta}),\beta_{+}(n^{\delta})}\big(\lvert D_{K}f\rvert^{2}\big)^{\frac{1}{2}}\\
  \notag&\qquad+\big(1+\tfrac{a}{n^{\delta}}\big)^{1+\exponentpolynomialsobolevunipotent_{0}}\tfrac{K}{\vol{\orbitnNpiece{E}{}}(1+\exponentpolynomialsobolevunipotent_{0})}\Scal_{D_{0}}(f)+\tfrac{a}{n^{\delta}}\Scal_{D_{1}}(f),
\end{align}
Using effective equidistribution of pieces of closed horocycle orbits (cf.~\cite{KleinbockMargulisPieces}), we know that
\begin{equation}\label{eq:longhorocyclediscrepancy}
  \mu_{N,n;\alpha_{-}(n^{\delta}),\beta_{+}(n^{\delta})}(F)=\int_{\rquotient{\Gammar(N)}{\Gr}}F\der m_{\rquotient{\Gammar(N)}{\Gr}}+O\big(\tfrac{n^{-\exponentdecayhorocycles}}{\sqrt{\beta-\alpha+n^{-\delta}a}}\big)\Scal_{D_{2}}(F)
\end{equation}
for all~$F\in\compactsmooth(\rquotient{\Gammar(N)}{\Gr})\oplus\C$, for some~$\exponentdecayhorocycles>0$ and some~$\Ltwo$-Sobolev norm of degree~$D_{2}$, where neither the implicit constant nor~$\exponentdecayhorocycles$ depend on~$N$ by the uniformity of the spectral gap, cf.~\cite{Selberg65}. We will apply this bound to the function~$F=(D_{K}f)^{2}$. We first try to control the error term on the right of~\eqref{eq:longhorocyclediscrepancy}. To this end we use Proposition~\ref{prop:propertiessobolevnorms}~(\ref{item:multiplicativity}) to find some~$\Ltwo$-Sobolev norm of degree~$D_{3}$, such that for all~$F,\tilde{F}\in\compactsmooth(\rquotient{\Gammar(N)}{\Gr})$ we have
\begin{equation*}
  \Scal_{D_{2}}(F\tilde{F})\ll\Scal_{D_{3}}(F)\Scal_{D_{3}}(\tilde{F}).
\end{equation*}
In order to control $\Scal_{D_{3}}(D_{K}f)$ use Proposition~\ref{prop:propertiessobolevnorms}~(\ref{item:continuityofrepresentation}) in combination with the choice~$\gamma=u_{a}$ to find some exponent~$\exponentunipotentgrowth>0$ depending only on the degree~$D_{3}$, such that
\begin{equation*}
  \Scal_{D_{3}}(D_{K}f)\ll\frac{1}{K}\sum_{k=0}^{K-1}(1+ka)^{\exponentunipotentgrowth}\Scal_{D_{3}}(f)\ll(aK)^{\exponentunipotentgrowth}\Scal_{D_{3}}(f).
\end{equation*}
Hence we have shown
\begin{equation*}
  \Scal_{D_{2}}\big((D_{K}f)^{2}\big)\ll(aK)^{2\exponentunipotentgrowth}\Scal_{D_{3}}(f)^{2}.
\end{equation*}

We now turn to bounding the~$\Ltwo$-norm of $D_{K}f$ on~$\rquotient{\Gammar(N)}{\Gr}$, i.e.~the first term in the expression resulting from~\eqref{eq:longhorocyclediscrepancy}. 
As of \cite{Selberg65} there is some~$\exponentdecayharishchandraunipotent>0$ independent of $N$ and without loss of generality less than~$\frac{1}{2}$ such that for all $k\in\N$ we have
\begin{equation*}
  \big\lvert\langle u_{k}\cdot f,f\rangle_{L^{2}(\rquotient{\Gammar(N)}{\Gr})}-\lvert E_{f}\rvert^{2}\big\rvert\ll\big(1+k\big)^{-2\exponentdecayharishchandraunipotent}\Scal_{D_{4}}(f)^{2}.
\end{equation*}
The independence of $N$ is known as uniform effective decay of matrix coefficients for the action of $\Gr$ on congruence quotients. For the explicit calculation of the Harish-Chandra spherical function, we note that the maximal singular value of the matrix $u_{t}$ is comparable to $1+\lvert t\rvert$ and refer the reader to \cite[Sect.~3.7]{HeeOh} for further details. Observe that for any sequence~$(x_{k})_{k\in\N}$ we have
\begin{equation}\label{eq:summation}
  \sum_{k_{1}=1}^{K-1}\sum_{k_{2}=0}^{k_{1}-1}x_{k_{1}-k_{2}}=\sum_{k_{1}=0}^{K-2}\sum_{k_{2}=k_{1}+1}^{K-1}x_{k_{2}-k_{1}}.
\end{equation}
Combining these two facts and using that~$f$ is real-valued we obtain
\begin{align*}
  \int_{\rquotient{\Gammar(N)}{\Gr}}(D_{K}f)^{2}&\der m_{\rquotient{\Gammar(N)}{\Gr}}\leq\frac{1}{K^{2}}\sum_{k_{1},k_{2}=0}^{K-1}\big\lvert\langle u_{(k_{1}-k_{2})a}\cdot f,f\rangle_{L^{2}(\rquotient{\Gammar(N)}{\Gr})}-E_{f}^{2}\big\rvert\\
  &\ll\frac{1}{K^{2}}\sum_{k_{1}=1}^{K-1}\sum_{k_{2}=0}^{k_{1}-1}\big\lvert\langle u_{(k_{1}-k_{2})a}\cdot f,f\rangle_{L^{2}(\rquotient{\Gammar(N)}{\Gr})}-E_{f}^{2}\big\rvert\\
  &\qquad+\frac{1}{K}\big\lvert E_{f^{2}}-E_{f}^{2}\big\rvert\\
  &\ll\frac{1}{K^{2}}\sum_{k_{1}=1}^{K-1}\sum_{k_{2}=0}^{k_{1}-1}\big(1+(k_{1}-k_{2})a\big)^{-2\exponentdecayharishchandraunipotent}\Scal_{D_{4}}(f)^{2}+K^{-1}\Scal_{D_{0}}(f)^{2}.
\end{align*}

It remains to bound the first sum. To this end one calculates
\begin{align*}
  \frac{1}{K^{2}}\sum_{k_{1}=1}^{K-1}\sum_{k_{2}=0}^{k_{1}-1}\big(1+(k_{1}-k_{2})a\big)^{-2\exponentdecayharishchandraunipotent}&\leq\frac{1}{a K^{2}}\sum_{k_{1}=1}^{K-1}\int_{1}^{1+k_{1}a}t^{-2\exponentdecayharishchandraunipotent}\der t\\
  &\leq\frac{1}{a(1-2\exponentdecayharishchandraunipotent)K^{2}}\sum_{k_{1}=1}^{K-1}(1+k_{1}a)^{1-2\exponentdecayharishchandraunipotent}\\
  &\leq\frac{1}{a^{2}(1-2\exponentdecayharishchandraunipotent)K^{2}}\int_{1+a}^{1+Ka}t^{1-2\exponentdecayharishchandraunipotent}\der t\\
  &\ll\frac{(1+Ka)^{2-2\exponentdecayharishchandraunipotent}}{(2-2\exponentdecayharishchandraunipotent)(1-2\exponentdecayharishchandraunipotent)(Ka)^{2}}\\
  &\ll (1+Ka)^{-2\exponentdecayharishchandraunipotent}.
\end{align*}
Thus, combining the two steps we obtain
\begin{equation}
  \label{eq:boundmeansquarediscrepancy}
  \int_{\rquotient{\Gammar(N)}{\Gr}}(D_{K}f)^{2}\ll (1+aK)^{-2\exponentdecayharishchandraunipotent}\Scal_{D_{4}}(f)^{2}+K^{-1}\Scal_{D_{0}}(f)^{2}
\end{equation}
for some~$\Ltwo$-Sobolev norm of degree~$D_{4}$.
As~$\lvert E\rvert\asymp\frac{\beta-\alpha}{n^{\delta}}$, we have
\begin{equation*}
  \vol{\orbitnNpiece{E}{}}\asymp\tfrac{\beta-\alpha}{n^{\delta-1}}.
\end{equation*}
We also note that 
\begin{equation*}
  \vol{\orbitnNpiece{\alpha_{-}(n^{\delta})}{\beta_{+}(n^{\delta})}}=n(\beta-\alpha+\tfrac{a}{n^{\delta}}).
\end{equation*}
Let~$D=\max\{D_{0},\ldots,D_{4}\}$, so that~$\Scal_{D_{i}}\ll\Scal_{D}$. Combining this with the bounds from~\eqref{eq:boundbydiscrepancyonpiece},~\eqref{eq:longhorocyclediscrepancy} and~\eqref{eq:boundmeansquarediscrepancy}, denoting~$\maxexponentunipotentgrowth=\max\{1+\eta_{0},\exponentunipotentgrowth\}$, and plugging in the bounds for the volumes of the pieces~$\orbitnNpiece{E}{}$ and~$\orbitnNpiece{\alpha_{-}(n^{\delta})}{\beta_{+}(n^{\delta})}$ respectively, we found that
\begin{align*}
  \lvert A_{n,a}^{\alpha,\beta}(f)-E_{f}\rvert&\ll\Scal_{D}(f)\Big\{\Big(\tfrac{n^{\delta}(\beta-\alpha+n^{-\delta}a)}{\beta-\alpha}\Big)^{\frac{1}{2}}\Big[\tfrac{1}{(1+aK)^{\exponentdecayharishchandraunipotent}}+\tfrac{n^{-\exponentdecayhorocycles/2}(aK)^{\maxexponentunipotentgrowth}}{(\beta-\alpha+n^{-\delta}{a})^{1/4}}+K^{-\frac{1}{2}}\Big]\\
  &\qquad+\tfrac{Kn^{\delta-1}}{(\beta-\alpha)(1+\exponentpolynomialsobolevunipotent_{0})}\Big(1+\tfrac{a}{n^{\delta}}\Big)^{1+\exponentpolynomialsobolevunipotent_{0}}+\tfrac{a}{n^{\delta}}\Big\}\\
  &\ll a^{\maxexponentunipotentgrowth}\Scal_{D}(f)\Big\{\Big(\tfrac{n^{\delta}(\beta-\alpha+n^{-\delta}a)}{\beta-\alpha}\Big)^{\frac{1}{2}}\Big[K^{-\exponentdecayharishchandraunipotent}+\tfrac{n^{-\exponentdecayhorocycles/2}K^{\maxexponentunipotentgrowth}}{(\beta-\alpha+n^{-\delta}{a})^{1/4}}+K^{-\frac{1}{2}}\Big]\\
  &\qquad+\tfrac{Kn^{\delta-1}}{(\beta-\alpha)(1+\exponentpolynomialsobolevunipotent_{0})}\Big(1+\tfrac{1}{n^{\delta}}\Big)^{1+\exponentpolynomialsobolevunipotent_{0}}+\tfrac{1}{n^{\delta}}\Big\}\\
\end{align*}
We can always assume that~$\delta,\exponentdecayhorocycles$ and~$\exponentdecayharishchandraunipotent$ are sufficiently small. In particular, we assume that~$0<\delta<\min\{\frac{\exponentdecayharishchandraunipotent\exponentdecayhorocycles}{\maxexponentunipotentgrowth+\exponentdecayharishchandraunipotent},1-\frac{\exponentdecayharishchandraunipotent\exponentdecayhorocycles}{2\maxexponentunipotentgrowth+\exponentdecayharishchandraunipotent}\}$. This implies, that
\begin{equation*}
  \exponentdecayrationalpointscongruence=\min\big\{\exponentdecayharishchandraunipotent\tfrac{\exponentdecayhorocycles}{2\exponentdecayharishchandraunipotent+2\maxexponentunipotentgrowth}-\tfrac{\delta}{2},\tfrac{\exponentdecayhorocycles-\delta}{2}-\maxexponentunipotentgrowth\tfrac{\exponentdecayhorocycles}{2\exponentdecayharishchandraunipotent+2\maxexponentunipotentgrowth},1-\delta-\tfrac{\exponentdecayhorocycles}{2\exponentdecayharishchandraunipotent+2\maxexponentunipotentgrowth},\delta\big\},
\end{equation*}
is a positive number. If we choose~$K\asymp n^{\frac{\exponentdecayhorocycles}{2\exponentdecayharishchandraunipotent+2\maxexponentunipotentgrowth}}$ and assume that~$\beta-\alpha+n^{-\delta}a\leq 1$, we finally obtain
\begin{equation*}
  \lvert A_{n,a}^{\alpha,\beta}(f)-E_{f}\rvert\ll\tfrac{a^{\maxexponentunipotentgrowth}}{\beta-\alpha}n^{-\exponentdecayrationalpointscongruence}\Scal_{D}(f).
\end{equation*}
\end{proof}
For~$\XS$, the equidistribution of rational points on long horocycles follows from Proposition~\ref{prop:equidistributionrationalpointscongruence} by combining the relation between rational points on long horocycles in~$\XS$ and rational points on long horocycles in congruence quotients~$\Xr(S^{m})$,~$m\in\N_{0}^{\Sf}$ as explicated in the proof of Corollary~\ref{cor:equidistributionrationalpointssolenoid}. We will not use this later and thus leave this case to the reader.

As a corollary, we can now show effective equidistribution of the rational points in the product of a torus and a congruence quotient. More generally, we have the following
\begin{cor}\label{cor:rationalpointscongruencejoining}
  There are~$\exponentpolynomialsobolevunipotent_{1},\exponentdecayrationalpointscongruencejoining>0$, fixed bases~$\Xfrak_{1},\Xfrak_{2}$ of~$\Lie(\R)$ and~$\liesltwo$ respectively, and a degree~$D\in\N$ such that for all~$N_{1},N_{2},a,b\in\N$, for all~$\varphi\in\smooth(\rquotient{N_{1}\Z}{\R})$, and~$f\in\compactsmooth(\rquotient{\Gammar(N_{2})}{\Gr})$, for all~$n\in\N$, we have
  \begin{equation*}
    \bigg\lvert\frac{1}{nN_{1}N_{2}}\sum_{k=0}^{nN_{1}N_{2}-1}\varphi(N_{1}\Z+\tfrac{ak}{n})f(\Gammar(N_{2})u_{\tfrac{bk}{n}}a_{\sqrt{n}}^{-1})-E_{\varphi}E_{f}\bigg\rvert\ll (ab)^{\exponentpolynomialsobolevunipotent_{1}}n^{-\exponentdecayrationalpointscongruencejoining}\Scal_{1}(\varphi)\Scal_{2}(f),
  \end{equation*}
  where~$\Scal_{1},\Scal_{2}$ denote the~$\Ltwo$-Sobolev norm of degree~$D$ with respect to~$\Xfrak_{1}$ and~$\Xfrak_{2}$ on~$\smooth(\rquotient{N_{1}\Z}{\R})$ and~$\compactsmooth(\Xr(N_{2}))$ respectively.
\end{cor}
\begin{proof}
  We can assume that~$f\in\compactsmooth(\rquotient{\Gammar(N_{2})}{\Gr})\oplus\C\mathbf{1}_{\rquotient{\Gammar(N_{2})}{\Gr}}$ and in particular we assume~$E_{\varphi}=E_{f}=0$. Using Proposition~\ref{prop:equidistributionrationalpointscongruence}, we can assume without loss of generality that~$\varphi$ is non-constant. Let~$\delta\in(0,1)$ arbitrary. By the Mean Value Theorem, we have
  \begin{equation*}
    d(s,t)<\delta\implies\lvert\varphi(s)-\varphi(t)\rvert<\delta\lVert\varphi^{\prime}\rVert_{\infty}\quad\big(s,t\in\rquotient{N_{1}\Z}{\R}\big).
  \end{equation*}
  Set~$K_{\delta}=\lfloor\frac{N_{1}}{\delta}\rfloor$ and write
  \begin{equation*}
    \rquotient{N_{1}\Z}{\R}=\underbrace{\big[K_{\delta}\delta,N_{1}\big)}_{=P_{K_{\delta}}}\sqcup\bigsqcup_{l=0}^{K_{\delta}-1}\underbrace{\big[l\delta,(l+1)\delta\big)}_{=P_{l}},
  \end{equation*}
  and note that by \ref{prop:propertiessobolevnorms}~(\ref{item:sobolevembedding}), for all~$t_{l}\in P_{l}$, if~$t\in P_{l}$, then~$\varphi(t)=\varphi(t_{l})+O(\delta)\Scal_{D_{\T}}(\varphi)$ where~$\Scal_{D_{\T}}$ is some~$\Ltwo$-Sobolev norm on~$\smooth(\rquotient{N_{1}\Z}{\R})$ with its degree denoted~$D_{\T}\in\N$. For the remainder the points~$\{t_{l};l=0,\ldots,K_{\delta}\}\subseteq\rquotient{N_{1}\Z}{\R}$ are chosen so that~$t_{l}\in P_{l}$ and we denote~$z_{l}=\varphi(t_{l})$. We will write
  \begin{equation*}
    B_{n}(\varphi,f)=\frac{1}{nN_{1}N_{2}}\sum_{k=0}^{nN_{1}N_{2}-1}\varphi(N_{1}\Z+\tfrac{ak}{n})f(\Gammar(N_{2})u_{\frac{bk}{n}}a_{\sqrt{n}}^{-1}).
  \end{equation*}
  Note that the sum might contain some multiplicity which we will have to take into account. In total, the interval~$[0,N_{1})$ contains~$\frac{nN_{1}}{(a,nN_{1})}$-many points of the form~$\frac{ak}{n}\mod N_{1}$ with~$0\leq k<nN_{1}N_{2}$. To this end consider the map~$\quotient{\Z}{nN_{1}\Z}\to\quotient{\Z}{nN_{1}\Z}$ given by~$k\mapsto ak$. Denote by~$L$ the lowest common multiple of~$a$ and~$nN_{1}$. The kernel of the map then is a cyclic subgroup generated by~$\frac{L}{a}$ and in particular has cardinality~$\frac{nN_{1}}{L/a}=(a,nN_{1})$, or alternatively, the map is~$(a,nN_{1})$-to-one. Thus the image has cardinality~$\frac{nN_{1}}{(a,nN_{1})}$ and as we let~$k$ run through a full set of representatives of~$\quotient{\Z}{nN_{1}\Z}$, the claim follows. Using this, we can rewrite
  \begin{equation*}
    B_{n}(\varphi,f)=\frac{1}{N_{2}}\sum_{r=0}^{N_{2}-1}\frac{(a,nN_{1})}{nN_{1}}\sum_{l=0}^{K_{\delta}}\sum_{k\in \frac{n}{a}P_{l}\cap\Z}\varphi(N_{1}\Z+\tfrac{ak}{n})f(\Gammar(N_{2})u_{b\frac{rnN_{1}+k}{n}}a_{\sqrt{n}}^{-1}).
  \end{equation*}
  Given~$0\leq r<N_{2}$, let
  \begin{equation*}
    B_{n,r}=\frac{(a,nN_{1})}{nN_{1}}\sum_{l=0}^{K_{\delta}}\sum_{k\in \frac{n}{a}P_{l}\cap\Z}\varphi(N_{1}\Z+\tfrac{ak}{n})f(\Gammar(N_{2})u_{b\frac{rnN_{1}+k}{n}}a_{\sqrt{n}}^{-1}).
  \end{equation*}
  Using the notation introduced and applying Proposition~\ref{prop:equidistributionrationalpointscongruence}, assuming~$\delta<\frac{a}{b}$, we have
  \begin{equation*}
    \Big\lvert\frac{1}{\lvert\frac{n}{a}P_{l}\cap\Z\rvert}\sum_{k\in\frac{n}{a}P_{l}\cap\Z}f(\Gammar(N_{2})u_{b\frac{rnN_{1}+k}{n}}a_{\sqrt{n}}^{-1})\Big\rvert\ll\delta^{-1}ab^{\maxexponentunipotentgrowth-1}n^{-\exponentdecayrationalpointscongruence}\Scal_{D}(f)
  \end{equation*}
  for some degree-$D$~$\Ltwo$-Sobolev norm~$\Scal_{D}$ on~$\compactsmooth(\rquotient{\Gammar(N_{2})}{\Gr})$, where we assume without loss of generality that~$d$ was chosen so that the Sobolev Embedding Theorem \ref{prop:propertiessobolevnorms}~(\ref{item:sobolevembedding}) applies. We also used that the $\Ur$-orbit of $\Gammar(N_{2})$ identifies with $\rquotient{\R}{N_{2}\Z}$, which implies that the bound is valid independent of the value of~$r$. Next we note that~$\lvert\frac{n}{a}P_{l}\cap\Z\rvert\asymp\delta\frac{n}{a}$ and thus again denoting by~$L$ the lowest common multiple of~$a$ and~$nN_{1}$, we get~$\frac{\lvert\frac{n}{a}P_{l}\cap\Z\rvert}{nN_{1}/(a,nN_{1})}\asymp\delta\frac{n}{L}$. Note next that~$N_{1}\asymp\delta K_{\delta}$ and thus~$K_{\delta}\delta\frac{n}{L}\leq 1$. Hence combining all these, we find
  \begin{align*}
    \lvert B_{n,r}\rvert&\ll\Big\lvert\sum_{l=0}^{K_{\delta}-1}z_{l}\tfrac{\lvert\frac{n}{a}P_{l}\cap\Z\rvert}{nN_{1}/(a,nN_{1})}\tfrac{1}{\lvert\frac{n}{a}P_{l}\cap\Z\rvert}\sum_{k\in\frac{n}{a}P_{l}\cap\Z}f(\Gammar(N_{2})u_{b\frac{rnN_{1}+k}{n}}a_{\sqrt{n}}^{-1})\Big\rvert\\
    &\qquad+\delta\Scal_{D_{\T}}(\varphi)\Big\lvert\sum_{l=0}^{K_{\delta}-1}\tfrac{\lvert\frac{n}{a}P_{l}\cap\Z\rvert}{nN_{1}/(a,nN_{1})}\tfrac{1}{\lvert\frac{n}{a}P_{l}\cap\Z\rvert}\sum_{k\in\frac{n}{a}P_{l}\cap\Z}f(\Gammar(N_{2})u_{b\frac{rnN_{1}+k}{n}}a_{\sqrt{n}}^{-1})\Big\rvert\\
    &\qquad+\tfrac{\lvert\frac{n}{a}P_{K_{\delta}}\cap\Z\rvert}{nN_{1}/(a,nN_{1})}\lVert\varphi\rVert_{\infty}\lVert f\rVert_{\infty}\\
    &\ll\delta^{-1}ab^{\maxexponentunipotentgrowth-1}n^{-\exponentdecayrationalpointscongruence}\lVert\varphi\rVert_{\infty}K_{\delta}\delta\tfrac{n}{L}+\delta\Scal_{D_{\T}}(\varphi)\lVert f\rVert_{\infty}K_{\delta}\delta\tfrac{n}{L}+\delta\tfrac{n}{L}\lVert\varphi\rVert_{\infty}\lVert f\rVert_{\infty}\\
    &\ll\big(\delta^{-1}ab^{\maxexponentunipotentgrowth-1}n^{-\exponentdecayrationalpointscongruence}+2\delta\big)\Scal_{D^{\prime}}(\varphi)\Scal_{D}(f),
  \end{align*}
  where~$D_{\T}\leq D^{\prime}$ was chosen so that Proposition~\ref{prop:propertiessobolevnormstorus}~(\ref{item:sobolevembeddingtorus}) applies. Choose~$\exponentdecayrationalpointscongruencejoining=\frac{\exponentdecayrationalpointscongruence}{2}$ and~$\exponentpolynomialsobolevunipotent_{1}=\max\{1,\maxexponentunipotentgrowth\}$. If~$n^{-\exponentdecayrationalpointscongruencejoining}<\frac{a}{b}$, then we can choose~$\delta=n^{-\exponentdecayrationalpointscongruencejoining}$ and obtain
  \begin{equation*}
    \lvert B_{n}(\varphi,f)\rvert\leq\frac{1}{N_{2}}\sum_{r=0}^{N_{2}-1}\lvert B_{n,r}\rvert\ll (ab)^{\exponentpolynomialsobolevunipotent_{1}}n^{-\exponentdecayrationalpointscongruencejoining}\Scal_{D^{\prime}}(\varphi)\Scal_{D}(f).
  \end{equation*}
  Otherwise, we have~$(ab)^{\exponentpolynomialsobolevunipotent_{1}}n^{-\exponentdecayrationalpointscongruencejoining}\geq\tfrac{a^{2}b}{b}\geq 1$ and thus for these~$n$ the inequality holds with implicit constant equal to two. This proves the Corollary.
\end{proof}
We can now prove an effective equidistribution statement for rational points of a certain denominator along expanding closed horospheres.
\begin{cor}\label{cor:equidistributionrationalpoints}
  Let~$F\in\tensorcompactsmooth(\TS\times\XS)$,~$a,b\in\N$. Then
  \begin{equation*}
    \bigg\lvert\frac{1}{n}\sum_{k=0}^{n-1}F\big(\ZS+\Delta(\tfrac{ak}{n}),\GammaS\Delta(u_{bk/n})a_{\sqrt{n}}^{-1}\big)-\int_{\TS\times\XS}F\bigg\rvert\leq C(ab)^{\exponentpolynomialsobolevunipotent_{1}}n^{-\exponentdecayrationalpointsS}\Scal(F),
  \end{equation*}
  for positive constants~$\exponentdecayrationalpointsS,C,\exponentpolynomialsobolevunipotent_{1}$ which are independent of~$n,S$ and~$F$, and some~$\Ltwo$-Sobolev norm~$\Scal$ on~$\tensorcompactsmooth(\TS\times\XS)$ that does not depend on~$F$ or~$n$.
\end{cor}
\begin{proof}
  By the triangle inequality it suffices to prove that for any pure tensor~$F=\varphi\otimes f$ we have
  \begin{equation*}
    \bigg\lvert\frac{1}{n}\sum_{k=0}^{n-1}\varphi\big(\ZS+\Delta(\tfrac{ak}{n})\big)f\big(\GammaS\Delta(u_{bk/n})a_{\sqrt{n}}^{-1}\big)-\int_{\TS}\varphi\int_{\XS}f\bigg\rvert\leq C(ab)^{\exponentpolynomialsobolevunipotent_{1}}n^{-\exponentdecayrationalpointsS}\Scal_{\TS}(\varphi)\Scal_{\XS}(f),
  \end{equation*}
  where~$C$ and~$\exponentdecayrationalpointsS$ do not depend on~$n,S$ and~$F$. But this is a by now immediate consequence of Lemma~\ref{lem:projectionrationalpoints}, Corollary~\ref{cor:rationalpointscongruencejoining}, Lemma~\ref{lem:relationsobolevnorms}, Equation~\eqref{eq:relationsobolevnormstorus} and the Cauchy-Schwarz inequality. Note that we have to apply Corollary~\ref{cor:rationalpointscongruencejoining} to the pure level components with multiplicative parameters of the form~$S^{l}a$ and~$S^{m}b$ for varying~$l,m\in\N_{0}^{\Sf}$.
\end{proof}
%%%%%%%%%%%% Primitive Points %%%%%%%%%%%%
\section{Effective equidistribution of degree-$d$ residues}\label{sec:primitivepoints}
In this section, we will prove Theorem~\ref{thm:mainthm}. Given~$g\in\GS$ and a point~$(t,x)\in\TS\times\XS$, we write~$(t,x)g=(t,xg)$. Let~$S$ be a finite set of places of~$\Q$ such that~$\infty\in S$ and~$(S_{f},n)=1$ and define
\begin{equation*}
  \primitiverationalsnproduct=\set{\big(\ZS+\Delta(\tfrac{k}{n}),\GammaS\Delta(u_{k/n})\big)}{(k,n)=1}\subseteq\TS\times\XS.
\end{equation*}
The set~$\primitiverationalsnproduct$ is an invariant subset of~$\TS\times\XS$ for the action~$\Z^{\Sf}\curvearrowright\TS\times\XS$ given by~$M_{m}(t,x)=(S^{2m}t,x)a_{S^{-m}},$ where~$m\in\Z^{\Sf}$,~$t\in\QS$ and~$x\in\rquotient{\GammaS}{\GS}$. Indeed,~$S^{2m}$ is a unit mod~$n$ and hence the equality~$\GammaS u_{\frac{k}{n}}a_{S^{-m}}=\GammaS a_{S^{-m}}u_{S^{2m}\frac{k}{n}}$ together with~$a_{S^{-m}}\in\GammaS$ implies the claim. For every~$a\in\AS$ the set~$\primitiverationalsnproduct a$ is also invariant under~$\Z^{\Sf}$, as~$\AS$ is abelian.

\subsection{Effective mixing for the~$\times q$ map}
The proof of Theorem~\ref{thm:mainthm} will exploit effective mixing of the action~$M$ on the torus component, which we want to discuss in the beginning. In fact, the proof of the desired result works quite a bit more generally, i.e.~we will prove effective mixing for a class of toral endomorphisms on the~$S$-arithmetic extension.

For~$R\in(0,\infty)^{n}$, denote~$\lvert R\rvert=\prod_{i=1}^{n}R_{i}$ and~$C_{R}=\prod_{i=1}^{n}[0,R_{i})\subseteq\R^{n}$. We denote by~$\T(R)$ the torus~$\T(R)=\prod_{i=1}^{n}\T(R_{i})$. The Pontryagin dual~$\widehat{\T(R)}\cong\Z^{n}$ is given by the family of functions
\begin{equation*}
  \chi_{\mathbf{n},R}:\T(R)\to\S^{1},\quad x\mapsto\exp(2\pi\ii\textstyle{\sum_{i=1}^{n}}\tfrac{x_{i}n_{i}}{R_{i}}).
\end{equation*}
For what follows, we will use the following notation: Given~$v\in\R^{n}$ and~$R\in(0,\infty)^{n}$, we denote by~$v/R\in\R^{n}$ the vector obtained by componentwise division of the entries of~$v$ by the corresponding entries of~$R$. For any smooth function~$f:\T(R)\to\C$, we have the Fourier series expansion
\begin{equation*}
  f=\sum_{\mathbf{n}\in\Z^{n}}\alpha_{\mathbf{n}}(f)\chi_{\mathbf{n},R},\quad\text{where }\alpha_{\mathbf{n}}(f)=\frac{1}{\lvert R\rvert}\int_{C_{R}}f(t)\overline{\chi_{\mathbf{n},R}(t)}\der t,
\end{equation*}
and the convergence for the series holds both for the uniform topology and the topology defined by the~$\Ltwo$-norm. The~$\Ltwo$-norm of~$f\in\continuous(\T(R))$ is then given by the norm of the sequence of Fourier coefficients~$(\alpha_{\mathbf{n}}(f))_{\mathbf{n}\in\Z^{n}}$ in~$l^{2}(\Z^{n})$, i.e.
\begin{equation*}
  \lVert f\rVert_{2}^{2}=\sum_{\mathbf{n}\in\Z^{n}}\lvert\alpha_{\mathbf{n}}(f)\rvert^{2}.
\end{equation*}
Note that~$\lvert R\rvert a_{0}(f)=\int_{\T(R)}f$. Differentiability is characterized by the rate of convergence of the Fourier series: For a continuous function~$f$ on~$\T(R)$ to be~$k$ times continuously differentiable implies~$\sum_{\mathbf{n}\in\Z^{n}}\lvert \alpha_{\mathbf{n}}(f)\rvert^{2}\lVert \mathbf{n}/R\rVert^{2k}<\infty$. Combining these, any degree-$D$~$\Ltwo$-Sobolev norm~$\Scal$ on~$C^{\infty}(\T(R))$ hence satisfies
\begin{equation*}
  \Scal(f)^{2}\asymp\sum_{\mathbf{n}\in\Z^{n}}\lvert \alpha_{\mathbf{n}}(f)\rvert^{2}(1+\lVert\mathbf{n}/R\rVert)^{2D}.
\end{equation*}
Let $N\in\N^{n}$ arbitrary. An \emph{expanding endomorphism} of~$\T(N)$ is a map~$T_{A}:\T(N)\to\T(N)$ defined by multiplication with a matrix $A\in M_{n}(\Z)\cap\GL_{n}(\Q)$ which is diagonalizable over~$\C$ and whose eigenvalues are all larger than~$1$ in absolute value. We can now prove the following
\begin{prop}\label{prop:effectivemixingendomorphismstorus}
  Let~$N\in\N^{n}$ and let~$T_{A}:\T(N)\to\T(N)$ be an expanding endomorphism. There exists some~$\varrho>0$ independent of~$N$ such that the following is true. If~$D\in\N$ and~$f,g\in\smooth(\T(N))$. Then
  \begin{equation*}
    \lvert\langle f\circ T_{A},g\rangle_{\Ltwo(\T(N))}-\alpha_{0}(f)\overline{\alpha_{0}(g)}\rvert\ll \lvert N\rvert^{D}e^{-\varrho D}\Scal(f)\Scal(g),
  \end{equation*}
  where~$\Scal$ is the~$\Ltwo$-Sobolev norm on~$\smooth(\T(R))$ defined by
  \begin{equation*}
    \Scal(f)^{2}=\sum_{\mathbf{n}\in\Z^{n}}\lvert \alpha_{\mathbf{n}}(f)\rvert^{2}(1+\lVert\mathbf{n}/N\rVert)^{2D}.
  \end{equation*}
  The implicit constant depends only on~$A$.
\end{prop}
\begin{proof}
  Using Fourier series, the orthogonality relations for unitary characters and Cauchy-Schwarz this becomes a relatively simple calculation:
  \begin{align*}
    \lvert\langle f\circ T_{A},g\rangle_{\Ltwo(\T(N))}-\alpha_{0}(f)\overline{\alpha_{0}(g)}\rvert&=\bigg\lvert\sum_{\mathbf{n}\in\Z^{n}\setminus\{0\}}\alpha_{\mathbf{n}}(f)\cl{\alpha_{{}^{t}\!A\mathbf{n}}(g)}\bigg\rvert\\
    &\leq\lVert f\rVert_{2}\bigg(\sum_{\mathbf{n}\in\Z^{n}\setminus\{0\}}\lvert \alpha_{{}^{t}\!A\mathbf{n}}(g)\rvert^{2}\bigg)^{\frac{1}{2}}\\
    &\leq\lVert f\rVert_{2}\Scal(g)\sup_{\mathbf{n}\in\Z^{n}\setminus\{0\}}\frac{1}{\lVert{}^{t}\!A\mathbf{n}/N\rVert^{D}}.
  \end{align*}
  If~$A$ is diagonalizable over~$\C$, then so is~${}^{t}A$ and thus fix an eigenbasis~$\mathcal{B}=(v_{i})_{i=1}^{n}$ of~$\C^{n}$. Denote by~$\lambda_{i}\in\C$ the eigenvalue corresponding to~$v_{i}$ and define a norm
  \begin{equation*}
    \lVert\alpha_{1}v_{1}+\cdots+\alpha_{n}v_{n}\rVert_{\mathcal{B}}=\max\set{\lvert\alpha_{i}\rvert}{i=1,\ldots,n}.
  \end{equation*}
  Then~$\lVert {}^{t}\!Av\rVert_{\mathcal{B}}\geq\lVert v\rVert_{\mathcal{B}}\min\set{\lvert\lambda_{i}\rvert}{i=1,\ldots,n}$. Hence setting
  \begin{equation*}
    \varrho=\log\min\set{\lvert\lambda_{i}\rvert}{i=1,\ldots,n},
  \end{equation*}
  the claim follows from equivalence of norms on finite dimensional vector spaces.
\end{proof}
\begin{remark}\label{rem:effectivemixingendomorphismstorusindependence}
  The implicit constant in Proposition~\ref{prop:effectivemixingendomorphismstorus} depends only on the choice of the norm~$\lVert\cdot\rVert_{\mathcal{B}}$, i.e.~the choice of an eigenbasis for the matrix~${}^{t}\!A$. Given a commuting family of diagonalizable matrices~$\{A_{i};i\in I\}$ as in the proposition, the implicit constant can hence be chosen uniformly for this family.
\end{remark}
By Proposition~\ref{prop:effectivemixingendomorphismstorus} we obtain effective mixing of expanding toral endomorphisms on the~$S$-arithmetic extension of the torus. First, we define an extension~$T_{A}:\rquotient{\Z^{n}}{\Zs^{n}}\to\rquotient{\Z^{n}}{\Zs^{n}}$ by~$\Z^{n}+x\mapsto\Z^{n}+Ax$. We note that the isomorphism~$\rquotient{\ZS^{n}}{\QS^{n}}\cong\rquotient{\Z^{n}}{\Zs^{n}}$ is~$T_{A}$ equivariant. If~$\ell^{(1)},\ldots,\ell^{(n)}\in\N_{0}^{\Sf}$ are arbitrary, then the same is true for the isomorphism~$\biquotient{\Z^{n}}{\Zs^{n}}{(\Zsf[\ell^{(1)}]\times\cdots\times\Zsf[\ell^{(n)}])}\cong\T(S^{\ell^{(1)}},\ldots,S^{\ell^{(n)}})$. For what follows, we abuse notation as follows. Given a matrix~$\ell\in\N_{0}^{\Sf\times n}$ and denoting by~$\ell^{(1)},\ldots,\ell^{(n)}\in\N_{0}^{\Sf}$ the columns of~$\ell$, we denote
\begin{equation*}
  \Zsf[\ell]=\Zsf[\ell^{(1)}]\times\cdots\times\Zsf[\ell^{(n)}].
\end{equation*}
Similarly, we denote by~$S^{\ell}\in\R^{n}$ the vector~$(S^{\ell^{(1)}},\ldots,S^{\ell^{(n)}})$.
\begin{cor}\label{cor:effectivemixingsolenoidtoral}
  Let~$A$ be an expanding toral endomorphism,~$S$ be a finite set of places of~$\Q$. For every degree-$2D$ $\Ltwo$-Sobolev norm~$\Scal$ on~$\smooth(\TS^{n})$ and for all~$f,g\in\smooth(\TS^{n})$ we have
  \begin{equation*}
    \bigg\lvert\langle f\circ T_{A},g\rangle_{\Ltwo(\TS^{n})}-\int_{\TS^{n}}f\int_{\TS^{n}}\overline{g}\bigg\rvert\ll_{\Scal}e^{-\varrho D}\Scal(f)\Scal(g),
  \end{equation*}
  with~$\varrho>0$ as in Proposition~\ref{prop:effectivemixingendomorphismstorus}.
\end{cor}
\begin{proof}
  Assume that~$f,g\in\smooth(\TS^{n})$ are invariant under~$\Zsf[\ell]$ and~$\Zsf[\ell^{\prime}]$,~$\ell,\ell^{\prime}\in\N_{0}^{\Sf\times n}$, respectively. Recall that $\ell\vee \ell^{\prime}\in\N_{0}^{\Sf\times n}$ denotes the coordinate-wise maximum of $\ell$ and $\ell^{\prime}$, so that $\Z_{\Sf}[\ell\vee \ell^{\prime}]\leq \Z_{\Sf}[\ell]$ and $\Z_{\Sf}[\ell\vee \ell^{\prime}]\leq \Z_{\Sf}[\ell^{\prime}]$. In particular, both $f$ and $g$ are invariant under~$\Zsf[\ell\vee \ell^{\prime}]$. Let~$\tilde{f}_{\ell\vee \ell^{\prime}},\tilde{g}_{\ell\vee \ell^{\prime}}\in\smooth(\T(S^{\ell\vee \ell^{\prime}}))$ such that~$f=\tilde{f}_{\ell\vee \ell^{\prime}}\circ\pi_{\TS^{n}}^{(\ell\vee \ell^{\prime})}$ and~$g=\tilde{g}_{\ell\vee \ell^{\prime}}\circ\pi_{\TS^{n}}^{(\ell\vee \ell^{\prime})}$. As~$\pi_{\TS^{n}}^{(\ell\vee \ell^{\prime})}\circ T_{A}=T_{A}\circ\pi_{\TS^{n}}^{(\ell\vee \ell^{\prime})}$, it follows that
  \begin{equation*}
    \langle f\circ T_{A},g\rangle_{\Ltwo(\TS^{n})}=\langle\tilde{f}_{\ell\vee \ell^{\prime}}\circ T_{A},\tilde{g}_{\ell\vee \ell^{\prime}}\rangle_{\Ltwo(\T(S^{\ell\vee \ell^{\prime}}))}.
  \end{equation*}
  Using Proposition~\ref{prop:effectivemixingendomorphismstorus}, we know
  \begin{equation*}
    \bigg\lvert\langle f\circ T_{A},g\rangle_{\Ltwo(\TS^{n})}-\int_{\TS^{n}}f\int_{\TS^{n}}\overline{g}\bigg\rvert\ll \lvert S^{\ell\vee \ell^{\prime}}\rvert^{D}e^{-\varrho D}\Scal(\tilde{f}_{\ell\vee \ell^{\prime}})\Scal(\tilde{g}_{\ell\vee \ell^{\prime}}).
  \end{equation*}
  If~$q_{\ell}^{\ell\vee \ell^{\prime}}:\T(S^{\ell\vee\ell^{\prime}})\to\T(S^{\ell})$ denotes the quotient map and~$\tilde{f}_{\ell}\in\smooth(\T(S^{\ell}))$ is the unique function satisfying~$f=\tilde{f}_{\ell}\circ\pi_{\TS^{n}}^{(\ell)}$, then~$\tilde{f}_{\ell\vee\ell^{\prime}}=\tilde{f}_{\ell}\circ q_{\ell}^{\ell\vee \ell^{\prime}}$ and similarly to the proof of Lemma~\ref{lem:relationsobolevnorms} we have~$X(\tilde{f}_{\ell\vee\ell^{\prime}})=X(\tilde{f}_{\ell})\circ q_{\ell}^{\ell\vee \ell^{\prime}}$ for any differential operator~$X$ in the Lie algebra of~$\R^{n}$. If~$\lambda_{\ell}$ and~$\lambda_{\ell\vee\ell^{\prime}}$ denote the Haar probability measures on~$\T(S^{\ell})$ and~$\T(S^{\ell\vee\ell^{\prime}})$ respectively, then~$\lambda_{\ell}=(q_{\ell}^{\ell\vee \ell^{\prime}})_{\ast}\lambda_{\ell\vee \ell^{\prime}}$. Combining these, we find~$\Scal(\tilde{f}_{\ell\vee \ell^{\prime}})=\Scal(\tilde{f}_{\ell})$ and similarly for~$g$. Clearly $\lvert S^{\ell\vee\ell^{\prime}}\rvert^{D}\leq\lvert S^{\ell+\ell^{\prime}}\rvert^{D}$. Hence
  \begin{equation}\label{eq:mixingfirststep}
    \bigg\lvert\langle f\circ T_{A},g\rangle_{\Ltwo(\TS^{n})}-\int_{\TS^{n}}f\int_{\TS^{n}}\overline{g}\bigg\rvert\ll\lvert S^{\ell+\ell^{\prime}}\rvert^{D}e^{-\varrho D}\Scal(\tilde{f}_{\ell})\Scal(\tilde{g}_{\ell^{\prime}}).
  \end{equation}

  Using this, the proof now works just like the effective equidistribution of periodic horocycle orbits discussed in Section~\ref{sec:longhorocycles}. Given~$\ell\in\N^{\Sf\times n}$, denote by~$\tilde{f}_{\ell},\tilde{g}_{\ell}\in\smooth(\T(S^{\ell}))$ the unique functions so that~$\pr[\ell]f=\tilde{f}_{\ell}\circ\pi_{\TS^{n}}^{(\ell)}$ and~$\pr[\ell]g=\tilde{g}_{\ell}\circ\pi_{\TS^{n}}^{(\ell)}$. Here,~$\pr[\ell]$ denotes componentwise application of the operator~$\pr[\ell^{(i)}]$,~$i=1,\ldots,n$. Let~$\Scal_{D,\ell}$ denote the~$\Ltwo$-Sobolev norm of degree~$D$ on~$\smooth(\T(S^{\ell}))$ and~$\Scal_{2D}$ the~$\Ltwo$-Sobolev norm of degree~$2D$ on~$\smooth(\TS^{n})$. Then what we just showed combined with the Cauchy-Schwarz inequality and Equation~\eqref{eq:relationsobolevnormstorus} implies
  \begin{align*}
    \bigg\lvert\langle f\circ T_{A},g\rangle_{\Ltwo(\TS^{n})}&-\int_{\TS^{n}}f\int_{\TS^{n}}\overline{g}\bigg\rvert\\
    &\leq\sum_{\ell,\ell^{\prime}\in\N^{\Sf\times n}}\bigg\lvert\langle\pr[\ell]f\circ T_{A},\pr[\ell^{\prime}]g\rangle_{\Ltwo(\TS^{n})}-\int_{\TS^{n}}\pr[\ell]f\int_{\TS^{n}}\overline{\pr[\ell^{\prime}]g}\bigg\rvert\\
    &\ll e^{-\varrho D}\sum_{\ell,\ell^{\prime}\in\N^{\Sf\times n}}\lvert S^{\ell+\ell^{\prime}}\rvert^{D}\Scal_{D,\ell}(\tilde{f}_{\ell})\Scal_{D,\ell}(\tilde{g}_{\ell^{\prime}})\leq e^{-\varrho D}\Scal_{2D}(f)\Scal_{2D}(g).
  \end{align*}
\end{proof}
\begin{cor}\label{cor:effectivemixingsolenoid}
  Let~$q\in\N\setminus\{1\}$. Then the~$\times q$-map~$T_{q}$ on~$\TS$ is exponentially mixing at arbitrary rate, i.e.~given~$D\in\N$, there is some~$\Ltwo$-Sobolev norm~$\Scal$ depending only on~$D$, such that for all~$f,g\in\smooth(\TS)$, we have
  \begin{equation*}
    \bigg\lvert\langle f\circ T_{q},g\rangle_{\Ltwo(\TS)}-\int_{\TS}f\int_{\TS}\cl{g}\bigg\rvert\ll q^{-D}\Scal(f)\Scal(g).
  \end{equation*}
  The implicit constant is independent of~$q$.
\end{cor}
\begin{proof}
  The family of~$\times q$-maps,~$q\in\N\setminus\{1\}$, is a commuting family of expanding toral endomorphisms with smallest eigenvalue~$q$.
\end{proof}

For the remainder of the paper we will look at the~$\times S^{2m}$ map on~$\TS$ and the action of~$a_{S^{-m}}$ on~$\XS$. The latter is also mixing with a spectral gap. In fact this holds for any element which is not contained in a compact subgroup.
\begin{prop}\label{prop:spectralgap}
  There are an~$\Ltwo$-Sobolev norm~$\Scal$ of degree~$D$ on~$\compactsmooth(\XS)$ and a positive constant~$\exponentdecaytimes$ such that for all~$f_{1},f_{2}\in\compactsmooth(\XS)$ and for all~$g\in\GS$ we have
  \begin{equation*}
    \bigg\lvert\langle g\cdot f_{1},f_{2}\rangle-\int_{\XS}f_{1}\int_{\XS}\cl{f}_{2}\bigg\rvert\ll \vvvert g\vvvert_{S}^{-\exponentdecaytimes}\Scal(f_{1})\Scal(f_{2}).
  \end{equation*}
  The degree of~$\Scal$, the implicit constant and~$\exponentdecaytimes$ are independent of~$S$.
\end{prop}
\begin{remark}
  In the discussions to follow,~$g$ will be a diagonal matrix with diagonal entries of all components given by~$S^{m}$ and~$S^{-m}$ with~$m\in\Z^{\Sf}$ fixed. In this case,~$\vvvert g\vvvert_{S}^{-\exponentdecaytimes}\leq S^{-\exponentdecaytimesprime\underline{m}}$ for some~$\exponentdecaytimesprime>0$ where~$\underline{m}\in\N_{0}^{\Sf}$ denotes the vector whose entries are the absolute values of the entries of~$m$.
\end{remark}
We do not give a proof of Proposition~\ref{prop:spectralgap}, but refer the reader to \cite{HeeOh} and \cite[p.~19]{GorodnikMaucourantOh} for more details. Assuming Proposition~\ref{prop:spectralgap}, we can deduce that~$\Z^{\Sf}\curvearrowright\TS\times\XS$ as introduced in the beginning of this section is mixing.
\begin{prop}\label{prop:jointmixing}
  Let~$S$ be a finite set of primes containing~$\infty$ and~$m\in\Z^{\Sf}$. Then~$M_{m}$ is a mixing transformation on~$\TS\times\XS$. Moreover, there exists an~$\Ltwo$-Sobolev norm~$\Scal$ on~$\tensorcompactsmooth(\TS\times\XS)$ and some~$\exponentdecaytimesjoint>0$ independent of~$m$ such that for all~$F,G\in\tensorcompactsmooth(\TS\times\XS)$ we have
  \begin{equation*}
    \bigg\lvert\langle F\circ M_{m},G\rangle-\int_{\XS}F\int_{\XS}\cl{G}\bigg\rvert\ll S^{-\exponentdecaytimesjoint\underline{m}}\Scal(F)\Scal(G).
  \end{equation*}
  The constants do not depend on~$S$.
\end{prop}
\begin{proof}
  This follows immediately from Corollary~\ref{cor:effectivemixingsolenoid} and Proposition~\ref{prop:spectralgap}. Assume first that~$F=\varphi\otimes f$ and~$G=\psi\otimes g$ with~$\varphi,\psi\in\smooth(\TS)$ and~$f,g\in\compactsmooth(\XS)$. Then
  \begin{align*}
    \lvert\langle F\circ M_{m},G\rangle-E_{F}\cl{E_{G}}\rvert&\leq\big\lvert\langle\varphi\circ T_{S^{2m}},\psi\rangle-E_{\varphi}\cl{E_{\psi}}\big\rvert\lVert f\rVert_{2}\lVert g\rVert_{2}\\
    &\quad+\lvert E_{\varphi}\rvert\lvert E_{\psi}\rvert\big\lvert\langle f\circ a_{S^{m}},g\rangle-E_{f}\cl{E_{g}}\big\rvert\\
    &\ll S^{-D\underline{m}}\Scal_{D,\TS}(\varphi)\Scal_{D,\TS}(\psi)\Scal_{D,\XS}(f)\Scal_{D,\XS}(g)\\
    &\quad+\vvvert a_{S^{m}}\vvvert_{S}^{-\exponentdecaytimes}\Scal_{D,\TS}(\varphi)\Scal_{D,\TS}(\psi)\Scal_{D,\XS}(f)\Scal_{D,\XS}(g),
  \end{align*}
  where~$D$ was chosen so that Proposition~\ref{prop:propertiessobolevnorms}~(\ref{item:sobolevembedding}) holds on~$\TS$ and so that~$\Scal_{D,\XS}$ is a valid choice in Proposition~\ref{prop:spectralgap}. Using the remark following Proposition~\ref{prop:spectralgap}, we deduce the claim. For general functions in~$\tensorcompactsmooth(\TS\times\XS)$, the statement now follows from the triangle inequality.
\end{proof}
\subsection{An adelic discrepancy operator}\label{sec:adelicdiscrepancy}
In order to complete the proof of Theorem~\ref{thm:mainthm}, we will need a similar tool like the discrepancy operator introduced in the proof of Proposition~\ref{prop:equidistributionrationalpointscongruence}. Given~$n\in\N$ and~$x>0$ we denote by~$\P(n,x)$ the set of primes~$p$ coprime to~$n$ and satisfying~$1<p<x$. We denote by~$\pi_{n}(x)$ the cardinality of~$\P(n,x)$. We fix~$0<\beta<\frac{1}{2}$ and focus on primes contained in~$\P(n,n^{\beta})$. Note that throughout this section we allow implicit constants to depend on~$\beta$. This dependency is often implicit. Let~$\pi:(0,\infty)\to\N$ denote the prime counting function, i.e.~$\pi(x)$ is defined to be the number of primes~$p$ satisfying~$p\leq x$. We know from the Prime Number Theorem, that
\begin{equation*}
  \pi(x)\asymp\frac{x}{\log x}
\end{equation*}
for sufficiently large~$x$. Let~$\omega(n)$ denote the number of distinct prime divisors of~$n$, then
\begin{equation*}
  \omega(n)\leq\log n,
\end{equation*}
and hence
\begin{equation}\label{eq:boundpinalphabeta}
  \pi(n^{\beta})\geq \pi_{n}(n^{\beta})\geq \pi(n^{\beta})-\log n\gg\pi(n^{\beta})
\end{equation}
for sufficiently large~$n$. We will later depend on the stronger result that~$\frac{\omega(n)\log\log n}{\log n}$ is bounded \cite[Sec.~22.10]{HardyWright}. We let~$S_{n,\beta}=\{\infty\}\cup\P(n,n^{\beta})$, so that~$(S_{n,\beta})_{\mathrm{f}}=\P(n,n^{\beta})$. Given some natural number~$d\in\N$, a function~$F\in\compactsmooth(\T_{S_{n,\beta}}\times X_{S_{n,\beta}})$ and~$(t,x)\in \T_{S_{n,\beta}}\times X_{S_{n,\beta}}$, we define
\begin{equation}\label{def:discrepancy}
  (D_{n,\beta,d}F)(t,x)=\frac{1}{\pi_{n}(n^{\beta})}\sum_{p\in \P(n,n^{\beta})}F\circ M_{de_{p}}(t,x)-E_{F},
\end{equation}
where~$e_{p}\in\Z^{\P(n,n^{\beta})}$ is defined by~$e_{p}(q)=\delta_{p,q}$. Using Proposition~\ref{prop:jointmixing}, we can bound the~$\Ltwo$-norm of~$D_{n,\beta,d}f$.
\begin{lemma}\label{lem:discrepancybound}
  There exists an~$\Ltwo$-Sobolev norm~$\Scal$ on~$\tensorcompactsmooth(\T_{S_{n,\beta}}\times X_{S_{n,\beta}})$ and some~$\sigma>0$ independent of~$n$ and~$\beta$ such that for all real-valued~$F\in\tensorcompactsmooth(\T_{S_{n,\beta}}\times X_{S_{n,\beta}})$ we have
  \begin{equation}\label{eq:bounddiscrepancy}
    \int_{\T_{S_{n,\beta}}\times X_{S_{n,\beta}}}(D_{n,\beta,d}F)^{2}\ll \beta n^{-\sigma\beta}\Scal(F)^{2}.
  \end{equation}
\end{lemma}
\begin{proof}
  We can always choose a weaker exponent in Proposition~\ref{prop:jointmixing} so that~$d\exponentdecaytimesjoint<1$. Combining Proposition~\ref{prop:propertiessobolevnorms}~(\ref{item:sobolevembedding}) with Proposition~\ref{prop:jointmixing}, for sufficiently large~$n$ we obtain
\begin{align}
  \notag\int_{\T_{S_{n,\beta}}\times X_{S_{n,\beta}}}&(D_{n,\beta,d}F)^{2}\leq\frac{1}{\pi_{n}(n^{\beta})^{2}}\sum_{p,q\in \P(n,n^{\beta})}\lvert\langle F\circ M_{d(e_{p}-e_{q})},F\rangle-E_{F}^{2}\rvert\\
  \label{2}&\ll\frac{1}{\pi_{n}(n^{\beta})}\lvert E_{F^{2}}-E_{F}^{2}\rvert+\frac{\mathcal{S}(F)^{2}}{\pi_{n}(n^{\beta})^{2}}\sum_{p\in \P(n,n^{\beta})}p^{-d\exponentdecaytimesjoint}\sum_{1<q<p}q^{-d\exponentdecaytimesjoint}\\
  \label{3}&\ll\frac{\Scal(F)^{2}}{\pi_{n}(n^{\beta})}+\frac{\mathcal{S}(F)^{2}}{\pi_{n}(n^{\beta})^{2}}\sum_{p\in \P(n,n^{\beta})}p^{1-d\exponentdecaytimesjoint}.
\end{align}
The last inequality is obtained using that
\begin{align*}
  \sum_{1<q<p}q^{-d\exponentdecaytimesjoint}\leq\int_{1}^{p}t^{-d\exponentdecaytimesjoint}\der t\ll p^{1-d\exponentdecaytimesjoint}.
\end{align*}
Let~$\delta>0$ arbitrary, then we similarly get
\begin{equation}\label{eq:boundgeneralizedprimenumbertheorem}
  \sum_{p\in\P(n,n^{\beta})}p^{\delta}\leq\int_{1}^{n^{\beta}+1}t^{\delta}\der t\ll n^{\beta(1+\delta)}.
\end{equation}
Using that~$\pi_{n}(n^{\beta})\asymp\frac{n^{\beta}}{\beta\log n}$ for sufficiently large~$n$, we obtain that for all~$d\in\N$
\begin{equation}\label{eq:bounddiscrepancytwo}
  \int_{X_{S_{n,\beta}}}(D_{n,\beta,d}f)^{2}\ll\beta\mathcal{S}(f)^{2}(n^{-\beta}+n^{-\beta d\exponentdecaytimesjoint})(\log n)^{2}\ll \beta n^{-\sigma\beta}\mathcal{S}(f)^{2}
\end{equation}
for some~$\sigma>0$.
\end{proof}
The next step is to bound~$\Scal(D_{n,\beta,d}f)$.
\begin{cor}\label{cor:sobolevdiscrepancy}
  Let~$\Scal$ be an~$\Ltwo$-Sobolev norm on~$\tensorcompactsmooth(\T_{S_{n,\beta}}\times X_{S_{n,\beta}})$. There exists a constant~$c>0$ depending only on the degree of~$\Scal$ such that for all~$F\in\tensorcompactsmooth(\T_{S_{n,\beta}}\times X_{S_{n,\beta}})$ we have
  \begin{equation*}
    \Scal(D_{n,\beta,d}F)\ll n^{\beta c}\Scal(F),
  \end{equation*}
  where the implicit constant depends only on the degree of~$\Scal$ and on~$\beta$.
\end{cor}
\begin{proof}
  Using the bound in~\eqref{eq:boundgeneralizedprimenumbertheorem} and Proposition~\ref{prop:propertiessobolevnorms}~(\ref{item:continuityofrepresentation}), we find some~$c>0$ depending only on the degree of~$\Scal$ such that
\begin{equation*}
  \Scal(D_{n,\beta,d}f)\leq\frac{1}{\pi_{n}(n^{\beta})}\sum_{p\in\P(n,n^{\beta})}p^{\frac{c}{2}}\Scal(f)\ll\frac{(\log n)n^{\beta(\frac{c}{2}+1)}}{n^{\beta}}\Scal(f)\ll n^{\beta c}\Scal(f).
\end{equation*}
The implicit constant is given by~\eqref{eq:boundgeneralizedprimenumbertheorem}, the inequalities~\eqref{eq:boundpinalphabeta}, and the bound~$\log n\ll n^{\beta c/2}$. In particular, it depends only on~$\beta$ and~$c$.
\end{proof}
Given~$a,b,n,l\in\N$ with~$(ab,n)=1$, define
\begin{align*}
  \Pcal_{\infty}^{\times d}(n;a,b)&=\{(\Z+\tfrac{ak^{d}}{n},\Gammar u_{bk^{d}/n}a_{\sqrt{n}}^{-1}):(k,n)=1\}\subseteq\Tr\times\Xr\\
  \Pcal_{\beta}^{\times d}(n;a,b)&=\big\{\big(\Z[(S_{n,\beta})^{-1}]+\Delta(\tfrac{ak^{d}}{n}),\Gamma_{S_{n,\beta}}\Delta(u_{bk^{d}/n})a_{\sqrt{n}}^{-1}\big):(k,n)=1\big\}\subseteq\T_{S_{n,\beta}}\times X_{S_{n,\beta}}.
\end{align*}
We denote by~$\mu_{n;a,b}^{\times d}$ the normalized counting measure on~$\Pcal_{\infty}^{\times d}(n;a,b)$ and by~$\mu_{n,\beta;a,b}^{\times d}$ the normalized counting measure on~$\Pcal_{\beta}^{\times d}(n;a,b)$. Using Lemma~\ref{lem:projectionrationalpoints}, we know that the natural projection
\begin{equation*}
  \T_{S_{n,\beta}}\times X_{S_{n,\beta}}\to\Tr\times\Xr
\end{equation*}
maps the set~$\Pcal_{\beta}^{\times d}(n;a,b)$ injectively onto~$\Pcal_{\infty}^{\times d}(n;a,b)$ and thus the push-forward of~$\mu_{n,\beta;a,b}^{\times d}$ under the natural projection equals~$\mu_{n;a,b}^{\times d}$.
\begin{lem}\label{lem:cardinalityresidues}
  Let~$n,a,b,d$ as above, then
  \begin{equation*}
    \lvert\Pcal_{\infty}^{\times d}(n;a,b)\rvert\gg\frac{\phi(n)}{d^{\omega(n)}},
  \end{equation*}
  where~$\phi(n)$ denotes the Euler totient function.
\end{lem}
\begin{proof}
  Since~$(n,ab)=1$, we can assume without loss of generality that~$a=b=1$. Furthermore, using the Chinese Remainder Theorem, we know that~$m\in(\quotient{\Z}{n\Z})^{\times}$ is a degree-$d$ residue mod~$n$, i.e.~there is some~$k\in(\quotient{\Z}{n\Z})^{\times}$ such that~$k^{d}\equiv m\,\mod n$, if and only if~$m$ is a degree-$d$ residue mod~$p^{\nu_{p}(n)}$ for all primes~$p$ dividing~$n$ where
  \begin{equation*}
    \nu_{p}(n)=\sup\{\nu\in\N:p^{\nu}|n\}.
  \end{equation*}
  In particular, we have
  \begin{equation*}
    \lvert\Pcal_{\infty}^{\times d}(n;1,1)\rvert=\prod_{p|n}\lvert\Pcal_{\infty}^{\times d}(p^{\nu_{p}(n)};1,1)\rvert
  \end{equation*}
  and it suffices to consider prime powers. Fix an arbitrary prime~$p$ and~$r\in\N$. Define
  \begin{equation*}
    f_{d}:\big(\quotient{\Z}{p^{r}\Z}\big)^{\times}\to\big(\quotient{\Z}{p^{r}\Z}\big)^{\times}, k\mapsto k^{d}.
  \end{equation*}
  By the First Isomorphism Theorem, we know that
  \begin{equation*}
    \lvert\Pcal_{\infty}^{\times d}(p^{r};1,1)\rvert=\frac{p^{r-1}(p-1)}{\lvert\ker f_{d}\rvert}.
  \end{equation*}
  In particular, we know that~$\big(d,p^{r-1}(p-1)\big)=1\implies\lvert\Pcal_{\infty}^{\times d}(p^{r};1,1)\rvert=\phi(p^{r})$. Now assume that~$\big(d,p^{r-1}(p-1)\big)\neq 1$ and assume that~$p$ is odd. In that case, it is well known that~$(\quotient{\Z}{p^{r}\Z})^{\times}$ is a cyclic group, i.e.
  \begin{equation*}
    (\quotient{\Z}{p^{r}\Z})^{\times}\cong\quotient{\Z}{p^{r-1}(p-1)\Z}
  \end{equation*}
  and as in general
  \begin{equation*}
    m|dk\iff\frac{m}{(m,d)}|k,
  \end{equation*}
  the kernel of multiplication by~$d$ in~$\quotient{\Z}{m\Z}$ is given by
  \begin{equation*}
    \{k\tfrac{m}{(m,d)}:k=0,\ldots,(m,d)-1\},
  \end{equation*}
  i.e.~the kernel of~$f_{d}$ in~$\quotient{\Z}{p^{r-1}(p-1)}$ has cardinality~$\big(d,p^{r-1}(p-1)\big)$. In particular
  \begin{equation*}
    \lvert\Pcal_{\infty}^{\times d}(p^{r};1,1)\rvert=\frac{\phi(p^{r})}{(\phi(p^{r}),d)}.
  \end{equation*}
  If~$p=2$ and~$r\geq 2$, then
  \begin{equation*}
    \big(\quotient{\Z}{2^{r}\Z}\big)^{\times}\cong\quotient{\Z}{2\Z}\times\quotient{\Z}{2^{r-2}\Z}.
  \end{equation*}
  Let~$x\in\quotient{\Z}{2\Z}$ and~$k\in\quotient{\Z}{2^{r-2}\Z}$. Then~$d(x,k)=0$ if and only if~$2|dx$ and~$2^{r-2}|dk$. As~$2|d$ by assumption, we know that~$2|dx$. Let~$d=(2^{r-2},d)d^{\prime}$. If~$(2^{r-2},d)=2^{r-2}$, then~$d(x,k)=0$ for all~$(x,k)\in\quotient{\Z}{2\Z}\times\quotient{\Z}{2^{r-2}\Z}$. Otherwise let~$(2^{r-2},d)=2^{m}$. Then in the second component the kernel of multiplication by~$d$ is generated by~$2^{r-2-m}$ and thus has cardinality~$2^{m}=(2^{r-2},d)$ and in particular
  \begin{equation*}
    \lvert\Pcal_{\infty}^{\times d}(2^{r};1,1)\rvert=\frac{\phi(2^{r})}{2(2^{r-2},d)}.
  \end{equation*}
  Combining all of these, we obtain
  \begin{equation*}
    \lvert\Pcal_{\infty}^{\times d}(n;1,1)\rvert=\prod_{p|n}\lvert\Pcal_{\infty}^{\times d}(p^{\nu_{p}(n)};1,1)\rvert\asymp\phi(n)\prod_{p|n}\frac{1}{(p^{\nu_{p}(n)},d)}\gg\frac{\phi(n)}{d^{\omega(n)}}.
  \end{equation*}
\end{proof}
\begin{proof}[Proof of Theorem~\ref{thm:mainthm}]
In what follows, we denote by~$\mu_{n,\beta;a,b}$ the normalized counting measure over the rational points as in Corollary~\ref{cor:equidistributionrationalpoints} for the places~$S_{n,\beta}$. Assume that~$F\in\tensorcompactsmooth(\T_{S_{n,\beta}}\times X_{S_{n,\beta}})$ is real-valued. Note that the set~$\Pcal_{\beta}^{\times d}(n;a,b)$ is invariant under~$M_{de_{p}}$ for all~$p\in S_{n,\beta}$. Using Corollary~\ref{cor:equidistributionrationalpoints} and Lemma~\ref{lem:cardinalityresidues}, we have
\begin{align*}
  \lvert\mu_{n,\beta;a,b}^{\times d}(F)-E_{F}\rvert^{2}&=\lvert\mu_{n,\beta;a,b}^{\times d}(D_{n,\beta,d}F)\rvert^{2}\ll\tfrac{d^{\omega(n)}n}{\phi(n)}\mu_{n,\beta;a,b}\big((D_{n,\beta,d}F)^{2}\big)\\
  &\ll(ab)^{\exponentpolynomialsobolevunipotent_{1}}\tfrac{ne^{(\log d)\omega(n)}}{\phi(n)}\bigg(\int_{\T_{S_{n,\beta}}\times X_{S_{n,\beta}}}(D_{n,\beta,d}F)^{2}+n^{-\exponentdecayrationalpointsS}\Scal\big((D_{n,\beta,d}F)^{2}\big)\bigg),
\end{align*}
for some~$\Ltwo$-Sobolev norm~$\Scal$ on~$\tensorcompactsmooth(\T_{S_{n,\beta}}\times X_{S_{n,\beta}})$. Using Proposition~\ref{prop:propertiessobolevnorms}~(\ref{item:multiplicativity}) and Corollary~\ref{cor:sobolevdiscrepancy}, we can find an~$\Ltwo$-Sobolev norm~$\Scal_{1}$ on~$\tensorcompactsmooth(\T_{S_{n,\beta}}\times X_{S_{n,\beta}})$ such that
\begin{equation*}
  \Scal\big((D_{n,\beta,d}F)^{2}\big)\leq\Scal_{1}(D_{n,\beta,d}f)^{2}\ll n^{2\beta c}\Scal_{1}(F)^{2}.
\end{equation*}
Recall that~$\omega(n)\ll\frac{\log n}{\log\log n}$ (cf.~\cite[\textsection22.10]{HardyWright}). %Note that we could always choose~$\beta$ smaller at the expense of a worse exponent and adjusting some of the implicit constants. So we choose~$\beta$ so that~$2\beta c<\frac{\exponentdecayrationalpointsS}{2}$.
Set $\beta=\frac{\exponentdecayrationalpointsS}{2c+\sigma}$. This choice is independent of~$n,a,b$ and~$F$. Combining all this with the result in Lemma~\ref{lem:discrepancybound}, we find an~$\Ltwo$-Sobolev norm~$\Scal_{2}$ and positive numbers~$c^{\prime},\kappa_{3}>0$ such that
\begin{align*}
  \lvert\mu_{n,\beta;a,b}^{\times d}(F)-E_{F}\rvert^{2}&\ll(ab)^{\exponentpolynomialsobolevunipotent_{1}}\frac{n^{1+(\log d)\frac{c^{\prime}}{\log\log n}}}{\phi(n)}(n^{-\sigma\beta}+n^{2\beta c-\exponentdecayrationalpointsS})\Scal_{2}(f)^{2}\\
  &\ll(ab)^{\exponentpolynomialsobolevunipotent_{1}}\Scal_{2}(f)^{2}\frac{n^{1-\kappa_{3}}}{\phi(n)}.
\end{align*}
Now we use
\begin{equation*}
  \lim\inf\frac{\phi(n)\log\log n}{n}>0
\end{equation*}
to obtain that for sufficiently large~$n$ we have
\begin{equation*}
  \frac{n^{1-\delta}}{\phi(n)}=\frac{n}{\phi(n)\log\log n}\frac{\log\log n}{n^{\delta}}\ll_{\delta}n^{-\delta/2}
\end{equation*}
for all~$\delta\in(0,1)$. Plugging this into the above result, we find
\begin{equation}\label{eq:sarithmeticmainthm}
  \lvert\mu_{n,\beta;a,b}^{\times d}(F)-E_{F}\rvert^{2}\ll(ab)^{\exponentpolynomialsobolevunipotent_{1}}\Scal_{2}(F)^{2}n^{-\frac{\kappa_{3}}{2}}.
\end{equation}
Identifying~$\tensorcompactsmooth(\Tr\times\Xr)$ with a subspace of~$\tensorcompactsmooth(\T_{S_{n,\beta}}\times X_{S_{n,\beta}})$ as outlined in Section \ref{sec:sobolev}, Theorem~\ref{thm:mainthm} now follows from Theorem~\ref{thm:boundsobolevnorm}.
\end{proof}
\begin{remark}
  It follows immediately from the argument that the degree-$d$ residues without coprimality assumption equidistribute with a rate. Indeed, the argument only used the invariance of the subsets under~$M_{de_{p}}$ and the fact that it is not too small in comparison to the set of all rational points.
\end{remark}
We end this section with a proof of Corollary~\ref{cor:cormainthm}, which shows the equidistribution of sequences of cosets of degree-$d$ residues in the separate factors. That is, we show that for any sequence~$b_{n}\in\Z$ satisfying~$(b_{n},n)=1$ the sets
\begin{equation}\label{eq:primitivepointsfactor}
  \{\Gamma_{\infty}u_{b_{n}k^{d}/n}a_{\sqrt{n}}^{-1}:(k,n)=1\}\subseteq\Xr
\end{equation}
equidistribute with a polynomial rate independent of the sequence~$b_{n}$. It will be immediate that the proof can easily be adapted to prove the analog for the torus. It will be useful for later purposes to state the corollary in the~$S$-arithmetic setup. In what follows, we denote by~$\pi_{X}:\T_{S_{n,\beta}}\times X_{S_{n,\beta}}\to X_{S_{n,\beta}}$ the projection onto the second coordinate. Note that for any~$a_{n}\in\Z$ coprime to~$n$, the image of the set~$\Pcal_{\beta}^{\times d}(n;a_{n},b_{n})$ under~$\pi_{X}$ is independent of~$a_{n}$ and coincides with the lift of the set in~\eqref{eq:primitivepointsfactor} to~$X_{S_{n,\beta}}$.
\begin{cor}\label{cor:equidistributionfactors}
  Fix~$d\in\N$ and a finite set~$S$ of places containing the infinite place. There exist an~$\Ltwo$-Sobolev norm~$\Scal_{\XS}$ on~$\compactsmooth(\XS)$ and positive constants~$C_{1},\exponentdecayprimitiverationalpointsSfactor>0$ such that the following holds. Given~$n\in\N$ and~$b\in\Z$, denote by~$\mu_{n,b;\XS}^{\times d}$ the normalized counting measure on
  \begin{equation*}
    \{\GammaS\Delta(u_{bk^{d}/n})a_{\sqrt{n}}^{-1}:(k,n)=1\}\subseteq\XS.
  \end{equation*}
  For all~$f\in\compactsmooth(\XS)$ and for all sequences~$b_{n}\in\Z$ satisfying~$(n,b_{n})=1$ we have
  \begin{equation*}
    \bigg\lvert\mu_{n,b_{n};\XS}^{\times d}(f)-\int_{\XS}f\der\nu_{S}\bigg\rvert\leq C_{1}n^{-\exponentdecayprimitiverationalpointsSfactor}\Scal_{\XS}(f).
  \end{equation*}
\end{cor}
\begin{proof}
  We assume without loss of generality that~$f$ is real-valued. Given two finite sets of places~$S\subseteq S^{\prime}$, denote~$K_{\neg\Sf}=\prod_{p\in\Sf^{\prime}\setminus\Sf}K_{p}[0]$. Fix a function~$f\in\compactsmooth(\XS)$ and note that for any~$\beta>0$ and sufficiently large~$n$ we can view~$f$ as an element in~$\compactsmooth(X_{S_{n,\beta}})$, as similarly to Corollary~\ref{cor:extension} we have 
  \begin{equation*}
    \biquotient{\SLtwo(\Z[S_{n,\beta}^{-1}])}{\SLtwo(\Q_{S_{n,\beta}})}{K_{\neg\Sf}}\cong\rquotient{\SLtwo(\ZS)}{\SLtwo(\QS)}
  \end{equation*}
  as~$\SLtwo(\QS)$-spaces. Thus we can embed~$\compactsmooth(\XS)$ in~$\compactsmooth(X_{S_{n,\beta}})$, identifying it with the subspace of~$K_{\neg\Sf}$-invariant functions for~$n\in\D(\prod_{p\in\Sf}p)$ sufficiently large, where we recall that for any natural number $m\in\N$ we denote by $\D(m)$ the set of natural numbers coprime to $m$. For any basis~$\Xfrak$ of~$\liesltwo(\R)$ and any fixed degree~$D$, with respect to the corresponding~$\Ltwo$-Sobolev norms this embedding is an isometry onto that subspace. Note that the analog of this holds for smooth functions defined on~$\TS$. In particular, it suffices to prove the bound for functions on~$X_{S_{n,\beta}}$ for sufficiently large~$n$. Also note that~$\mu_{n,b;X_{S_{n,\beta}}}^{\times d}=(\pi_{X})_{\ast}\mu_{n,\beta;a,b}^{\times d}$ for all~$(a,n)=1$, as the projection is independent of the torus component.

  The crucial point of the argument to follow is that the degree-$d$ residues form a not too thin subset of the set of the primitive rational points, and the latter, for a single factor, are invariant under multiplication by units mod~$n$. Hence, after projection to a single factor, we can apply Theorem~\ref{thm:mainthm} for the primitive rational points. More precisely, Lemma~\ref{lem:cardinalityresidues} implies that for all sequences~$a_{n},b_{n}\in\Z$ with~$(n,a_{n}b_{n})=1$ we have
  \begin{align*}
    \lvert\mu_{n,\beta;a_{n},b_{n}}^{\times d}(f\circ \pi_{X})-E_{f\circ \pi_{X}}\rvert^{2}&=\big\lvert\mu_{n,\beta;a_{n},b_{n}}^{\times d}\big(D_{n,\beta,d}(f\circ \pi_{X})\big)\big\rvert^{2}\\
    &\ll d^{\omega(n)}\mu_{n,\beta;a_{n},b_{n}}^{\times 1}\big(\big(D_{n,\beta,d}(f\circ \pi_{X})\big)^{2}\big).
  \end{align*}
  As~$D_{n,\beta,d}(f\circ \pi_{X})$ is constant in the~$\T_{S_{n,\beta}}$-component and as multiplication by~$b_{n}$ acts by permutation on the group of units mod~$n$, we have
  \begin{equation*}
    \mu_{n,\beta;a_{n},b_{n}}^{\times 1}\big(\big(D_{n,\beta,d}(f\circ \pi_{X})\big)^{2}\big)=\mu_{n,\beta;1,1}^{\times 1}\big(\big(D_{n,\beta,d}(f\circ \pi_{X})\big)^{2}\big).
  \end{equation*}
  In particular, Theorem~\ref{thm:mainthm} implies that
  \begin{align*}
    \lvert\mu_{n,\beta;a_{n},b_{n}}^{\times d}(f\circ \pi_{X})-E_{f\circ \pi_{X}}\rvert^{2}&\ll d^{\omega(n)}\int_{\T_{S_{n,\beta}}\times X_{S_{n,\beta}}}\big(D_{n,\beta,d}(f\circ \pi_{X})\big)^{2}\\
    &\qquad\qquad+d^{\omega(n)}n^{-\kappa}\Scal\big(\big(D_{n,\beta,d}(f\circ \pi_{X})\big)^{2}\big).
  \end{align*}
  Now we again combine Lemma~\ref{lem:discrepancybound} and Corollary~\ref{cor:sobolevdiscrepancy} with~$\omega(n)\ll\frac{\log n}{\log\log n}$ to obtain
  \begin{equation*}
    \lvert\mu_{n,\beta;a_{n},b_{n}}^{\times d}(f\circ \pi_{X})-E_{f\circ \pi_{X}}\rvert\leq C_{1}n^{-\exponentdecayprimitiverationalpointsSfactor}\Scal_{X_{S_{n,\beta}}}(f)
  \end{equation*}
  just like we did in the proof of Theorem~\ref{thm:mainthm}. Combining this with our initial remarks, the corollary follows.
\end{proof}
%%%%%%%%%%%% Two torus %%%%%%%%%%%%%%
\section{Equidistribution in the product of the two-torus and the modular surface}\label{sec:twotorus}
In this section, we are going to provide an ineffective argument for equidistribution of the primitive rational points in both the two-torus and the unit tangent bundle to the modular surface under certain congruence conditions. For what follows, given~$n,a,b,c,d\in\N$, we will denote
\begin{equation*}
  \primitiverationalsnabcrtwotorus=\set{(\tfrac{ak^{d}}{n},\tfrac{b\overline{k}^{d}}{n},\Gammar u_{ck^{d}/n}a_{\sqrt{n}}^{-1})}{(k,n)=1}\subseteq\T_{\infty}^{2}\times\Xr,
\end{equation*}
where~$\overline{k}\in\Z$ is any integer satisfying~$k\overline{k}\equiv1\mod n$.
Recall that we have previously identified~$\T_{\infty}\cong\Gammar\Ur\cong\Gammar\Vr$. Using this identification, the set~$\primitiverationalsnabcrtwotorus$ identifies with
\begin{equation*}
  \set{(\Gammar u_{ak^{d}/n},\Gammar u_{ck^{d}/n}a_{\sqrt{n}}^{-1},\Gammar u_{\overline{b}k^{d}/n}a_{\sqrt{n}}^{-2})}{(k,n)=1}\subseteq \Gammar\Ur\times\Xr\times\Gammar\Vr.
\end{equation*}
Denote by~$\measureprimitiverationalsnabctwotorusr$ the normalized counting measure on~$\primitiverationalsnabcrtwotorus$. In what follows, we will have to restrict ourselves to denominators~$n$ coprime to two distinct, fixed finite primes~$p,q$. For convenience we recall Theorem~\ref{thm:realtwotorus} in this notation.
\begin{thmwo}
  Let~$p,q$ be distinct finite places for~$\Q$ and let~$\kappa,\eta>0$ as in Theorem~\ref{thm:mainthm}. If~$a_{n},b_{n},c_{n}\in\Z$ satisfy~$(a_{n}b_{n}c_{n},n)=1$, then~$\measureprimitiverationalsnanbncntwotorusr$ equidistributes towards the invariant probability measure on the product~$\Tr^{2}\times\Xr$ as~$n\to\infty$ with~$n\in\D(pq)$.
\end{thmwo}

As was the case for the preceding results, we will prove Theorem~\ref{thm:realtwotorus} via a corresponding statement in the~$S$-arithmetic extension. From now on, unless stated otherwise,~$S$ is a fixed set of places containing the infinite place and at least two distinct finite places, i.e.~$\lvert S\rvert\geq 3$. Given~$n\in\N$ such that~$(n,\prod_{p\in\Sf}p)=1$, we consider the subsets
\begin{equation*}
  \primitiverationalsnabctwotorus=\set{\big(\ZS+\Delta(\tfrac{ak^{d}}{n}),\ZS+\Delta(\tfrac{b\overline{k^{d}}}{n}),\GammaS\Delta(u_{ck^{d}/n})a_{\sqrt{n}}^{-1}\big)}{(k,n)=1}
\end{equation*}
of~$\TS^{2}\times\XS$. Denote by~$\measureprimitiverationalsnabctwotorus$ the normalized counting measure on~$\primitiverationalsnabctwotorus$.
\begin{proposition}\label{prop:sarithmetictwotorus}
  Let~$a_{n},b_{n},c_{n}\in\Z$ be chosen so that~$(a_{n}b_{n}c_{n},n)=1$, then~$\measureprimitiverationalsnanbncntwotorus$ equidistributes towards the invariant probability measure on the product~$\TS^{2}\times\XS$ as~$n\to\infty$ with~$n\in\D(\prod_{p\in\Sf}p)$.
\end{proposition}
In fact, what we show is disjointness of certain higher rank actions on~$\TS$ and on~$\XS$. In the proof of Proposition~\ref{prop:sarithmetictwotorus} we will show that any limit is a joining of the~$\Z^{\Sf}$ actions on~$\TS$ and on~$\XS$ given by the times-$p$ maps, their inverses, and right-multiplication with the diagonal lattice elements~$a_{p}^{-1}$ for~$p\in\Sf$. The heart of the argument consists of showing that the trivial joining is the only joining for these actions.
\begin{proof}[Proof of Theorem {\ref{thm:realtwotorus}} assuming Proposition {\ref{prop:sarithmetictwotorus}}]
  Set~$S=\{\infty,p,q\}$. As~$(n,pq)=1$ by assumption, we have~$\frac{a_{n}k^{d}}{n},\frac{b_{n}\overline{k^{d}}}{n},\frac{c_{n}k^{d}}{n}\in\Zs$ and thus the projection~$\Pi:\TS^{2}\times\XS\to\Tr^{2}\times\Xr$ defined by mapping to the space of~$\Zs^{2}\times K[0]$-orbits sends~$\primitiverationalsnanbncntwotorus$ injectively onto~$\primitiverationalsnanbncnrtwotorus$. As of Proposition~\ref{prop:sarithmetictwotorus}, the sequence of points~$\primitiverationalsnanbncntwotorus$ equidistributes with respect to the~$\QS^{2}\times\GS$-invariant probability measure~$m_{\TS^{2}}\otimes\nu_{S}$ and thus its projection~$\primitiverationalsnanbncnrtwotorus$ equidistributes with respect to~$\Pi_{\ast}(m_{\TS^{2}}\otimes\nu_{S})=m_{\Tr^{2}}\otimes m_{\Xr}$, the~$\R^{2}\times\Gr$-invariant probability measure on~$\Tr^{2}\times\Xr$.
\end{proof}
In order to prove Proposition~\ref{prop:sarithmetictwotorus}, we proceed as follows. We will first show that every measure~$\measureprimitiverationalsnanbncntwotorus$ is invariant under the action of~$\Z^{\Sf}$ on~$\TS^{2}\times\XS$ given by
\begin{equation}\label{eq:theaction}
  T_{m}(t,s,x)=(S^{2dm}t,S^{-2dm}s,x a_{S^{dm}}^{-1})\quad(t,s\in\TS,x\in\XS,m\in\Z^{\Sf}),
\end{equation}
where~$a_{S^{dm}}\in\GammaS$ is the diagonal embedding of the matrix~$\big(\begin{smallmatrix}S^{dm} & 0 \\ 0 & S^{-dm}\end{smallmatrix}\big)$. Hence every \weakstar limit is invariant under this action. We will show that every \weakstar limit is a measure of maximal entropy and hence uniqueness of the measure of maximal entropy will imply that the unique \weakstar limit is the invariant probability measure. In the proof we will apply higher rank rigidity arguments and more precisely apply a result from \cite{definitivejoining}.
\begin{lem}\label{lem:invariantsets}
  The sets~$\primitiverationalsnanbncntwotorus$ are invariant under~$\Z^{\Sf}$. In particular, the measures~$\measureprimitiverationalsnanbncntwotorus$ are~$\Z^{\Sf}$-invariant.
\end{lem}
\begin{proof}
  Let~$u\in\Z$ such that~$uS^{m}\equiv 1\mod n$. For any~$k\in\Z$ and for every~$a\in\Z$ we have
  \begin{equation*}
    S^{m}(u\tfrac{ak^{d}}{n})\equiv\tfrac{ak^{d}}{n}\mod\Z
  \end{equation*}
  and thus it follows that
  \begin{equation*}
    u\tfrac{ak^{d}}{n}+\ZS=S^{-m}\tfrac{ak^{d}}{n}+\ZS.
  \end{equation*}
  In particular we get
  \begin{align*}
    T_{m}\big(\ZS+&\Delta(\tfrac{a_{n}k^{d}}{n}),\ZS+\Delta(\tfrac{b_{n}\overline{k^{d}}}{n}),\GammaS u_{c_{n}k^{d}/n}a_{\sqrt{n}}^{-1}\big)\\
    &=\big(\ZS+\Delta(\tfrac{a_{n}S^{2dm}k^{d}}{n}),\ZS+\Delta(\tfrac{b_{n}\overline{S^{2dm}k^{d}}}{n}),\GammaS u_{c_{n}S^{2dm}k^{d}/n}a_{\sqrt{n}}^{-1}\big).
  \end{align*}
  This implies the claim.
\end{proof}
\subsection{Lyapunov exponents and coarse Lyapunov subgroups}\label{sec:lyapunov}
We denote by~$\HS$ the group
\begin{equation*}
  \HS=\US\times\VS\times\GS\leq\SLtwo(\QS)^{3}.
\end{equation*}
Let~$\HGamma=\HS\cap\SLtwo(\ZS)^{3}$ and note that there is a natural isomorphism
\begin{equation*}
  \Psi_{S}:\TS^{2}\times\XS\to\rquotient{\HGamma}{\HS}.
\end{equation*}
The action of~$\Z^{\Sf}$ described in~\eqref{eq:theaction} above can be described as follows. Given~$p\in\Sf$, let~$\alpha_{p}\in\SLtwo(\QS)^{3}$ denote the element given as~$\alpha_{p}=(a_{p^{d}},a_{p^{d}},a_{p^{d}})$. Note that conjugation by~$\alpha_{p}$ preserves the lattice~$\HGamma$. Hence for all~$g\in\HS$ the map~$\HGamma g\mapsto\alpha_{p}\cdot\HGamma g=\HGamma\alpha_{p}g\alpha_{p}^{-1}$ is well-defined. We denote by~$e_{p}\in\Z^{\Sf}$ the vector satisfying~$(e_{p})_{q}=\delta_{p,q}$ for all~$q\in\Sf$, so that we get the equivariance~$\alpha_{p}\circ\Psi_{S}=\Psi_{S}\circ T_{de_{p}}$. Setting~$\alpha_{S^{m}}=\prod_{p\in\Sf}\alpha_{p}^{m_{p}}$ for~$m\in\Z^{\Sf}$, we have
\begin{equation}\label{eq:homogeneousequivariance}
  \alpha_{S^{m}}\circ\Psi_{S}=\Psi_{S}\circ T_{dm},
\end{equation}
where~$T_{dm}$ was defined in~\eqref{eq:theaction}. For the dynamics, we are particularly interested in the following subgroups of~$\HS$. We denote by~$H^{(1)}$ and~$H^{(2)}$ the embeddings of~$U$ and~$V$ in the first and second coordinates respectively. Furthermore, the groups~$H^{(-)}$ and~$H^{(+)}$ shall denote the upper- and lower-triangular unipotent groups in the third coordinate respectively. These groups shall be called the \emph{elementary unipotent subgroups}. We set~$\Ical=\{1,2,+,-\}$ and remark here that~$\HS=\langle\HS^{(i)};i\in\Ical\rangle$ for all finite sets of places~$S$ of~$\Q$. A similar statement holds for~$\Hs$. In what follows, a \emph{Lyapunov exponent} for~$\alpha$ is an additive functional~$\chi:\Z^{\Sf}\to\R$ which is non-trivial and for which there exists an elementary unipotent subgroup~$H^{(i)}$ and a place~$p\in S$---i.e.~$\chi$ is a Lyapunov exponent on~$H^{(i)}$---such that
\begin{equation*}
  \lVert\alpha_{S^{m}}(h-\one)\alpha_{S^{-m}}\rVert_{p}=e^{\chi(m)}\lVert h-\one\rVert_{p}\qquad(h\in H_{p}^{(i)}).
\end{equation*}
We let
\begin{equation*}
  \Phi_{i,p}=\set{\chi}{\chi\text{ is a Lyapunov exponent on }\Hp^{(i)}},
\end{equation*}
as well as~$\Phi_{p}=\bigcup_{i\in\Ical}\Phi_{i,p}$,~$\Phi_{i}=\bigcup_{p\in S}\Phi_{i,p}$, and~$\Phi=\bigcup_{i\in\Ical}\Phi_{i}=\bigcup_{p\in S}\Phi_{p}$.

One calculates
\begin{align*}
  \alpha_{p}(h-\one)\alpha_{p}^{-1}&=p^{2d}(h-\one)\quad \text{if }h\in\HS^{(1)}\text{ or }h\in\HS^{(-)},\text{ and}\\
  \alpha_{p}(h-\one)\alpha_{p}^{-1}&=p^{-2d}(h-\one)\quad \text{if }h\in\HS^{(2)}\text{ or }h\in\HS^{(+)}.
\end{align*}
In particular, for all distinct~$p,q\in\Sf$, we have 
\begin{align*}
  \lVert\alpha_{p}(h-\one)\alpha_{p}^{-1}\rVert_{q}&=\lVert h-\one\rVert_{q}\quad\text{if }h\in\Hq^{(i)}, i\in\{1,2,+,-\},\text{ and}\\
  \lVert\alpha_{p}(h-\one)\alpha_{p}^{-1}\rVert_{p}&=p^{-2d}\lVert h-\one\rVert_{p}\quad \text{if }h\in\Hp^{(1)}\text{ or }h\in\Hp^{(-)},\text{ and}\\
  \lVert\alpha_{p}(h-\one)\alpha_{p}^{-1}\rVert_{p}&=p^{2d}\lVert h-\one\rVert_{p}\quad \text{if }h\in\Hp^{(2)}\text{ or }h\in\Hp^{(+)},\\
  \lVert\alpha_{p}(h-\one)\alpha_{p}^{-1}\rVert_{\infty}&=p^{2d}\lVert h-\one\rVert_{\infty}\quad \text{if }h\in\Hr^{(1)}\text{ or }h\in\Hr^{(-)},\text{ and}\\
  \lVert\alpha_{p}(h-\one)\alpha_{p}^{-1}\rVert_{\infty}&=p^{-2d}\lVert h-\one\rVert_{\infty}\quad \text{if }h\in\Hr^{(2)}\text{ or }h\in\Hr^{(+)}.
\end{align*}
As the map~$\Z^{\Sf}\to\HS$,~$m\mapsto\alpha_{S^{m}}$ is a homomorphism, we obtain the following
\begin{cor}\label{cor:weights}
  The functionals appearing in~$\Phi$ are of the form~$\chi=\varepsilon\chi_{p}$ where~$\varepsilon\in\{\pm1\}$ and
  \begin{equation*}
    \chi_{p}(m)=\begin{cases}
      m_{p}2d\log p & \text{if }p\in\Sf,\\
      \sum_{q\in\Sf}m_{q}2d\log q & \text{else.}
    \end{cases}
  \end{equation*}
  Furthermore, we have
  \begin{align*}
    \Phi_{1,p}&=\Phi_{-,p}=\{-\chi_{p}\},\quad\Phi_{2,p}=\Phi_{+,p}=\{\chi_{p}\},\quad(p\in\Sf)\\
    \Phi_{1,\infty}&=\Phi_{-,\infty}=\{\chi_{\infty}\},\quad\Phi_{2,\infty}=\Phi_{+,\infty}=\{-\chi_{\infty}\}.
  \end{align*}
\end{cor}
Given a Lyapunov exponent~$\chi\in\Phi$, we define the corresponding Lyapunov subgroup by
\begin{equation*}
  H_{\chi}^{-}=\prod_{\substack{p\in S,i\in\Ical \\ -\chi\in\Phi_{i,p}}}\Hp^{(i)}.
\end{equation*}
\begin{cor}\label{cor:lyapunovsubgroups}
  The Lyapunov subgroups are given as follows:
  \begin{align*}
    H_{\chi_{p}}^{-}&=\Hp^{(1)}\Hp^{(-)},\qquad H_{-\chi_{p}}^{-}=\Hp^{(2)}\Hp^{(+)}\qquad(p\in\Sf),\\
    H_{\chi_{\infty}}^{-}&=\Hr^{(2)}\Hr^{(+)},\qquad H_{-\chi_{\infty}}^{-}=\Hr^{(1)}\Hr^{(-)}.
  \end{align*}
\end{cor}
We note that in our case two Lyapunov exponents are equivalent in the sense of \cite[Section~2.2]{definitivejoining} if and only if they are equal. Thus each Lyapunov exponent is in fact a \emph{coarse} Lyapunov exponent. This is of importance as the statements we will use in general hold for coarse Lyapunov subgroups as opposed to Lyapunov subgroups defined here. In Corollary~\ref{cor:lyapunovsubgroups} we made use of the higher-rank assumption~$\lvert S\rvert>2$, as in the case~$S=\{\infty,p\}$ we would have~$\chi_{\infty}=\chi_{p}$ and thus~$H_{\chi_{p}}=H_{\chi_{\infty}}$. 
\subsection{Two disjointness results}\label{sec:disjointness}
Using our understanding of the (coarse) Lyapunov subgroups discussed in Section~\ref{sec:lyapunov}, we can prove two disjointness results for the~$\Z^{\Sf}$-actions under consideration. The proofs are applications of the product structure of leafwise measures for higher rank actions and the classical Abramov-Rokhlin formula, where the factors are chosen so that the Lyapunov subgroup for the factors are trivial. The first disjointness result gives rise to ineffective equidistribution in the setup of Section~\ref{sec:primitivepoints} discussed in Corollary \ref{cor:ineffectivetorustimessurface} but more importantly serves as an input to the second disjointness result.
\begin{prop}[Disjointness for two factors]\label{prop:disjointnesstwofactors}
  Let~$\mu$ be an ergodic joining of the~$\Z^{\Sf}$ actions on~$(\TS,m_{\TS})$ and~$(\XS,\nu_{S})$, that is,~$\mu$ is a probability measure on~$\TS\times\XS$, invariant and ergodic under the action~$\Z^{\Sf}\curvearrowright\TS\times\XS$,~$m.(t,x)=(S^{2dm}t,xa_{S^{dm}}^{-1})$ and satisfying~$(\pi_{\TS})_{\ast}\mu=m_{\TS}$ and~$(\pi_{\XS})_{\ast}\mu=\nu_{S}$. Then~$\mu=m_{\TS}\otimes\nu_{S}$, that is, the systems are disjoint. The same holds for the action~$\Z^{\Sf}\curvearrowright\TS\times\XS$,~$m.(t,x)=(S^{-2dm}t,xa_{S^{dm}}^{-1})$.
\end{prop}
\begin{proof}
  We only prove the first case. The second follows by interchanging the roles of the subgroups~$\US$ and~$\VS$. Using the discussion from Section~\ref{sec:lyapunov} and restricting to the subgroup~$\US\times\GS=\tildeHS$, we note that the Lyapunov exponents are exactly the same for this smaller system, however, the Lyapunov subgroups are given by
  \begin{align*}
    \tilde{H}_{\chi_{p}}^{-}&=\tildeHp^{(1)}\tildeHp^{(-)},\qquad \tilde{H}_{-\chi_{p}}^{-}=\tildeHp^{(+)}\qquad(p\in\Sf)\\
    \tilde{H}_{\chi_{\infty}}^{-}&=\tildeHr^{(+)},\phantom{\tildeHr^{(2)}}\qquad \tilde{H}_{-\chi_{\infty}}^{-}=\tildeHr^{(1)}\tildeHr^{(-)},
  \end{align*}
  where~$\tilde{H}^{(1)}$ denotes the embedding of~$U$ in the first coordinate of~$\tilde{H}$ and~$\tilde{H}^{(-)}$ and~$\tilde{H}^{(+)}$ denote the upper- and lower-triangular unipotents in the second coordinate. We now apply a consequence of the product structure for leafwise measures and the Abramov-Rokhlin formula, namely using Corollary~6.5 from \cite{definitivejoining}, we obtain
  \begin{equation*}
    h_{\mu}(\alpha_{p},\tilde{H}_{\chi_{\infty}}^{-})=h_{\nu_{S}}(\alpha_{p},\tilde{H}_{\chi_{\infty}}^{-})+h_{\mu}(\alpha_{p},\underbrace{\tilde{H}_{\chi_{\infty}}^{-}\cap\US}_{=\{1\}})=h_{\nu_{S}}(\alpha_{p},\tilde{H}_{\chi_{\infty}}^{-}),
  \end{equation*}
  where~$h_{\mu}(\alpha,\tilde{H}_{\chi_{\infty}}^{-})$ denotes the entropy contribution for the Lyapunov subgroup~$\tilde{H}_{\chi_{\infty}}^{-}$. Using Theorem~7.9 in \cite{Pisa}, it follows that~$\mu$ is~$\tilde{H}_{\chi_{\infty}}^{-}$-invariant. By a similar argument one obtains that~$\mu$ is~$\tilde{H}_{-\chi_{p}}^{-}$-invariant for any~$p\in\Sf$. In particular,~$\mu$ is~$\tildeHS^{(+)}$-invariant.

  Let~$\Bcal_{\TS}$ be the Borel-$\sigma$-algebra on~$\TS$ and let~$\Ccal=\pi_{\TS}^{-1}\Bcal_{\TS}$ be the~$\sigma$-algebra defined by the fibers of the canonical projection to~$\TS$. Consider the disintegration
  \begin{equation}\label{eq:disintegration}
    \mu=\int\mu_{(t,x)}^{\Ccal}\der\mu(t,x)
  \end{equation}
  of~$\mu$ with respect to~$\Ccal$. As the atoms of~$\Ccal$ are~$\tildeHS^{(+)}$-invariant and as on a set of full measure the conditional measure~$\mu_{(t,x)}^{\Ccal}$ is determined by the atom~$[(t,x)]_{\Ccal}$, almost all measures appearing in the disintegration are~$\tildeHS^{(+)}$-invariant. We refer to Chapter~5 in \cite{vol1} for details on conditional measures. As~$\nu_{S}$ is ergodic for the action of~$\tildeHS^{(+)}$, extremality of ergodic measures among invariant measures implies~$(\pi_{\XS})_{\ast}\mu_{(t,x)}^{\Ccal}=\nu_{S}$ for almost all~$(t,x)$. As~$\mu_{(t,x)}^{\Ccal}$ is concentrated on the atom~$[(t,x)]_{\Ccal}=\{t\}\times\XS$, we get~$\mu_{(t,x)}^{\Ccal}=\delta_{t}\otimes\nu_{S}$ for almost all~$(t,x)\in\TS\times\XS$ where~$\delta_{t}$ denotes the Dirac measure at~$t$. Hence by means of Fubini's Theorem and the assumption that~$(\pi_{\TS})_{\ast}\mu=m_{\TS}$ it follows that~$\mu=m_{\TS}\otimes\nu_{S}$.
\end{proof}
We show how to deduce an ineffective~$S$-arithmetic equidistribution result in the spirit of Theorem~\ref{thm:mainthm} for sequences of values~$(a_{n},c_{n})$.
\begin{cor}\label{cor:ineffectivetorustimessurface}
  Let~$a_{n},c_{n}\in\Z$ be a sequence satisfying~$(n,a_{n}c_{n})=1$. Let~$\mu_{n,S;a_{n},c_{n}}^{\times d}$ denote the normalized counting measure on the set
  \begin{equation*}
    \Pcal_{S}^{\times d}(n;a_{n},c_{n})=\big\{\big(\ZS+\Delta(\tfrac{a_{n}k^{d}}{n}),\GammaS\Delta(u_{c_{n}k^{d}/n})a_{\sqrt{n}}^{-1}\big):(k,n)=1\big\}\subseteq\TS\times\XS.
  \end{equation*}
  As~$n\to\infty$ in~$\D(\prod_{p\in\Sf}p)$, the measures~$\mu_{n,S;a_{n},c_{n}}^{\times d}$ equidistribute towards the~$\QS\times\GS$-invariant probability measure on~$\TS\times\XS$.
\end{cor}
\begin{proof}
  Denote by~$\pi_{\XS}:\TS\times\XS\to\XS$ and~$\pi_{\TS}:\TS\times\XS\to\TS$ the canonical projections. Corollary~\ref{cor:equidistributionfactors} implies that any \weakstar limit of the sequence~$\mu_{n,S;a_{n},c_{n}}^{\times d}$ projects to the unique~$\GS$-invariant probability measure~$\nu_{S}$ on~$\XS$ under~$\pi_{\XS}$ and similarly to the unique~$\QS$-invariant probability measure~$m_{\TS}$ on~$\TS$ under~$\pi_{\TS}$. For every~$p\in\Sf$ the measure~$\mu_{n,S;a_{n},c_{n}}^{\times d}$ is invariant under simultaneous multiplication by~$p^{2d}$ in the first component and right multiplication by~$a_{p^{d}}^{-1}$ in the second component. Hence any \weakstar limit is a joining for the actions of~$\Z^{\Sf}$ on~$(\TS,m_{\TS})$ and on~$(\XS,\nu_{S})$ defined respectively by~$m.t=S^{2dm}t$ and~$m.x=xa_{S^{dm}}^{-1}$ for~$m\in\Z^{\Sf}$,~$t\in\TS$ and~$x\in\XS$.

  Let now~$\mu$ be a \weakstar limit of a sequence of measures~$\mu_{n,S;a_{n},b_{n}}^{\times d}$ and let
  \begin{equation*}
    \mu=\int_{\TS\times\XS}\mu_{(t,x)}^{\Ecal}\der\mu(t,x)
  \end{equation*}
  be an ergodic decomposition of~$\mu$. Then almost every~$\mu_{(t,x)}^{\Ecal}$ is an ergodic joining for the~$\Z^{\Sf}$ actions on~$(\TS,m_{\TS})$ and~$(\XS,\nu_{S})$ respectively. It thus suffices to show that the only ergodic joining is the trivial joining~$m_{\TS}\otimes\nu_{S}$. This was done in Proposition~\ref{prop:disjointnesstwofactors}.
\end{proof}
The second disjointness result forms the heart of this section and is the main input for the completion of the proof of Proposition~\ref{prop:sarithmetictwotorus}.
\begin{prop}[Disjointness for three factors]\label{prop:disjointnessthreefactors}
  Fix~$d\in\N$ and let~$\nu$ be a probability measure on~$\TS\times\TS\times\XS$ which is~$\Z^{\Sf}$-invariant and ergodic for the action defined by
  \begin{equation*}
    m.(t,s,x)=(S^{2dm}t,S^{-2dm}s,xa_{S^{dm}}^{-1})\qquad(t,s\in\TS,x\in\XS).
  \end{equation*}
  Assume that~$\nu$ projects to~$m_{\TS}$ in the first two factors and to~$\nu_{S}$ in the last factor. Then~$\nu$ equals the product measure, i.e.~$\nu=m_{\TS}\otimes m_{\TS}\otimes\nu_{S}$.
\end{prop}
\begin{proof}
  In what follows, we denote by~$m_{1}$ the unique~$\US\times\GS$-invariant probability measure on~$\rquotient{\UGamma\times\GGamma}{\US\times\GS}$ and by~$\pi_{1}$ the projection sending~$(t,s,x)\in\rquotient{\HGamma}{\HS}$ to~$(t,x)\in\TS\times\XS$. Similarly, we let~$m_{2}$ denote the unique~$\VS\times\GS$-invariant probability measure on~$\rquotient{\VGamma\times\GGamma}{\VS\times\GS}$ and~$\pi_{2}$ the projection sending~$(t,s,x)\in\rquotient{\HGamma}{\HS}$ to~$(s,x)\in\TS\times\XS$. As of Proposition \ref{prop:disjointnesstwofactors}, we know that~$(\pi_{i})_{\ast}\nu=m_{i}$ for both~$i=1,2$. Thus by Corollary~6.5 from \cite{definitivejoining}, we know that for any~$p\in\Sf$ we have
  \begin{equation*}
    h_{\nu}(\alpha_{p},H_{\chi_{p}}^{-})=h_{m_{1}}(\alpha_{p},H_{\chi_{p}}^{-})+h_{\nu}(\alpha_{p},\underbrace{H_{\chi_{p}}^{-}\cap\VS}_{=\{1\}})=h_{m_{1}}(\alpha_{p},H_{\chi_{p}}^{-}).
  \end{equation*}
  Again, Theorem 7.9 in \cite{Pisa} implies that~$\nu$ is~$H_{\chi_{p}}^{-}$-invariant. The same argument shows
  \begin{equation*}
    h_{\nu}(\alpha_{p},H_{-\chi_{p}}^{-})=h_{m_{2}}(\alpha_{p},H_{-\chi_{p}}^{-})+h_{\nu}(\alpha_{p},\underbrace{H_{-\chi_{p}}^{-}\cap\US}_{=\{1\}})=h_{m_{2}}(\alpha_{p},H_{-\chi_{p}}^{-}),
  \end{equation*}
  and thus~$\nu$ is~$\Up\times\Vp\times\SLtwo(\Qp)$-invariant. Applying the same argument for the Lyapunov subgroups~$H_{\chi_{\infty}}^{-}$ and~$H_{-\chi_{\infty}}^{-}$, one obtains invariance of~$\nu$ under~$\Ur\times\Vr\times\SLtwo(\R)$. Hence~$\nu$ is~$\HS$-invariant.
\end{proof}
\subsection{Completion of the proof of Theorem {\ref{thm:realtwotorus}}}
Using the disjointness results form Section~\ref{sec:disjointness}, we can prove Proposition~\ref{prop:sarithmetictwotorus}. 
\begin{proof}[Proof of Proposition {\ref{prop:sarithmetictwotorus}}]
  Let~$\overline{\nu}$ be a \weakstar limit of a sequence of measures~$\measureprimitiverationalsnanbncntwotorus$ where~$n\in\D(\prod_{p\in\Sf}p)$. Then~$\overline{\nu}$ is invariant under the action of~$\Z^{\Sf}$ given by
  \begin{equation*}
    m.(t,s,x)=T_{dm}(t,s,x)\qquad(t,s\in\TS,x\in\XS)
  \end{equation*}
  where~$T_{dm}$ was defined in~\eqref{eq:theaction}. Using our assumptions on~$a_{n},b_{n},c_{n}$ together with Corollary~\ref{cor:equidistributionfactors}, the measure~$\overline{\nu}$ projects to~$m_{\TS}$ in the first two components respectively and to~$\nu_{S}$ in the third component. In particular, every ergodic component of~$\overline{\nu}$ with respect to the~$\Z^{\Sf}$-action satisfies the assumptions of Proposition~\ref{prop:disjointnessthreefactors} and thus equals the product measure as of Proposition~\ref{prop:disjointnessthreefactors}. This completes the proof.
\end{proof}
%%%%%%%%%%%%% Bibliography %%%%%%%%%%%%%%
\def\cprime{$'$} \providecommand{\bysame}{\leavevmode\hbox
  to3em{\hrulefill}\thinspace}
\providecommand{\MR}{\relax\ifhmode\unskip\space\fi MR }
\providecommand{\MRhref}[2]{%
  \href{http://www.ams.org/mathscinet-getitem?mr=#1}{#2} }
\providecommand{\href}[2]{#2}

\bibliographystyle{alphabetic}

\end{document}